\documentclass[a4paper, reqno]{amsart}

\usepackage{a4wide}
\usepackage{graphicx,xcolor}
\usepackage[hidelinks]{hyperref}

\usepackage[shortlabels]{enumitem}
\setlist[enumerate,1]{label={\arabic*. }}

\usepackage{amsmath,amssymb}
\usepackage{dsfont} 
\usepackage{mathrsfs} 
\usepackage{esint} 
\usepackage{tikz}
\usepackage{overpic}

\usepackage[normalem]{ulem}

\newcommand{\dd}[1]{\mathrm{d}{#1}} 

\newcommand{\dist}{\mathrm{dist}}
\newcommand{\N}{\mathbb{N}}
\newcommand{\Z}{\mathbb{Z}}
\newcommand{\R}{\mathbb{R}}
\newcommand{\E}{\mathbb{E}}
\renewcommand{\P}{\mathbb{P}}
\newcommand{\Q}{\mathbb{Q}}
\newcommand{\eps}{\varepsilon}
\newcommand{\Rn}{\R^n}

\newcommand{\locdefect}[4]{\int_{#1} \frac{|#2(x) - #2(y)|^p}{|#3-#4|^{n+sp}} \, \dd{#3} \, \dd{#4}}
\newcommand{\wspseminorm}[2]{|#1|^p_{W^{s,p}(#2)}}
\newcommand{\wsp}[1]{W^{s,p}(#1)}
\def\Wsp{W^{s,p}}
\def\csp{\mathrm{cap}_{s,p}}
\newcommand{\Dsp}{\mathcal{D}_{s,p}}
\newcommand{\cspqe}{\csp\text{-q.e.}}
\newcommand{\cspqeon}{\cspqe\text{ on }}
\def\CSP{C_{s,p}}
\def\cspK{\mathrm{cap}_{\mathscr{K}}}
\def\CSPK{C_{\mathscr{K}}}
\def\Ksp{\dot{W}^{s,p}}
\newcommand{\DspK}{\mathcal{D}_{\mathcal{K}}}

\DeclareMathOperator\supp{supp}
 
\DeclareMathOperator{\lip}{Lip}

\providecommand{\norm}[1]{\lVert#1\rVert}

\let\implies\Rightarrow

\newtheorem{thm}{Theorem}[section]
\newtheorem{prop}[thm]{Proposition}
\newtheorem{lmm}[thm]{Lemma}

\theoremstyle{definition}

\newtheorem{claim}{Step}

\theoremstyle{remark}
\newtheorem{rmk}[thm]{Remark}

\numberwithin{equation}{section}

\usepackage[dvipsnames]{xcolor}
\definecolor{Mybgcolor}{HTML}{2E3440}
\definecolor{Mytextcolor}{HTML}{ECEFF4}
\definecolor{Mycommentcolor}{HTML}{88C0D0}
\definecolor{Mycommentcolor2}{HTML}{EBCB8B}
\definecolor{FDcolor}{HTML}{88C0D0}
\definecolor{MFcolor}{HTML}{0C63F0}

\renewcommand{\L}{\mathcal{L}}
\newcommand{\floor}[1]{\lfloor #1 \rfloor}

\newcommand{\ie}{\textit{i.e.}, }

\definecolor{MatplotlibC0}{HTML}{1f77b4}

\newcommand{\obstacle}{T}
\newcommand{\domain}{U}
\newcommand{\functional}{\mathcal{F}}
\newcommand{\functionalK}{\mathcal{K}}

\newcommand{\ballCriticalRadiusJI}{B_{\lambda_j \rho_i}(\eps_j x_i)}
\newcommand{\ballCriticalRadiusEpsIomega}{B_{\lambda_\eps \rho_i(\omega)}(\eps x_i(\omega))}

\newcommand{\ballDoubleCriticalRadiusEpsIomega}{B_{2\lambda_\eps \rho_i(\omega)}(\eps x_i(\omega))}
\newcommand{\ballDoubleCriticalRadiusEpsJ}{B_{2 \lambda_\eps \rho_j}(\eps x_j)}
\newcommand{\ballDoubleCriticalRadiusEpsJomega}{B_{2\lambda_\eps \rho_j(\omega)}(\eps x_j(\omega))}

\newcommand{\ballDoubleCriticalRadiusJI}{B_{2\lambda_j \rho_i}(\eps_j x_i)}
\newcommand{\ballContainingObstacleEpsIomega}{B_{\frac{\eps}{\truncParam}}(\eps x_i(\omega))}
\newcommand{\ballContainingObstacleEpsI}{B_{\frac{\eps}{\truncParam}}(\eps x_i)}
\newcommand{\diagonalDelta}{\Delta_\delta}
\newcommand{\borel}{\mathcal{B}} 
\newcommand{\IF}{\mathds{1}} 
\newcommand{\sampleSpace}{\Omega}
\newcommand{\sigmaAlgebra}{\mathscr{F}}
\newcommand{\prob}{\P}
\newcommand{\probSpace}{(\sampleSpace,\sigmaAlgebra,\mathbb{P})}
\newcommand{\capacityConstantErgodic}{\gamma}
\newcommand{\PP}{N} 
\newcommand{\PPg}{N_g} 
\newcommand{\probLawPP}{\mathcal{P}}
\newcommand{\EM}{M} 
\newcommand{\EMg}{M_g} 
\newcommand{\IG}{m_g} 
\newcommand{\MD}{\pi} 
\newcommand{\spaceOfMeasures}{\mathcal{N}}
\newcommand{\spaceOfMeasuresWithSimpleGround}{\widehat{\spaceOfMeasures}}
\newcommand{\eventSimpleGround}{\widehat{\sampleSpace}}

\newcommand{\truncParam}{R} 
\newcommand{\joiningLemmaEnergyDiffParam}{L} 
\newcommand{\annulusJI}{C_j^{i,h_j}}
\newcommand{\annulusJIomega}{C_j^{i,h_j^\omega}}
\newcommand{\annulusJK}{C_j^{k,h_j}}

\newcommand{\ballBoundingAnnulusJI}{\tilde{B}_j^{i,h_j}}
\newcommand{\ballBoundingAnnulusJK}{\tilde{B}_j^{k,h_j}}
\newcommand{\ballBoundingAnnulusJIomega}{\tilde{B}_j^{i,h_j^\omega}}
\newcommand{\ballContainingAnnulusJI}{B_j^{i,h_j}}
\newcommand{\ballContainingAnnulusJIomega}{B_j^{i,h_j^\omega}}
\newcommand{\limTruncParamJoiningLemmaParamJtoInfty}{\lim_{\truncParam\to\infty}\lim_{\joiningLemmaEnergyDiffParam\to\infty}\lim_{j\to\infty}}
\begin{document}

\newcommand{\goodEps}{G_{\eps,\truncParam}}
\newcommand{\goodEpsOmega}{\goodEps^\omega}
\newcommand{\veryGoodEps}{V\goodEps}
\newcommand{\veryGoodEpsOmega}{\veryGoodEps^\omega}
\newcommand{\notVeryGoodEps}{N\veryGoodEps}
\newcommand{\notVeryGoodEpsOmega}{\notVeryGoodEps^\omega}
\newcommand{\insideDomainEps}{I_\eps}
\newcommand{\insideDomainEpsOmega}{\insideDomainEps^\omega}
\newcommand{\thinnedInsideDomainEpsOmega}{\insideDomainEps^{2/\truncParam, \omega}}
\newcommand{\thinnedInsideDomainEps}{\insideDomainEps^{2/\truncParam}}
\newcommand{\obstacleEpsI}{\setOfObstaclesEps^i}
\newcommand{\obstacleEpsIOmega}{\setOfObstaclesEps^{i,\omega}}
\newcommand{\setOfObstaclesEps}{\obstacle_\eps}
\newcommand{\setOfObstaclesEpsOmega}{\setOfObstaclesEps^\omega}
\newcommand{\goodJ}{G_{j,\truncParam}}
\newcommand{\goodJOmega}{\goodJ^\omega}
\newcommand{\veryGoodJ}{V\goodJ}
\newcommand{\veryGoodJOmega}{\veryGoodJ^\omega}
\newcommand{\notVeryGoodJ}{N\veryGoodJ}
\newcommand{\notVeryGoodJOmega}{\notVeryGoodJ^\omega}
\newcommand{\insideDomainJ}{I_j}
\newcommand{\insideDomainJOmega}{\insideDomainJ^\omega}
\newcommand{\thinnedInsideDomainJOmega}{\insideDomainJ^{2/\truncParam, \omega}}
\newcommand{\thinnedInsideDomainJ}{\insideDomainJ^{2/\truncParam}}
\newcommand{\obstacleJI}{\setOfObstaclesJ^i}
\newcommand{\obstacleJIOmega}{\setOfObstaclesJ^{i,\omega}}
\newcommand{\setOfObstaclesJ}{\obstacle_j}
\newcommand{\setOfObstaclesJOmega}{\setOfObstaclesJ^\omega}

\newcommand{\errordRL}{\vartheta}

\newcommand{\intersectionOfEventsGivenBy}{Let $\sampleSpace' \in \sigmaAlgebra$ be the intersection of the events of probability one given by }
\newcommand{\toAvoidCumbersomeNotationOmegaOmitted}{To avoid cumbersome notation, $\omega$ will be omitted in the computations below. }
\newcommand{\subsequenceForWhichLimitsExist}{possibly taking a subsequence for which the limits below exist, }
\newcommand{\sureEvent}{There exists $\sampleSpace' \in \sigmaAlgebra$ with $\prob(\sampleSpace') = 1$ such that, for every $\omega \in \sampleSpace'$, }
\newcommand{\sureEventWithFollowingProperty}{There exists $\sampleSpace' \in \sigmaAlgebra$ such that $\prob(\sampleSpace') = 1$ with the following property. }

\title[Stochastic homogenization of fractional obstacle problems]
{Stochastic homogenization of fractional obstacle problems}
\author[F. Deangelis]{Francesco Deangelis}
\address[Francesco Deangelis]{Applied Mathematics M\"unster, University of M\"unster\\
	Einsteinstrasse 62, 48149 M\"unster, Germany}
\email{francesco.deangelis@uni-muenster.de}

\author[M. Focardi]{Matteo Focardi} 
\address[Matteo Focardi]{DiMaI U.\ Dini, Universit\`a di Firenze, V.le G.B. Morgagni 67/A, 50134 Firenze, Italy}
\email[Matteo Focardi]{matteo.focardi@unifi.it}

\author[C. I. Zeppieri]{Caterina Ida Zeppieri}
\address[Caterina Ida Zeppieri]{Applied Mathematics M\"unster, University of M\"unster\\
	Einsteinstrasse 62, 48149 M\"unster, Germany}
\email{caterina.zeppieri@uni-muenster.de}

\begin{abstract}
We prove a stochastic homogenization result for a class of \emph{nonlinear} and \emph{nonlocal} variational problems in domains with many small randomly distributed (bilateral) obstacles. Our model case is a Dirichlet problem for the \emph{fractional} $p$-Laplacian, $p>1$, where a pinning condition $u=0$ is imposed on the solution in a \emph{random} collection of small balls whose centers and radii are generated by a \emph{stationary marked point process}. Such a general obstacle distribution allows for \emph{clustering effects} to appear with positive probability. Under suitable moment conditions on the obstacle radii, we identify a critical scaling regime in which the fractional $p$-capacity density of the obstacles is asymptotically additive \emph{almost surely}.  In turn, this key property allows us to derive an effective homogenized problem which is formally analogous to the one obtained in the periodic setting or under the assumption of well-separation for the obstacles. 
The analysis also extends to the case of \emph{randomly shaped obstacles} and to a broad class of \emph{nonlocal interaction kernels}. 
At the methodological level, the paper develops a streamlined proof strategy with several new ingredients, among them the use of Palm measures.
\end{abstract}

\maketitle

{\small
	\noindent\keywords{\textbf{Keywords:} Stochastic homogenization, $\Gamma$-convergence, obstacle problems, nonlocal functionals, fractional capacity, stationary processes, marked point processes, Palm measures.}
	
	\smallskip
	
	\noindent\subjclass{\textbf{MSC 2020:} 49J45, 49J55, 35R11,  60G10, 60G55, 60G57, 74Q15.}
}	

\tableofcontents

\section{Introduction}

\subsection{Overview} In this paper we study the effective behavior of sequences of \emph{nonlinear} and \emph{nonlocal} Dirichlet problems posed in domains containing a diverging number of increasingly small \emph{randomly distributed} (bilateral) obstacles. 

To set up the framework, let $\Omega$ be the sample space of an underlying probability space and let $\eps>0$ be a small parameter. The prototypical problem we consider is 
\begin{equation}\label{intro:P_eps}
\begin{cases}
(-\Delta)^s_p u = f  & \text{in}\; U\setminus T^\omega_\eps, 
\cr 
u=0 & \text{on}\; \partial U \cup T^\omega_\eps,  
\end{cases}
\end{equation}
where $U\subset \R^n$ is open, bounded and Lipschitz,  $(-\Delta)^s_p$ denotes the (regional) fractional $p$-Laplacian of order $s$, with $p\in (1,\infty)$ and $s\in (0,1)$. In \eqref{intro:P_eps} the source term $f$ belongs to $L^{p'}(U)$ and, for $\omega\in \Omega$, $T_\eps^\omega$ denotes a realization of the \emph{random} obstacle set defined as 
\begin{equation}\label{intro:T_eps}
T_\eps=\bigcup \{B_{\lambda_\eps \rho} (\eps x) \colon (x,\rho) \in \supp N, \, x\in \eps^{-1}U\}.
\end{equation}
In \eqref{intro:T_eps} $\lambda_\eps>0$ is a secondary small scale satisfying $\lambda_\eps\ll \eps$, and $N$ is a \emph{marked point process}. More precisely, $N$ is a random counting measure of the form  
\begin{equation*}
    \PP = \sum_{i = 1}^{\infty} \delta_{(x_i,\rho_i)},
\end{equation*}
where $\{x_i\}_{i \in \N}$ is a sequence of $\R^n$-valued random variables and $\{\rho_i\}_{i \in \N}$ is a sequence of $\R_+$-valued random variables, representing the marks associated to the points $\{x_i\}_{i \in \N}$. 
Under \emph{minimal} structural assumptions on $N$, which in particular shall be \emph{stationary} and \emph{ergodic} with respect to shifts in $\R^n$ (see \eqref{eq: prob law invariant under shifts} and \eqref{eq: prob law ergodic under shifts} for the precise definitions), in this paper we prove that, when $sp\in (0,n)$, there is a critical scaling $\lambda_\eps$ for the obstacle radii such that, as $\eps\to 0$, the solutions to \eqref{intro:P_eps} converge \emph{almost surely} to the solution of the effective problem
\begin{equation}\label{intro:P_0}
\begin{cases}
(-\Delta)^s_p u + \gamma\, |u|^{p-2}u = f  & \text{in}\; U, 
\cr 
u=0 & \text{on}\; \partial U,  
\end{cases}
\end{equation}
where $\gamma>0$ is explicit and encodes the asymptotic geometry and distribution of the obstacles. More precisely, $\gamma$ represents the averaged limit density of the fractional $p$-capacity of the random obstacle set $T_\eps$, as we are going to explain in detail below. 
\subsection{A brief literature review}
The asymptotic behavior of solutions to local Dirichlet problems in periodically perforated domains (or in domains with obstacles) has been a very active research area since the seminal works of Marchenko and Khruslov \cite{marchenko-khruslov}, Rauch and Taylor \cite{rauch-taylor-1,rauch-taylor-2}, and Cioranescu and Murat \cite{cioranescu-murat-2,cioranescu-murat-1}.
From a mathematical point of view, these problems are interesting because, when the perforations are at a critical size, the limit equation shows an extra zero-order term, originally dubbed ``a strange term coming from nowhere'', which reflects the additive asymptotic behavior of the capacity of the homogenizing obstacles.

Several approaches have been developed over the years to study local problems posed in perforated domains; comprehensive treatments can be found in the monographs \cite{attouch,dal-maso-gamma-convergence,cioranescu-donato,braides-beginners,marchenko-khruslov-2,cioranescu-damlamian-griso}.  In particular, Dal Maso gave a complete variational solution \cite{dal-maso-obstacle-1,dal-maso-obstacle-2} by building on the $\Gamma$‑convergence framework previously introduced by De Giorgi, Dal Maso, and Longo \cite{de-giorgi-dal-maso-longo}.
In the 2000s the interest in the homogenization of obstacle problems resurfaced in the nonlocal community thanks to Caffarelli and Mellet \cite{caffarelli-mellet}, who were the first to study these problems for the fractional Laplacian. Their method relies on the celebrated Caffarelli-Silvestre extension \cite{caffarelli-silvestre}, which allows to reduce the problem posed for the fractional Laplacian to a local problem. Indeed, thanks to \cite{caffarelli-silvestre} the solution to the nonlocal problem in $\mathbb R^n$ can be interpreted as the boundary trace of the solution of a \emph{degenerate} but local elliptic equation in the higher-dimensional positive half-space $\mathbb{R}^{n+1}_{+}$.
Caffarelli and Mellet's proof is then based on the construction of oscillating test functions for degenerate elliptic PDEs and their homogenization result holds for ``well separated'' randomly shaped obstacles placed on a periodic lattice.

A genuinely nonlocal approach was later proposed by the second author by resorting to $\Gamma$-convergence techniques  combined with an ergodic theorem for multiparameter stationary processes \cite{focardi-aperiodic-fractional,focardi-vector-obstacle}. Circumventing the construction of the oscillating test functions, this variational approach applies to a broad class of nonlocal and nonlinear problems and easily extends beyond the periodic setting (see also \cite{focardi-fractional}). Moreover, an explicit expression for the limiting capacity density of the obstacles is obtained in \cite{focardi-aperiodic-fractional,focardi-vector-obstacle}, whereas it remained implicit in the extension-based works \cite{caffarelli-mellet,caffarelli-mellet-local}. We stress, however, that the analysis in \cite{focardi-aperiodic-fractional,focardi-vector-obstacle} strongly relies on the spatial distribution of the obstacles, in particular ruling out clustering effects. More precisely, in the works mentioned above the obstacles are required to be \emph{well separated}, a condition that enables local constructions as in the periodic setting.    

We also mention here the recent work \cite{Alicandro_et_al} where the authors study nonlocal variational problems of convolution type in periodically perforated domains, via $\Gamma$-convergence.  

On the other hand, in the local setting, Giunti, Höfer, and Velázquez \cite{giunti-hoefer-velazquez} significantly extended homogenization results for the Poisson equation in perforated domains by considering perforations generated by a stationary and ergodic marked point process as in \eqref{intro:T_eps}, under the assumption that both the centers and the radii exhibit short-range correlations. This framework allows the centers $\{x_i\}_{i \in \mathbb N}$ to be arbitrarily close and the radii $\{\rho_i\}_{i \in \mathbb N}$ to be arbitrarily large, so that clustering of perforations can occur with probability one (cf.\ \cite{grimmett,meester}).
Assuming finiteness of the average limit capacity density of the perforations (which, in view of the scaling properties of the elliptic capacity, amounts to a suitable moment condition on the radii) the authors prove that, for $\lambda_\eps = \eps^{n/(n-2)}$, a homogenization result analogous to those in \cite{cioranescu-murat-2,cioranescu-murat-1} holds. Their proof is based on the construction of oscillating test functions in the spirit of Cioranescu and Murat, with the main difficulty being to ensure that the presence of clusters does not break down this construction. The key observation in \cite{giunti-hoefer-velazquez} is that the imposed moment condition and the decorrelation assumption over large distances (formulated as a quantitative strong mixing condition) guarantee that, almost surely, clustering holes have asymptotically negligible capacity.

More recently, a nonlinear counterpart of this result was obtained under similar assumptions by Scardia, Zemas, and Zeppieri \cite{scardia-zemas-zeppieri} via a purely variational approach, thus providing the first homogenization result of obstacle problems for nonlinear elliptic PDEs (with variational structure) without the assumption of well-separation. 

\subsection{Main contributions of the paper and method of proof}  
In the present paper, we extend the homogenization results of \cite{giunti-hoefer-velazquez, scardia-zemas-zeppieri} to the nonlocal setting. More precisely, we establish a homogenization result for a broad class of nonlocal and nonlinear obstacle problems under minimal assumptions on the geometry and size of the obstacles. Our result also provides a substantial extension, in the nonlocal framework, of the homogenization results obtained in \cite{focardi-fractional, focardi-aperiodic-fractional}. Indeed, our approach does not rely on any separation assumption, thus accommodating genuinely irregular configurations of obstacles. Moreover, the analysis applies to a class of interaction kernels which is significantly larger than the one considered in \cite{focardi-aperiodic-fractional}.

In contrast to \cite{giunti-hoefer-velazquez, scardia-zemas-zeppieri}, we remove the short-range correlation assumption on the underlying marked point process, which is now only required to be stationary and ergodic (see also \cite{Bastug}) and to satisfy a suitable moment condition on the obstacle radii. Moreover, we cover also the case of randomly shaped obstacles (see also \cite{Sato}).
In the same spirit as \cite{focardi-aperiodic-fractional, scardia-zemas-zeppieri}, our approach is purely variational and avoids the delicate PDE techniques and regularity results employed in \cite{caffarelli-mellet}, which are moreover specific to the fractional Laplacian. This yields a robust and unified framework for the stochastic homogenization of a broad class of nonlocal and nonlinear obstacle problems.
In addition, the paper develops a streamlined proof strategy with several methodological simplifications and new ingredients with respect to \cite{giunti-hoefer-velazquez,scardia-zemas-zeppieri} (see the comments hereafter).

We now provide an overview of the main contribution of the paper and outline its proof strategy. For the sake of presentation, we focus on the model case of problem \eqref{intro:P_eps}, where the random obstacle set $T_\eps$ consists of a union of balls with random centers and radii, as defined in \eqref{intro:T_eps} (see Figure~\ref{fig: obstacle set}).
\begin{figure}
    \begin{overpic}[width=0.4\textwidth]{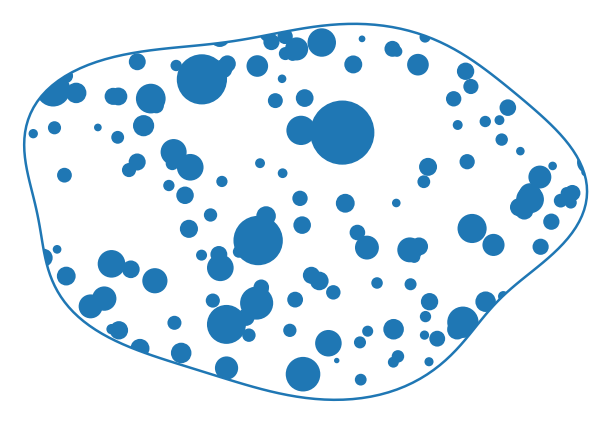}
        \put (80,10) {\color{MatplotlibC0}$\domain$}
    \end{overpic}
    \caption{In blue a realization of the random set $T_\eps$. (In the picture the centers of the balls are generated by a Poisson point process, while the radii are i.i.d. random variables with a log-normal distribution.)}
    \label{fig: obstacle set}
\end{figure}
To keep the presentation intuitive, we disregard technicalities and assume only that the marked point process $N$ is stationary and ergodic. We refer the reader to Section \ref{ss:main result}
for the complete list of assumptions, as well as for the treatment of the case in which ergodicity is relaxed
(cf. Remark \ref{thm: main result, nonergodic}). 

We observe that problem \eqref{intro:P_eps} has a variational structure. Indeed, for fixed $\eps$, the unique solution to \eqref{intro:P_eps} coincides with the minimizer in the fractional Sobolev space $\Wsp_0(\domain)$ of the functional   
\begin{equation}\label{intro:funct_compl} 
    \displaystyle 
    \int_{\domain\times \domain}\frac{|u(x)-u(y)|^p}{|x-y|^{n+sp}}\, \dd{x} \, \dd{y} -p\int_U fu \,\dd{x} +\chi_\eps^\omega(u),
\end{equation}
where, for every $\omega \in \Omega$, 
\[
\chi_\eps^\omega(u)=
\begin{cases}
0 & \text{if }\; u=0\,\, \text{ on } \setOfObstaclesEpsOmega \cap \domain,
\cr
+\infty & \text{otherwise.}
\end{cases}
\]
By the fundamental property of $\Gamma$-convergence, the convergence of the functionals in \eqref{intro:funct_compl} yields the convergence of the corresponding minimizers and hence of the solutions to \eqref{intro:P_eps}. Indeed, we note that in this setting the compactness of minimizing sequences directly follows from the Poincar\'e Inequality in $\Wsp_0(\domain)$. Moreover, since the linear term in \eqref{intro:funct_compl} can be treated as a continuous perturbation, it actually suffices to focus on the asymptotic behavior of    
\begin{equation}\label{intro:functional} 
    \functional_\eps^\omega(u) = \begin{cases} 
\displaystyle\int_{\domain\times \domain}\frac{|u(x)-u(y)|^p}{|x-y|^{n+sp}}\, \dd{x} \, \dd{y} & \text{ if } u\in \Wsp(\domain),\,
        {u}=0\,\, \text{ on } \setOfObstaclesEpsOmega \cap \domain,
        \cr
        +\infty & \text{otherwise in }\, L^p(U).
        \end{cases}
        \end{equation}
In the main result of the paper, Theorem~\ref{thm: main result}, we prove that for $sp \in (0,n)$, $\lambda_\eps=\eps^{n/(n-sp)}$, and under the moment condition 
 \begin{equation}\label{intro:moment}
    \E\bigg[\sum_{x_i\in Q} \rho_i^{n-sp}\bigg]<+\infty, 
    \end{equation}
where $Q$ is a unit cube in $\mathbb R^n$, the random functionals $\functional_\eps$ $\Gamma$-converge, \emph{almost surely}, to the deterministic functional given by
\begin{equation*}
        \functional(u) =\begin{cases} 
            \displaystyle\int_{\domain\times \domain}\frac{|u(x)-u(y)|^p}{|x-y|^{n+sp}}\, \dd{x} \, \dd{y} +\capacityConstantErgodic\int_{\domain}|u|^p\dd{x} & \text{if }\; u\in \Wsp(\domain),
            \cr
            +\infty & \text{otherwise in }\; L^p(\domain),
            \end{cases}
            \end{equation*}
where 
\begin{equation}\label{intro:gamma}
    \capacityConstantErgodic = \csp(B_1)\, \E\bigg[\sum_{x_i\in Q} \rho_i^{n-sp}\bigg].
\end{equation}
In \eqref{intro:gamma} $\csp(B_1)$ denotes the fractional $p$-capacity, or in short the $(s,p)$-capacity, of the unit ball (see Section \ref{ss:sobslob} for its definition). We observe that $\capacityConstantErgodic$ factorizes into the contribution of the local geometry (via the $(s,p)$-capacity of the fundamental obstacle shape) and that of the global distribution of the obstacles (via the $(n-sp)$-moment of the radii). A similar formula is obtained in 
\cite{giunti-hoefer-velazquez,scardia-zemas-zeppieri}. Moreover, we refer the reader to Remarks~\ref{rmk:On the assumptions}, \ref{r:Wald}, and \ref{thm: main result, nonergodic} for different characterizations of the capacitary constant $\capacityConstantErgodic$.
While the necessity of the moment condition \eqref{intro:moment} is immediate from \eqref{intro:gamma}, the main contribution of the present work is to show that \eqref{intro:moment} is also sufficient to establish a homogenization result analogous to that of \cite{focardi-aperiodic-fractional} in this more general setting.

In the same spirit as \cite{giunti-hoefer-velazquez}, a crucial step in the proof of the homogenization theorem is a decomposition result of the random obstacle set $T_\eps$. Roughly speaking, we decompose $T_\eps$ into the disjoint union of two sets $T_\eps^{g}$ and $T_\eps^{b}$, where $T_\eps^{g}$ consists of ``small'' and $\eps$-separated obstacles (the so-called \emph{good} obstacles), while  $T_\eps^{b}$ contains the remaining ones (the so-called \emph{bad} obstacles), in particular clusters (see \eqref{eq: definition of good indices}-\eqref{eq: definition of not very good indices} for the precise definitions). By construction, the obstacles in $T_\eps^{g}$ are also well separated from a suitable enlargement of the bad set $T_\eps^{b}$ (see \eqref{eq: very good enlarged obstacles and not very good obstacles with double radius are disjoint}). 
We then show that, almost surely and for $\eps$ sufficiently small, 
\[
\csp(T_\eps \cap Q) \simeq \csp(T_\eps^g \cap Q) + \csp(T_\eps^b \cap Q)
\]
and $\csp(T_\eps^b \cap Q) \to 0$ as $\eps \to 0$. Therefore, thanks to the $\eps$-separation of the obstacles in $T_\eps^g$, we obtain
\[
\lim_{\eps\to 0}\csp(T_\eps \cap Q) = \lim_{\eps\to 0}\csp(T_\eps^g \cap Q)=\csp(B_1)\, \lim_{\eps \to 0}\eps^n \sum_{\eps x_i\in Q} \rho_i^{n-sp}. 
\]
Above we have also used the choice $\lambda_\eps = \eps^{n/(n-sp)}$, together with the scaling property of the $(s,p)$-capacity, namely $\csp(B_{\rho})=\csp(B_1) \rho^{n-sp}$. Finally, exploiting the stationarity and ergodicity of $N$ together with \eqref{intro:moment}, invoking the ergodic theorem for marked point processes, we deduce that, almost surely,
\[
\lim_{\eps\to 0}\csp(T_\eps \cap Q) =\gamma.
\]
The above heuristic argument shows that the moment condition \eqref{intro:moment} is in fact sufficient to ensure that, even in the presence of clusters, when $\lambda_\eps$ is critical, (with probability one) the $(s,p)$-capacity density of the obstacle set $T_\eps$ is finite and asymptotically additive, as in the periodic setting.    

The proof of Theorem~\ref{thm: main result} builds on the argument outlined above, which, however, is not sufficient on its own to conclude. Indeed, both the lower and upper bound inequalities (cf.\ Proposition \ref{prop: gamma-liminf} and Proposition \ref{prop: gamma-limsup}, respectively) require two additional key ingredients: a nonlocal version of the so-called Joining Lemma (cf.\ \cite[Proof of Lemma~3.9]{focardi-aperiodic-fractional}) and a probabilistic discrete approximation of the capacitary term in the $\Gamma$-limit (cf.\ Proposition \ref{lmm: integral approximation G_j,R}).

More precisely, the Joining Lemma (originally proved in the local periodic setting by Ansini and Braides \cite{ansini-braides}) is reminiscent of a slicing and averaging argument due to De Giorgi. It allows us to identify suitable annular neighborhoods around the good obstacles where a sequence with equibounded energy can be modified so as to become locally constant without generating additional energy concentration. The resulting sequence can then be used as a competitor in the capacitary problem, and its appropriately chosen constant boundary values enable the reconstruction of the capacitary term in $\functional$ thanks to a refined version of the ergodic theorem for marked point processes that plays a crucial role (cf.\ Proposition \ref{prop: ergodic theorem mpp convergence of measures}).

To carry out this program, we introduce some technical simplifications and develop some methodological novelties with respect to the approaches in \cite{giunti-hoefer-velazquez,scardia-zemas-zeppieri}.
First, to make the application of the Joining Lemma possible we introduce a substantially simplified good/bad obstacles decomposition (see Section \ref{subsec:set-of-indices}), which yields the geometric separation properties required in the proof through a shorter argument than in \cite{giunti-hoefer-velazquez,scardia-zemas-zeppieri}. In particular, instead of relying on \cite[Lemma 4.2]{giunti-hoefer-velazquez}, we select the relevant obstacles by means of a simple truncation procedure depending on a deterministic parameter, which tends to $+\infty$ at the end of the argument. Second, we formulate the relevant discrete-to-continuum limit passages in terms of almost sure convergence of suitably rescaled counting measures (see \eqref{eq: rescaled measures}) associated with the underlying marked point process (see Sections \ref{sec: ergodic theorems} and \ref{sec: capacity term as limit of random sums}). This allows us to derive the asymptotic behavior of random sums from general ergodic principles and to recast the corresponding limit computations as convergence-to-integral statements. A key point here is that, rather than passing to the limit only at the level of the original marked point process, we work directly with the \emph{thinned process} (see Section \ref{sec: thinned process}) associated with the obstacle selection. The probabilistic analysis of this selected configuration relies on Palm measures (see Section \ref{sec: Palm}), which allow us to compute the expectation measure of the thinned process as a function of the truncation parameter. This in turn makes it possible to apply the ergodic theorem directly at the level of the selected obstacles.

While in the proof of the lower bound inequality the contribution of clusters can be neglected, the upper bound is more delicate, as one needs to show that the lower bound is optimal along a suitable recovery sequence. In particular, the recovery sequence must be constructed so as to satisfy the pinning condition also on the clustering obstacles. This construction presents difficulties which are, to some extent, analogous to those encountered in the construction of oscillating test functions, although no explicit solution of PDEs is required.
Finally, in the upper bound construction we show that the capacitary contribution is concentrated along the diagonal set in $\mathbb R^n \times \mathbb R^n$ and arises from short-range interactions,  whereas long range interactions contribute only to the nonlocal term in the $\Gamma$-limit. 

To conclude this introduction we note that the $\Gamma$-convergence result described above can be extended to more general interaction kernels and to obstacles with random shapes (cf.\ Theorem~\ref{thm: main result random shapes}). Moreover, the ergodicity assumption on the marked point process can be relaxed, in which case the limiting functional features a random capacitary term (cf.\ Remark \ref{thm: main result, nonergodic}).

\section{Preliminaries and notation}
\subsection{Basic notation}
In this subsection we introduce some useful notation. 
In all that follows $n \in \N \setminus \{0\}$. We define $\R_+ := [0,+\infty)$ and $Q := (0,1)^n$.
We use the standard notation 
$|\cdot|$ for the Euclidean norm on $\R^n$.
The open ball centered at $x \in \R^n$ and with radius $r>0$ is denoted by $B_r(x)$; \ie $B_r(x):= \{y \in \R^n: |y-x| < r\}$, while the corresponding closed ball is denoted by $\overline{B}_r(x)$, that is $\overline{B}_r(x):= \{y \in \R^n: |y-x| \leq r\}$. If $x=0$, we write simply $B_r$ instead of $B_r(0)$.
Given two sets $A, A' \subset \R^n$ with $A \subset \subset A'$, we say that $\varphi \in \lip(\R^n; [0,1])$ is a \textit{cutoff function} between $A$ and $A'$ if $\varphi = 1$ on $A$, $\varphi = 0$ on $\R^n \setminus A'$ and $\lip(\varphi) \leq 1/\dist(A,\partial A')$.

For a given set $A\subset \R^n$ the characteristic function of $A$ is denoted by $\mathds{1}_A$, that is 
\begin{equation*}
    \mathds{1}_A(x) := \begin{cases}
        1 \quad &\text{if } x \in A, \\
        0 \quad &\text{if } x \notin A.
    \end{cases}
\end{equation*}
Throughout the paper the parameter $\eps$ varies in a strictly decreasing sequence of positive real numbers converging to zero.
To avoid cumbersome notation, whenever we need to pick a sequence $\{\eps_j\}_{j \in \N} \downarrow 0$ , we will simply write $j$ instead of $\eps_j$, if $\eps_j$ appears as a subscript. For example, we will write $\insideDomainJ$, $\functional_j$, $u_j$ in place of $I_{\eps_j}$, $\functional_{\eps_j}$, $u_{\eps_j}$.

Let $\L^n$ denote the Lebesgue measure on $\R^n$. If $\domain$ is a Lebesgue-measurable subset of $\R^n$ and $u \in L^1(\domain)$, the integral average of $u$ on $\domain$ is denoted by $(u)_\domain$, that is $(u)_\domain := \fint_\domain u = \L^n(\domain)^{-1} \int_\domain u(x) \, \dd{x}$.

Let $\delta >0$; then $\diagonalDelta := \{(x,y) \in \R^n \times \R^n : |x-y| < \delta\}$ denotes the open $\delta$-neighborhood of the diagonal set in $\R^n \times \R^n$.

We use the standard notation $\lfloor \cdot \rfloor$ for the floor function; \ie $\lfloor t \rfloor:= \max\{m \in \Z: m \leq t\}$, for every $t \in \R$.

In all that follows we write $\lesssim_{M_1, M_2, \ldots}$ when an inequality holds up to a multiplicative constant depending on the parameters $M_1, M_2, \ldots$

\subsection{Fractional capacity}\label{ss:sobslob}
Throughout the paper, we assume that $p \in (1,\infty)$, $s \in (0,1)$ and $sp\in (0,n)$. Let $A \subset \R^n$ be open. We use the standard notation $\Wsp(A)$ for the Sobolev-Slobodeckij space, with norm $\norm{u}_{\Wsp(A)} := \norm{u}_{L^p(A)} + |u|_{\Wsp(A)}$, where 
$$
|u|_{\Wsp(A)}^p:=\int_{A\times A}\frac{|u(x)-u(y)|^p}{|x-y|^{n+sp}}\, \dd{x} \, \dd{y}
$$
denotes the usual fractional Gagliardo seminorm of $u$
to the $p$-power.

Let $\obstacle\subset \R^n$. The fractional capacity of $T$ is denoted by $\csp(\obstacle)$ and defined as
\begin{equation}\label{e:csp}
\csp(\obstacle):=\inf_{\{A\in\mathcal{A}(\R^n):\,A \supset \obstacle\}}
\inf\left\{|u|^p_{\Wsp(\R^n)}:\,
u\in\Wsp(\R^n),\,u\geq 1\, \L^n\text{-a.e. on } A\right\},
\end{equation}
where $\mathcal{A}(\R^n)$ stands for the collection of all open subsets of $\R^n$.

We recall that a property holds $\csp$ \emph{quasi everywhere} (in short, $\csp$-q.e.) on $A$, if it holds up to a set of $\csp$ zero. 
Moreover, every function $u$ in $\Wsp(A)$
has a \emph{precise representative} $\tilde{u}$ defined $\csp$-q.e. (see \cite{AdamsHedberg96, Warma2015}) 
and using the precise representative we can equivalently write
\begin{equation*}
  \csp(\obstacle)=\inf\left\{|u|_{\Wsp(\R^n)}^p:\,
u\in\Wsp(\R^n),\,\tilde{u}\geq 1\text{ q.e. on } \obstacle\right\}.
\end{equation*}

\noindent More generally, we shall deal with translation-invariant, even symmetric, homogeneous kernels $\mathscr{K}: \mathbb{R}^n \setminus \{0\} \to [0, \infty)$ 
that are comparable to the standard fractional kernel. Specifically, we assume that $\mathscr{K}$ is  $\L^n$-measurable and such that
for some $c > 0$, and for all $x \neq 0$ and $t \in\R\setminus\{0\}$, 
obeys to
\begin{equation}\label{e:Kernels}
\mathscr{K}(tx)=|t|^{-(n+sp)}\mathscr{K}(x),
\quad
c^{-1} \leq \mathscr{K}(x) |x|^{n+sp}\leq c\,.
\end{equation}
In particular, $\mathscr{K}$ is even symmetric and $-(n+sp)$-positive homogeneous.
Then consider the functional $\functionalK: L^p(\R^n)\times\mathcal{A}(\R^n) \longrightarrow [0, +\infty]$
\begin{equation}\label{e:functionalK}
 \functionalK(u,A):=\int_{A\times A}{\mathscr{K}(x-y)|u(x)-u(y)|^p}\, \dd{x} \, \dd{y}\,,
\end{equation}
and introduce the counterpart of $\csp$ associated with the kernel $\mathscr{K}$, namely
\begin{equation}\label{e:cspK}
\cspK(\obstacle):=\inf_{\{A\in\mathcal{A}(\R^n):\,A \supset \obstacle\}}
\inf\left\{\functionalK(u,\Rn):\,
u\in\Wsp(\R^n),\,u\geq 1\, \L^n\text{-a.e. on } A\right\}.
\end{equation}
One can easily verify that $\cspK$ is $(n-sp)$-homogeneous and comparable to $\csp$, namely
\begin{equation}\label{e:comparable-kernels}
c^{-1}\csp(\obstacle)\leq\cspK(\obstacle)\leq
c\,\csp(\obstacle),
\end{equation}
for every $\obstacle\subset\Rn$. 

For fixed $\obstacle\subset\R^n$, the existence and uniqueness of a function $u$ realizing $\cspK(T)$ is obtained in the 
homogeneous space $\Ksp(\Rn)=\{u\in L^{p^\ast}(\Rn):\, |u|_{\Wsp(\Rn)}<+\infty\}$,
$p^\ast:=\frac{np}{n-sp}$, appealing to the direct methods of the Calculus of Variations and using the strict convexity of the
$\mathcal{K}$ and the fact that the set $\{u\in\Ksp(\R^n):\,\tilde{u}\geq 1\text{ q.e. on }\obstacle\}$ is
convex and strongly closed. In what follows such a function is referred to as the $\mathscr{K}$-\emph{capacitary potential} for $T$.

For later use, it is also convenient to recall two different notions of relative fractional capacities introduced in \cite[Section~2.4]{focardi-aperiodic-fractional}.
Namely, let $0<r\leq R$, and for $T\subset B_r$ define
\begin{equation*}
    \cspK(\obstacle,B_R;r):=\inf\left\{\functionalK(u,B_R):\, 
    u\in\Wsp(\R^n),\, u=0 \text{ on }\R^n\setminus \overline{B}_r,\,
    \tilde{u}\geq 1\text{ q.e. on } \obstacle\right\},
\end{equation*}
and
\begin{equation*}
    \CSPK(\obstacle,B_R):=\inf\left\{\functionalK(u,\R^n):\,
    u\in\Wsp(\R^n),\, u=0 \text{ on }\R^n\setminus \overline{B}_R,\,
    \tilde{u}\geq 1\text{ q.e. on } \obstacle\right\}.
\end{equation*}
In fact, the former quantity is useful in the proof of the lower-bound inequality (cf. Proposition~\ref{prop: gamma-liminf}), 
while the latter is used in the proof of the upper bound inequality (cf. Proposition~\ref{prop: gamma-limsup}). 
In particular, by the very definitions it is easy to deduce that 
\begin{equation}\label{e:max cap}
\max\{\cspK(\obstacle),\cspK(\obstacle,B_R;r)\}\leq\CSPK(\obstacle,B_R)
\end{equation}
for every $\obstacle\subset B_r$, $0<r<R$.

In what follows we shall prove the uniform convergence of the above defined relative capacities to the global one.
To this aim it is convenient to introduce some notation to simplify the calculations below: 
for any $\L^n$-measurable function $w$ and any $\L^{2n}$-measurable 
subset $E$ of 
$\R^n\times\R^n$ we define the \textit{locality defect} of the functional $\mathcal{K}$ as 
\begin{equation*}
    \DspK(w,E) := \int_E \mathscr{K}(x-y)|w(x)-w(y)|^p \, \dd{x} \, \dd{y}. 
\end{equation*}
Hence, by definition and being $\mathscr{K}$ even symmetric, for any pair of disjoint $\L^n$-measurable subsets $A$ and $A'$ of $\R^n$, we have
\begin{equation}\label{e:locality defect}
   \mathcal{K}(w,A\cup A') = \mathcal{K}(w,A) + \mathcal{K}(w,A') + 2 \DspK(w,A \times A').
\end{equation}
The next result establishes some useful properties of the relative capacities introduced above. 
It extends \cite[Lemma~2.12]{focardi-aperiodic-fractional} proved for the standard kernel $|\cdot|^{-(n+sp)}$ to the more general class defined in \eqref{e:Kernels}. 
The strategy of proof is that of \cite[Lemma~2.12]{focardi-aperiodic-fractional} with the additional insight that the $\mathscr{K}$-capacitary potential of the ball is vanishing at infinity, as a consequence of the regularity theory developed in \cite{DCKP14}.
In what follows it is useful to notice that if $x\in B_R$ 
  \begin{equation}\label{e:stima Adams}
\int_{B_R^c}\frac 1{|x-y|^{n+sp}}\dd{y}
\lesssim_{n,s,p}\dist^{-sp}(x,\partial B_R)\,.
  \end{equation}
The latter estimate is a consequence of a direct integration using polar coordinates
(see \cite[Lemma~A.1]{focardi-aperiodic-fractional}).

\begin{lmm}\label{l:loccap}
Let $\mathscr{K}: \mathbb{R}^n \setminus \{0\} \to [0, \infty)$ be $\L^n$ measurable and satisfy \eqref{e:Kernels}.
For every $\rho>0$ it holds 
\begin{equation}\label{e:cap1}
\lim_{r\to+\infty}\sup_{T\subseteq B_\rho}|\CSPK(T,B_r)-\cspK(T)|=0.
\end{equation}
Moreover, 
for every $0<\rho<r<R$
\begin{equation}\label{e:cap2}
\sup_{T\subseteq B_\rho}\left(\cspK(T)-\cspK(T,B_R;r)\right)\lesssim_{n,s,p}
\frac{r^{sp}}{(R-r)^{sp}}\CSPK(B_\rho,B_r).
\end{equation}
In addition, if $R(r)>0$ is such that $R(r)/r\to+\infty$ as $r\to+\infty$, then 
\begin{equation}\label{e:cap3}
\lim_{r\to+\infty}\sup_{T\subseteq B_\rho}|\cspK(T)-\cspK(T,B_{R(r)};r)|=0\,,
\end{equation}
and 
\begin{equation}
    \label{eq: locality defect on B(r) x B(r)^c}
    \lim_{r\to+\infty}\sup_{T\subseteq B_\rho}\DspK(\xi_r^T,B_{R(r)}\times {B}_{R(r)}^c)=0,
\end{equation}
for every function $\xi_r^T$ such that 
\begin{equation}\label{e:almost-min}
\mathcal{K}(\xi_r^T,\Rn)\leq \CSPK(T,B_r)+\frac 1r.
\end{equation}
\end{lmm}
\begin{proof}
Let $\rho>0$, $T\subseteq B_\rho$, $r>\rho$ be fixed.

Let us first prove \eqref{e:cap1}. To this aim let $u^T$, $u^{B_\rho}\in\Ksp(\Rn)$ be the 
$\mathscr{K}$-capacitary potentials of $T$ and $B_\rho$, respectively; then $0\leq u^T\leq u^{B_\rho}\leq 1$
$\L^n$ a.e. on $\R^n$ by \cite[Lemma~2.11]{focardi-aperiodic-fractional}. In addition, $u^{B_\rho}$ satisfies 
\begin{equation}
    \lim_{|x|\to\infty}u^{B_\rho}(x)=0
\end{equation}
in view of the results contained in \cite{FPZ26}.

With fixed $\delta>0$ consider the Lipschitz map
$\psi_\delta(t):=\frac{t-\delta}{1-\delta}\vee 0$, 
$\mathrm{Lip}(\psi_\delta)\leq(1-\delta)^{-1}$, and set
$w_\delta(x):=\psi_\delta(u^T(x))$. 
Up to $\L^n$ negligible sets, $\{w_\delta>0\}=\{u^T>\delta\}\subseteq
\{u^{B_\rho}>\delta\}\subseteq B_{R_\delta}$, for some $R_\delta\to+\infty$
as $\delta\to 0^+$. Then $w_\delta\in\Wsp(\Rn)$ with 
\[
\mathcal{K}(w_\delta,\Rn)\leq\frac{1}{(1-\delta)^p}\mathcal{K}(u^T,\Rn)=
\frac{1}{(1-\delta)^p}\cspK(T),\quad
\|w_\delta\|_{L^p(\Rn)}\leq\frac{1}{1-\delta}\|u^T\|_{L^p(B_{R_\delta})}.
\]
Moreover, $\tilde{w}_\delta\geq 1$ q.e. on $T$, and being $\cspK(\cdot)$ an increasing set function, we conclude that
\[
0\leq\CSPK(T,R_\delta)-\cspK(T)\leq\left(\frac{1}{(1-\delta)^p}-1\right)
\cspK(T)\leq\left(\frac{1}{(1-\delta)^p}-1\right)\cspK(B_\rho).
\]
In conclusion, \eqref{e:cap1} follows since 
$(0,+\infty)\ni r\to\CSPK(T,B_r)$ is monotone decreasing.

To prove estimate \eqref{e:cap2}, let $u\in\Wsp(\R^n)$ be any admissible function 
for the minimum problem defining $\csp(T,B_R;r)$, then $u$ is also admissible 
for the one defining $\csp(T)$. Using the estimate in \eqref{e:stima Adams} we have
\begin{eqnarray*}
\lefteqn{\cspK(T)\leq\mathcal{K}(u,\Rn)\stackrel{\eqref{e:locality defect}}{=}
\mathcal{K}(u,B_R)+\mathcal{K}(u,B_R^c)+2\DspK(u,B_R\times B_R^c)}\\&&
\stackrel{u|_{B_r^c}=0}{=}
\mathcal{K}(u,B_R)+2\int_{B_r}|u(x)|^p\int_{B^c_R}\mathscr{K}(x-y)\dd{y}\dd{x}\\&&
\stackrel{\eqref{e:Kernels},\eqref{e:stima Adams}}{\lesssim_{n,s,p}}
\mathcal{K}(u,B_R)+c \int_{B_r}\frac{|u(x)|^p}{\dist^{sp}(x,\partial B_R)}dx\\
&&\leq\mathcal{K}(u,B_R)+\frac{c }{(R-r)^{sp}}\int_{B_r}|u(x)|^pdx
\lesssim_{n,s,p}\mathcal{K}(u,B_R)+\frac{c  r^{sp}}{(R-r)^{sp}}\mathcal{K}(u,B_r)\,,
\end{eqnarray*}
where in the last inequality we used the scaled version of the Poincaré Inequality on balls and again \eqref{e:Kernels}.
By passing to the infimum over the admissible test functions we infer
\[
\cspK(T)-\cspK(T,B_R;r)\leq\frac{c  r^{sp}}{(R-r)^{sp}}\cspK(T,B_R;r)\leq
\frac{c  r^{sp}}{(R-r)^{sp}}\CSPK(T,B_r).
\]
We deduce statement \eqref{e:cap2} since $\CSPK(\cdot,B_r)$
is a monotone increasing set function.

The limit in \eqref{e:cap3} follows at once from formulas \eqref{e:cap1}, \eqref{e:cap2}, and the inequality in \eqref{e:max cap}.

Finally, let $\xi_r^T$ satisfy \eqref{e:almost-min}. Then $\xi_r^T$ is admissible for the problem defining
$\cspK(T,B_R;r)$, so that for all $r<R$ we get
\[
\cspK(T,B_R;r)\leq\mathcal{K}(\xi_r^T,B_R)\leq \CSPK(T,B_r)+\frac 1r\,.
\]
Therefore, recalling the definition of $\DspK$, by combining \eqref{e:cap1} and \eqref{e:cap3} 
we get \eqref{eq: locality defect on B(r) x B(r)^c}.
\end{proof}

\subsection{Marked point processes}\label{ss:mpp}
For a comprehensive treatment of marked point processes we refer the reader to the monographs \cite{daley-vere-jones-1, daley-vere-jones-2}. Below we recall some notions and results which are relevant to the problem under examination. 

\subsubsection{Basic definitions}\label{sec: mpp, basic definitions}
Let $X$ be a topological space. The Borel $\sigma$-algebra on $X$ is denoted by $\mathcal B(X)$. We say that a Borel measure is boundedly finite if it is finite on every bounded Borel set.
Let $\spaceOfMeasures$ be the set of boundedly finite integer-valued measures on $\borel(\R^n \times \R_+)$. For any $\mu \in \spaceOfMeasures$, we define its ground measure $\mu_g$ as 
\[
\mu_g(A) = \mu(A \times \R_+),
\]
for every $A \in \borel(\R^n)$. If $\mu_g$ is a simple counting measure on $\borel(\R^n)$, \ie if $\mu_g(\{x\}) \in \{0,1\}$ for every $x \in \R^n$, we write $\mu \in \spaceOfMeasuresWithSimpleGround$.

We say that a sequence $\{\mu_j\}_{j \in \N} \subset \spaceOfMeasures$ $w^\#$-converges (``weak hash'' converges) to $\mu \in \spaceOfMeasures$ if 
\[
\int_{\R^n \times \R_+} f \, \dd{\mu}_j \to \int_{\R^n \times \R_+} f \, \dd{\mu}
\]
for every bounded continuous function $f$ on $\R^n \times \R_+$ vanishing outside a bounded set. This notion of convergence defines the $w^\#$-topology (``weak hash'' topology) on $\spaceOfMeasures$, whose  Borel $\sigma$-algebra is denoted by $\borel(\spaceOfMeasures)$. The $w^\#$-topology is metrizable on $\spaceOfMeasures$ (see \cite[A2.6.1]{daley-vere-jones-1}).
Moreover, by \cite[Proposition~9.1.IV]{daley-vere-jones-2}, $\spaceOfMeasures$, with the $w^\#$-topology, is a complete separable metric space and $\borel(\spaceOfMeasures)$ is the smallest $\sigma$-algebra with respect to which the mappings $\spaceOfMeasures \ni \mu \mapsto \mu(B)$ are measurable for every $B \in \borel(\R^n\times\R_+)$.

Let $\probSpace$ be a probability space. In all that follows follows, whenever a statement holds true $\mathbb P$-a.e. in $\Omega$ we may equivalently write \emph{almost surely}. 

A \textit{marked point process} (m.p.p.)  on $\R^n$ with marks in $\R_+$ is a measurable mapping  $ \PP \colon (\sampleSpace,\sigmaAlgebra) \rightarrow (\spaceOfMeasures,\borel(\spaceOfMeasures))$ such that $\mathbb{P}(\PP \in \spaceOfMeasuresWithSimpleGround)=1$. 
Let $N$ be a m.p.p. on $\R^n$ with marks in $\R_+$;
for every $\omega \in \sampleSpace$, set $\PP^\omega := \PP(\omega)$. The probability law of $\PP$ is denoted by $\probLawPP:= \prob \circ \PP^{-1}$. Define $\eventSimpleGround := \PP^{-1}(\spaceOfMeasuresWithSimpleGround)$; clearly $\eventSimpleGround \subset \sampleSpace$ and $\prob (\eventSimpleGround)=1$.

For $y \in \R^n$ we define the shift operator $S_y: \spaceOfMeasures \to \spaceOfMeasures$ as 
\begin{equation*}
    (S_y \mu)(B) := \mu(S_y'B),
\end{equation*}
for every $B \in \borel(\R^n \times \R_+)$, where 
\begin{equation*}
    S'_y B := \{(x+y,\rho): (x,\rho) \in B\}. 
\end{equation*}
We say that $\PP$ is stationary (with respect to $(S_y)_{y \in \R^n}$)  if
\begin{equation}
    \probLawPP \circ S_y = \probLawPP \quad \forall y \in \R^n. \label{eq: prob law invariant under shifts}
\end{equation}
Moreover, we say that $\PP$ is ergodic (with respect to $(S_y)_{y \in \R^n}$) if, in addition, 
\begin{equation}
    E \in \borel(\spaceOfMeasures), \quad S_y E = E \quad \forall y \in \R^n \, \implies \, \probLawPP(E) \in \{0,1\}.\label{eq: prob law ergodic under shifts}
\end{equation}
Now suppose that $(\tau_y)_{y \in \R^n}$ is a flow of $\mathbb{P}$-preserving transformations on $\sampleSpace$, \ie a group of transformations $(\tau_y)_{y \in \R^n} : \sampleSpace \to \sampleSpace$ such that $(y,\omega) \mapsto \tau_y \omega$ is a measurable map from $(\R^n \times \sampleSpace, \borel(\R^n) \times \sigmaAlgebra)$ to $(\sampleSpace, \sigmaAlgebra)$ and $\mathbb{P} \circ \tau_y = \mathbb{P}$ for every $y \in \R^n$. We say that $(\tau_y)_{y \in \R^n}$ is ergodic if
\begin{equation*}
    F \in \sigmaAlgebra, \quad \tau_y F = F \quad \forall y \in \R^n \, \implies \, \prob(F) \in \{0,1\}.
\end{equation*}
We observe that if
\begin{equation}
    \label{eq: stationarity assumption for MPP}
    \PP^{\tau_y \omega} = S_y \PP^\omega,
\end{equation}
for every $(y,\omega) \in \R^n \times \sampleSpace$, then \eqref{eq: prob law invariant under shifts} holds true. Moreover, if additionally $(\tau_y)_{y \in \R^n}$ is ergodic, then \eqref{eq: prob law ergodic under shifts} holds also true. 

Following \cite[Chapter 9.5]{daley-vere-jones-2} we introduce the \textit{expectation measures} of $\PP$ and $\PPg$ defined, respectively, as 
\begin{equation*}
    \EM(B) := \E[\PP(B)], \quad \EMg(A) := \E[\PPg(A)]
\end{equation*}
for every $B \in \borel(\R^n\times\R_+)$ and $A \in \borel(\R^n)$, where $\E$ denotes the expected value with respect to the probability measure $\mathbb P$. Throughout this section we tacitly assume that $\EM$ and $\EMg$ are boundedly finite. 

Thanks to \cite[Lemma~A2.7.II]{daley-vere-jones-1} (see also the comments before \cite[Lemma~12.2.III]{daley-vere-jones-2}) property \eqref{eq: prob law invariant under shifts} entails 
\[
\EM = \L^n \times \nu, 
\]
where $\nu(\R_+)< \infty$, since $M_g$ is boundedly finite. Therefore,
$\nu$ can be normalized and $\EM$ can be rewritten as
\begin{equation*}
    \EM = \IG \L^n \times \MD,
\end{equation*} 
where $\IG := \E[\PP_g(Q)] \in [0,+\infty)$ and $\MD$ is a probability measure on $(\R_+, \borel(\R_+))$. The constant $\IG$ is referred to as the \textit{intensity of the ground process}. 

\subsubsection{Ergodic theorems}\label{sec: ergodic theorems}
For later purposes, in this section we collect some useful consequences of the Ergodic Theorem for stationary and ergodic marked point processes \cite[Theorem~12.2.IV and Corollary~12.2.V]{daley-vere-jones-2}.

We start recalling that, since $\R^n \times \R_+$ is a complete separable metric space, we can appeal to \cite[Lemma~9.1.XIII]{daley-vere-jones-2} to deduce the existence of a sequence $\{x_i\}_{i \in \N}$ of $\R^n$-valued random variables and a sequence $\{\rho_i\}_{i \in \N}$ of $\R_+$-valued random variables such that almost surely 
\begin{equation*}
    \PP = \sum_{i = 1}^{\infty} \delta_{(x_i,\rho_i)}. 
\end{equation*}
Moreover, for every $\eps >0$, we define the rescaled measures  
\begin{equation}\label{eq: rescaled measures}
    \PP_\eps := \eps^n \sum_{i = 1}^\infty \delta_{(\eps x_i, \rho_i)}, \quad \PP_{g,\eps} = \eps^n \sum_{i=1}^\infty \delta_{\eps x_i}.
\end{equation}
The following result is a well known consequence of the Ergodic Theorem for stationary and ergodic point processes \cite[Theorem~12.2.IV and Corollary~12.2.V]{daley-vere-jones-2} and can be found, for example, in \cite[Claim 7.1]{faggionato}. 
\begin{prop} 
    \label{prop: ergodic theorem ground pp convergence of measures}
    Let $\PP$ be a stationary and ergodic m.p.p. on $\R^n$ with marks in $\R_+$. Then  
    there exists $\sampleSpace' \in \sigmaAlgebra$ with $\prob(\sampleSpace')=1$ such that
    \begin{equation*}
        \lim_{\eps \downarrow 0} \int_{\R^n} f(x) \, \dd{\PP_{g,\eps}^\omega}(x) = \int_{\R^n} f(x) \, \dd{\EM_g}(x),
    \end{equation*}  
    for every $\omega \in \sampleSpace'$ and every $f \in C_c^0(\R^n)$.     
\end{prop}
An analogous result holds true for $\PP_\eps$ as stated in the following proposition. 
\begin{prop}
    \label{prop: ergodic theorem mpp convergence of measures}
    Let $\PP$ be a stationary and ergodic m.p.p. on $\R^n$ with marks in $\R_+$.
    Then, for every $h \in L^1(\R_+,\MD)$ there exists $\sampleSpace_h \in \sigmaAlgebra$ with $\prob(\sampleSpace_h)=1$ such that
    \begin{equation*}
        \lim_{\eps \downarrow 0} \int_{\R^n \times \R_+} f(x)h(\rho) \, \dd{\PP_\eps^\omega}(x,\rho) = \int_{\R^n\times \R_+} f(x) h(\rho) \, \dd{\EM}(x,\rho),
    \end{equation*}
for every $\omega \in \sampleSpace_h$ and every $f \in C_c^0(\R^n)$.   
\end{prop} 
As a simple consequence, we also have the following result. 
\begin{prop}
    \label{prop: ergodic theorem, integral on eps^-1 U of h(rho)}
    Let $\PP$ be a stationary and ergodic m.p.p. on $\R^n$ with marks in $\R_+$. 
    Then, for every $h \in L^1(\R_+,\MD)$ there exists $\sampleSpace_h \in \sigmaAlgebra$ with $\prob(\sampleSpace_h)=1$ such that
    \begin{equation*}
        \lim_{\eps \downarrow 0} \eps^n \int_{(\eps^{-1} \domain) \times \R_+} h(\rho) \, \dd{\PP^\omega}(x,\rho) = \IG \L^n(\domain) \int_{\R_+} h(\rho) \, \dd{\MD}(\rho),
    \end{equation*}
for every $\omega \in \sampleSpace_h$ and for every $\domain \in \mathcal A(\R^n)$ bounded and Lipschitz.    
\end{prop}
The corresponding proofs are postponed to Appendix \ref{section: appendix}. 
\subsubsection{Palm distributions}\label{sec: Palm}
In this section we recall the notion of \textit{Palm distributions} for a stationary marked point process. These are introduced through a refined version of the Campbell Theorem (see \cite[Theorem~13.2.III]{daley-vere-jones-2} and the introduction to \cite[Section 13.4]{daley-vere-jones-2}). 

\begin{thm}[Refined Campbell's Theorem]\label{thm: refined Campbell theorem} Let $\PP$ be a stationary and ergodic m.p.p. on $\R^n$ with marks in $\R_+$. Then there exists a family $\{\probLawPP_{(0,\rho)}\}_{\rho \in \R_+}$ of probability measures on $(\spaceOfMeasures, \borel(\spaceOfMeasures))$, called Palm distributions associated to the m.p.p. $\PP$, such that
    \begin{equation}\label{eq: Palm}
        \E \left[\int_{\R^n \times \R_+} f(x,\rho,\PP) \, \dd{\PP}(x,\rho)\right] = \int_{\R^n \times \R_+ \times \spaceOfMeasures} f(x,\rho,S_{-x} \mu) \, \dd{\EM}(x,\rho) \, \dd{\probLawPP_{(0,\rho)}}(\mu),
    \end{equation}
for every nonnegative $(\borel(\R^n) \otimes \borel(\R_+) \otimes \borel(\spaceOfMeasures))$-measurable function $f$.     
    \end{thm} 

We notice that, taking $f(x,\rho,\mu) = \tilde{f}(x,\rho,S_x \mu)$ for a nonnegative $(\borel(\R^n) \otimes \borel(\R_+) \otimes \borel(\spaceOfMeasures))$-measurable function $\tilde{f}$, \eqref{eq: Palm} can be equivalently written as 
    \begin{equation}\label{eq: Palm-rw}
         \E \left[\int_{\R^n \times \R_+} f(x,\rho,S_x \PP) \, \dd{\PP}(x,\rho)\right] = \int_{\R^n \times \R_+ \times \spaceOfMeasures} f(x,\rho,\mu) \, \dd{\EM}(x,\rho) \, \dd{\probLawPP_{(0,\rho)}}(\mu).
\end{equation}
Let $A \in \borel(\R^n)$ and $A' \in \borel(\R_+)$ be such that $\EM(A \times A') \in (0,\infty)$.
Choosing $f(x,\rho,\mu) = \IF_A(x) \IF_{A'}(\rho) \IF_E(\mu)$, with $E \in \borel(\spaceOfMeasures)$, by \eqref{eq: Palm-rw} we get
\begin{equation*}
    \probLawPP_{(0,\rho)} (E) = \frac{1}{\EM(A\times A')} \E \left[\int_{\R^n \times \R_+ } \IF_A(x) \IF_{A'}(\rho) \IF_E(S_x \PP) \, \dd{\PP}(x,\rho)\right].   
\end{equation*}
Consider now the set of measures $\spaceOfMeasures_{(0,\rho)}:= \{\mu \in \spaceOfMeasures: \mu(\{(0,\rho)\}) > 0\}$. Clearly, if $(x,\rho) \in \supp \PP$, then $(0,\rho) \in \supp S_x \PP$, and thus $\IF_{\spaceOfMeasures_{(0,\rho)}}(S_x \PP) = 1$. The latter equality readily implies that $\probLawPP_{(0,\rho)}(\spaceOfMeasures_{(0,\rho)})=1$, hence $\probLawPP_{(0,\rho)}$ is concentrated on $\spaceOfMeasures_{(0,\rho)}$. Moreover, since $\spaceOfMeasures_{(0,\rho)} \subset \spaceOfMeasures_0 := \{\mu \in \spaceOfMeasures : \mu(\{0\}\times \R_+) > 0\}$ we also have $\probLawPP_{(0,\rho)} (\spaceOfMeasures_0) = 1$. Similarly, since $\probLawPP(\spaceOfMeasuresWithSimpleGround) = 1$, we have $\probLawPP_{(0,\rho)}(\spaceOfMeasuresWithSimpleGround) = 1$.

\subsubsection{Thinned process}\label{sec: thinned process}
In this section we introduce some notion of thinning of point processes. Let $\PP$ be a m.p.p. on $\R^n$ with marks in $\R_+$;
for $\delta >0$ fixed, we define the thinned process 
\begin{equation*}
    \PP^\delta := \sum_{i \in I^\delta} \delta_{(x_i,\rho_i)} = \sum_{i = 1}^\infty \IF_{Y_\delta}(S_{x_i} \PP) \delta_{(x_i,\rho_i)}, 
\end{equation*}
with
\begin{equation*}
    I^\delta := \{i \in \N: \min_{j \neq i} |x_i - x_j| \geq \delta\}, \quad Y_\delta = \{\mu \in \spaceOfMeasures: \mu(B_\delta \times \R_+) = 1\}.
\end{equation*}

We observe that in the context of homogeneous Poisson point processes $N^\delta$ is known as \emph{first Matérn hard-core process} \cite{matern}.

Now let $\PP$ be stationary and ergodic with respect to $(S_y)_{y \in \R^n}$. Since the map $\PP \mapsto \PP^\delta$ commutes with $S_y$ for every $y \in \R^n$, it is immediate to deduce that $\PP^\delta$ is stationary and ergodic with respect to $(S_y)_{y \in \R^n}$. Therefore, using Theorem~\ref{thm: refined Campbell theorem}, we can compute the expectation measure of $N^\delta$ as follows. 
Let $A \in \borel(\R^n)$ and $A' \in \borel(\R_+)$; in view of the definition of $\PP^\delta$ we have 
\begin{equation}\label{eq: expectation measure thinned process} 
    \begin{split}
        \E[\PP^\delta(A \times A')] &= \E \left[\sum_{i = 1}^\infty \IF_{Y_\delta}(S_{x_i} \PP) \IF_A(x_i) \IF_{A'}(\rho_i)\right]\\
        & = \E \left[\int_{\R^n \times \R_+} \IF_A(x)\IF_{A'}(\rho) \IF_{Y_\delta}(S_x \PP) \, \dd{\PP}(x,\rho)\right] \\
        &= \int_{\R^n \times \R_+ \times \spaceOfMeasures} \IF_A(x) \IF_{A'}(\rho) \IF_{Y_\delta}(\mu) \, \dd{\EM}(x,\rho) \, \dd{\probLawPP_{(0,\rho)}}(\mu) \\
        &= \IG \L^n(A) \int_{A'} \probLawPP_{(0,\rho)}(Y_\delta) \, \dd{\MD}(\rho). 
    \end{split}
\end{equation}
Therefore, thanks to \eqref{eq: expectation measure thinned process}, the expectation measure $\EM^\delta$ of the thinned process $\PP^\delta$ can be expressed in terms of the Palm measure $\probLawPP_{(0, \rho)}$ as
\[
\EM^\delta =\IG \L^n \times (\probLawPP_{(0, \cdot)}(Y_\delta) \MD).
 \]
Furthermore, taking $A' = \R_+$ in \eqref{eq: expectation measure thinned process}, we deduce that the expectation measure $\EM_g^\delta$ of the ground process $\PP_g^\delta$ is given by 
\[
\EM_g^\delta = \IG \int_{\R_+} \probLawPP_{(0,\rho)} (Y_\delta) \, \dd{\MD}(\rho) \, \L^n.
\]
We conclude this section by observing that
\begin{equation}\label{eq: Palm limit delta}
\lim_{\delta \downarrow 0}\probLawPP_{(0,\rho)}(Y_\delta) = 1,    
\end{equation}
for every $\rho \in \R_+$. Indeed, we have
\[
\probLawPP_{(0,\rho)}(Y_\delta) = \int_{\spaceOfMeasures_0 \cap \spaceOfMeasuresWithSimpleGround} \IF_{Y_\delta}(\mu) \, \dd{\probLawPP_{(0,\rho)}}(\mu);
\]
moreover, every $\mu \in \spaceOfMeasures_0 \cap \spaceOfMeasuresWithSimpleGround$ belongs to $Y_\delta$ for sufficiently small $\delta >0$. Hence, \eqref{eq: Palm limit delta} follows by the Dominated Convergence Theorem.

\section{Statement of the main result}\label{ss:main result}
In this section we introduce the random nonlocal functionals we are going to study and state the main result of this work. 

\subsection{Assumptions}\label{ss: assumptions}
Let $\PP$ be a m.p.p. on $\R^n$ with marks in $\R_+$. As already observed, $\PP$ can be written as a countable sum of Dirac deltas; \ie
\begin{equation*}
    \PP = \sum_{i = 1}^{\infty} \delta_{(x_i,\rho_i)} \quad \mathbb{P}\text{-a.e.}
\end{equation*}
for some $\R^n$-valued random variables $\{x_i\}_{i \in \N}$ and   $\R_+$-valued random variables $\{\rho_i\}_{i \in \N}$.

We assume that $\PP$ satisfies the following hypotheses:
\begin{enumerate}[label=(H\arabic*)]
    \item\label{h: stationarity and ergodicity} $\PP$ is stationary and ergodic;
    \item \label{h: boundedly-finite} the expectation measures $\EM$ and $\EMg$ are boundedly finite; 
    \item \label{h: positive intensity} $\IG:=\E[\PPg(Q)]>0$; 
    \item \label{h: moment condition} $\{\rho_i\}_{i \in \N}$ satisfy the moment condition
    \begin{equation}\label{eq: moment condition}
    \E\bigg[\sum_{x_i\in Q} \rho_i^{n-sp}\bigg]<+\infty. 
    \end{equation}
\end{enumerate}
\begin{rmk}[On the assumptions]\label{rmk:On the assumptions}
Some comments on hypotheses \ref{h: stationarity and ergodicity}-\ref{h: moment condition} are in order. 

We notice that thanks to \ref{h: stationarity and ergodicity} and \ref{h: boundedly-finite} we know that $ \EM = \IG \L^n \times \MD$, with $\IG<+\infty$ and $\MD$ probability measure on $(\R_+,\mathcal B(\R_+))$, while \ref{h: positive intensity} ensures that the ground process $\PPg$ is nontrivial. 

Moreover, in view of Theorem~\ref{thm: refined Campbell theorem}, by choosing $f(x,\rho,\PP)=\IF_{Q}(x)\,\rho^{n-sp}$ in \eqref{eq: Palm}, we obtain
\begin{equation}\label{e:expectation rhoi vs pi}
\E\bigg[\sum_{x_i\in Q} \rho_i^{n-sp}\bigg]=\IG \int_{\R_+} \rho^{n-sp} \, \dd{\MD}(\rho).
\end{equation}
Therefore, the moment condition \eqref{eq: moment condition} in \ref{h: moment condition}  can be  equivalently replaced by 
\begin{equation}
    \label{eq: integrability assumption on radii}
    \int_{\R_+} \rho^{n-sp} \, \dd{\MD}(\rho) < +\infty.
\end{equation}
\end{rmk}

\subsection{The functionals}
Let $\domain \subset \R^n$ be open, bounded, and Lipschitz. For $\omega \in \eventSimpleGround$ and $\eps >0$ consider the set of indices  
\begin{equation*}
    \insideDomainEpsOmega:= \{i \in \N: x_i(\omega) \in \eps^{-1}\domain\}
\end{equation*}
and the parameter
\begin{equation*}
    \lambda_\eps := \eps^\frac{n}{n-sp}.
\end{equation*}
We define the \textit{obstacle set} $\setOfObstaclesEpsOmega$ as 
\begin{equation*}
    \setOfObstaclesEpsOmega  := \bigcup_{i \in \insideDomainEpsOmega} \obstacleEpsIOmega,
\end{equation*}
where
\begin{equation*}
    \obstacleEpsIOmega := \ballCriticalRadiusEpsIomega. 
\end{equation*}

Finally, we define the functionals $\functional_\eps^\omega: L^p(\domain) \longrightarrow [0, +\infty]$ as 
\begin{equation}\label{eq: functionals-e} 
    \functional_\eps^\omega(u) := 
   \begin{cases}
        \displaystyle\int_{\domain\times \domain}\frac{|u(x)-u(y)|^p}{|x-y|^{n+sp}}\, \dd{x} \, \dd{y} & \text{ if } u\in \Wsp(\domain),\,
        \tilde{u}=0\,\,  \csp\hbox{-}\text{q.e. on } \setOfObstaclesEpsOmega \cap \domain,\\\\
        +\infty & \text{ otherwise. }
  \end{cases}
\end{equation}
\subsection{The main result}
We are now in a position to state the main theorem of this work, that is, an \emph{almost sure} $\Gamma$-convergence result for the random functionals $\functional_\eps$. 
\begin{thm}[$\Gamma$-convergence]\label{thm: main result}
Let $\PP$ be a m.p.p. on $\R^n$ with marks in $\R_+$ satisfying \ref{h: stationarity and ergodicity}-\ref{h: moment condition}. Let $\functional_\eps^\omega$ be the functionals defined as in \eqref{eq: functionals-e}. Then there exists $\sampleSpace' \in \sigmaAlgebra$ with $\mathbb{P}(\sampleSpace')=1$ such that, for every $\omega \in \sampleSpace'$, $(\functional_\eps^\omega)$ $\Gamma$-converges in the  $L^p(\domain)$-topology, as $\eps \downarrow 0$, to the deterministic functional $\functional : L^p(\domain) \longrightarrow [0,+\infty]$ given by 
    \begin{equation*}
        \functional(u) := 
        \begin{cases}
            \displaystyle\int_{\domain\times \domain}\frac{|u(x)-u(y)|^p}{|x-y|^{n+sp}}\, \dd{x} \, \dd{y} +\capacityConstantErgodic\int_{\domain}|u|^p\dd{x} &\text{ if } u \in \Wsp(\domain), \\\\
            +\infty &\text{ otherwise, }
        \end{cases} 
    \end{equation*}
where 
\begin{equation*}
    \capacityConstantErgodic := \csp(B_1)\, \E\bigg[\sum_{x_i\in Q} \rho_i^{n-sp}\bigg].
\end{equation*}
\end{thm}

Some remarks are in order. 

\begin{rmk}[Periodically distributed obstacles with random radii]\label{rmk:periodic}
  We observe that the case of periodically distributed obstacles with random radii can be recovered as a special case of Theorem~\ref{thm: main result}.

    To this end, let $\{z_i\}_{i \in \N}$ be an enumeration of $\Z^n$. We notice that the ground process corresponding to $\{z_i\}_{i \in \N}$ is stationary with respect to \emph{integer} translations in $\R^n$, but not stationary with respect to $(S_y)_{y \in \R^n}$; therefore this does not satisfy \ref{h: stationarity and ergodicity}. Hence, to produce a ground process which is both periodic and 
    $(S_y)_{y \in \R^n}$-stationary we need to consider a random translation of $\Z^n$ as follows. Let $\zeta$ be a uniformly distributed random variable on $[0,1)^n$. For any $i \in \N$, let $x_i := z_i + \zeta$. Then the ground process $\PPg = \sum_{i=1}^\infty \delta_{x_i}$ is stationary with respect to $(S_y)_{y \in \R^n}$. To see that, we fix $y \in \R^n$ and we define the random variable $\zeta_y$ by
    \begin{equation*}
    \zeta_y := (\zeta-y)\,\mathrm{mod}\,\Z^n,
    \end{equation*}
    where $(\cdot)\,\mathrm{mod}\,\Z^n$ denotes the unique representative in $[0,1)^n$. We also define
    \begin{equation*}
    z_y := (\zeta-y)-\zeta_y \in \Z^n,
    \end{equation*}
    so that $\zeta-y = \zeta_y + z_y$. Observe that $\zeta_y$ and $\zeta$ have the same probability distribution. Moreover,
    \begin{equation*}
    S_y \PPg = \sum_{i=1}^\infty \delta_{z_i + \zeta-y} = \sum_{i=1}^\infty \delta_{z_i + \zeta_y + z_y} = \sum_{i=1}^\infty \delta_{z_i + \zeta_y},
    \end{equation*}
    since clearly $\Z^n + z_y = \Z^n$. Hence, $S_y\PPg$ has the same probability distribution as $\sum_{i=1}^\infty \delta_{z_i + \zeta} = \PPg$.
    Furthermore, for every $\omega \in \sampleSpace$, $\PPg^\omega$ is $\Z^n$-periodic, being a translated copy of $\Z^n$. 
\end{rmk}
\begin{rmk}[Wald identity]\label{r:Wald}
We observe that, under some additional assumptions on the m.p.p., the constant $\capacityConstantErgodic$ in Theorem~\ref{thm: main result} has a more explicit expression in the spirit of the classical Wald Identity (see e.g.\ \cite{Blackwell}). Namely, 
besides stationarity and ergodicity, suppose that the random variables $\rho_i^{n-sp}$ and $\IF_Q(x_i)$ are independent for every $i \in \N$. Moreover, assume that $\E[\rho_i^{n-sp}] = \E[\rho_1^{n-sp}]$ for every $i \in \N$.
Then we have
    \begin{equation*}
        \capacityConstantErgodic = \IG \E[\rho_1^{n-sp}].
    \end{equation*}
Indeed, under these assumptions, $\E[\rho_i^{n-sp} \IF_Q(x_i)] = \E[\rho_1^{n-sp}] \E[\IF_Q(x_i)]$ for every $i \in \N$, and hence, since $\rho_i^{n-sp} \IF_Q(x_i) \geq 0$, 
    by a consequence of the Monotone Convergence Theorem \cite[Theorem 2.15]{folland}, we obtain 
    \begin{equation*}
        \begin{split}
            \capacityConstantErgodic &= \E \sum_{x_i \in Q} \rho_i^{n-sp} = \E \sum_{i = 1}^\infty \rho_i^{n-sp} \IF_Q(x_i) = \sum_{i = 1}^\infty \E[\rho_i^{n-sp} \IF_Q(x_i)] \\
            &= \E[\rho_1^{n-sp}] \, \E \sum_{i=1}^\infty \IF_Q(x_i) = \E[\rho_1^{n-sp}] \, \E[\PPg(Q)] = \IG \E[\rho_1^{n-sp}].
        \end{split}
\end{equation*}
\end{rmk}

\begin{rmk}[Stationary nonergodic m.p.p.] \label{thm: main result, nonergodic}
It is also worth mentioning here that Theorem~\ref{thm: main result} is stated and proven in the ergodic case  for notational simplicity only. In fact, if in \ref{h: stationarity and ergodicity} the ergodicity assumption is dropped, then Theorem~\ref{thm: main result} still holds up to replacing the constant $\gamma$ with the random variable $\kappa$ defined as follows 
    \begin{equation}\label{e:kappa-random}
        \kappa^\omega := \csp(B_1) \, \E\bigg[\sum_{x_i\in Q} \rho_i^{n-sp} \, \big | \, \mathcal{I} \bigg](\omega). 
    \end{equation}
Here we additionally assume that the probability space $\probSpace$ where $\PP$ is defined is endowed with a flow of $\prob$-preserving transformations $(\tau_y)_{y \in \R^n}$ (see Section \ref{sec: mpp, basic definitions} for the definition) satisfying \eqref{eq: stationarity assumption for MPP}.
Moreover, in \eqref{e:kappa-random}, $\mathcal{I}$ denotes the sub-$\sigma$-algebra of $(\tau_y)_{y \in \R^n}$-invariant events; \ie 
    \begin{equation*}
    \mathcal{I} := \{E \in \mathscr{F} : \tau_y E = E \quad \forall y \in \R^n\}
    \end{equation*}
and $\E[ \cdot | \mathcal{I}]$ a version of the conditional expectation given $\mathcal{I}$. 
This additional structure is only used to formulate the nonergodic case on the underlying sample space and is not restrictive, since one may canonically introduce a new probability space $(\tilde{\sampleSpace},\tilde{\sigmaAlgebra},\tilde{\prob}) := (\spaceOfMeasures,\borel(\spaceOfMeasures),\probLawPP)$, endowed with the flow $\tilde{\tau}_y:=S_y$, together with the marked point process $\tilde{\PP}(\omega):=\omega$, which has the same law as $\PP$ and satisfies all the required properties.

By definition of conditional expectation, we readily obtain that $\E[\kappa] = \capacityConstantErgodic$. Then, by \ref{h: moment condition} we get that $\kappa^\omega < +\infty$ for $\P$-a.e. $\omega \in \sampleSpace$.

Furthermore, as in the ergodic case, $\kappa$ can be rewritten in terms of an integral over $\R_+$ with respect to some \emph{random} measure. Specifically, if $\mathcal{M}_{\R_+}$ denotes the set of boundedly finite measures on $\R_+$, by \cite[Lemma~12.2.III]{daley-vere-jones-2} there exists an $\mathcal{I}$-measurable random measure $\psi$ on $\R_+$, \ie a measurable mapping from $(\Omega, \mathcal{I})$ to $(\mathcal{M}_{\R_+}, \mathcal{B}(\mathcal{M}_{\R_+}))$, such that almost surely
\begin{equation}
    \label{eq: definition of conditional expectation measure}
\mathbb{E} \left[ \int_{\mathbb{R}^n \times \R_+} f(x, \rho) \, \dd{N} (x, \rho) \,\middle| \, \mathcal{I} \right]
= \int_{\R^n \times \R_+} f(x, \rho) \, \dd{x} \, \dd{\psi}(\rho),
\end{equation}
for all nonnegative measurable functions $f$ on $\R^n \times \R_+$. 
Then, choosing $f(x,\rho) = \IF_{Q}(x) \rho^{n-sp}$, we get
\begin{equation*}
    \kappa^\omega = \csp(B_1) \, \int_{\R_+} \rho^{n-sp} \, \dd{\psi^\omega}(\rho),
\end{equation*}
for $\P$-a.e. $\omega \in \sampleSpace$.

We conclude this remark by observing that a proof of the $\Gamma$-convergence result in the stationary case can be obtained from the proof in the ergodic case with minor changes. 
\end{rmk}

\subsection{A generalization to the case of random shapes and anisotropic kernels}\label{ss:random shapes}
In this subsection we discuss a far reaching generalization of Theorem~\ref{thm: main result} to the case of randomly shaped obstacles and anisotropic, translation invariant, even symmetric, $-(n+sp)$-positive homogeneous, measurable kernels. 

To this end, let $\mathscr{K}$ be a measurable kernel satisfying \eqref{e:Kernels}, and let 
$\mathcal{K}$ be the corresponding functional as in \eqref{e:functionalK}.
Moreover, let $\PP$ be a m.p.p.\ on $\R^n$ with marks in $\R_+$ satisfying \ref{h: stationarity and ergodicity}-\ref{h: moment condition}. For every $\eps >0$ and $\omega \in \eventSimpleGround$, we define the new \textit{obstacle set} $S_\eps^\omega$ as 
\begin{equation*}
    S_\eps^\omega  := \bigcup_{i \in \insideDomainEpsOmega} S_\eps^{i,\omega},
\end{equation*}
where $S_\eps^i$ are random sets such that:
\begin{enumerate}[label=(H\arabic*), start=5]
    \item\label{h: geometric assumption on obstacle of random shape} $S_\eps^{i,\omega} \subseteq \obstacleEpsIOmega$;
    \item\label{h: stationarity of random capacities} there exists a sequence $\{\gamma_i\}_{i \in \N}$ of $\R_+$-valued random variables such that the marked point process $\sum_{i \in \N} \delta_{(x_i,\gamma_i)}$ is stationary and ergodic and the following scaling property holds 
    \begin{equation}
        \label{eq: capacity of random obstacle}
         \cspK(S_\eps^{i,\omega}) = \eps^n \gamma_i(\omega),
    \end{equation}
\end{enumerate}
for every $\eps >0$, $\omega \in \eventSimpleGround$, and $i \in \N$.

We observe that, in view of \eqref{e:comparable-kernels}, assumptions \ref{h: geometric assumption on obstacle of random shape} and \ref{h: stationarity of random capacities} in particular imply 
$0 \leq \gamma_i \leq  c \rho_i^{n-sp}$, for some $c>0$.
Therefore, in view of \ref{h: moment condition} we immediately deduce that
\begin{equation}\label{eq: moment condition for random capacities}
    \E\bigg[\sum_{x_i\in Q} \gamma_i\bigg]<+\infty.
\end{equation}
We now consider the functionals $\widetilde\functionalK_\eps^\omega: L^p(\domain) \longrightarrow [0, +\infty]$ defined as 
\begin{equation}\label{eq: functionals-e random shapes} 
    \widetilde\functionalK_\eps^\omega(u) := 
   \begin{cases}
        \displaystyle\mathcal{K}(u,\domain \times \domain) & \text{ if } u\in \Wsp(\domain),\,
        \tilde{u}=0\,\,  \csp\hbox{-}\text{q.e. on } S_\eps^\omega \cap \domain,\\\\
        +\infty & \text{ otherwise. }
  \end{cases}
\end{equation}
We emphasize that the homogeneity of $\mathscr{K}$ in \eqref{e:Kernels} is necessary to prove a $\Gamma$-convergence result which is subsequence independent, 
as observed already in the local setting (cf. \cite[Remark~2.7]{ansini-braides}).
The following almost sure $\Gamma$-convergence result holds true (cf. \cite[Theorem~4.1]{focardi-aperiodic-fractional}
for the case of rotation-invariant kernels).
\begin{thm}[$\Gamma$-convergence for randomly shaped obstacles]\label{thm: main result random shapes}
Let $\PP$ be a m.p.p.\ on $\R^n$ with marks in $\R_+$ satisfying \ref{h: stationarity and ergodicity}-\ref{h: moment condition}. In addition, assume that \ref{h: geometric assumption on obstacle of random shape} and \ref{h: stationarity of random capacities} hold with respect to some measurable kernel $\mathscr{K}$
satisfying \eqref{e:Kernels}. 

Let $\widetilde\functionalK_\eps^\omega$ be the functionals defined as in \eqref{eq: functionals-e random shapes}. Then there exists $\sampleSpace' \in \sigmaAlgebra$ with $\mathbb{P}(\sampleSpace')=1$ such that, for every $\omega \in \sampleSpace'$, $(\functionalK_\eps^\omega)$ $\Gamma$-converges in the  $L^p(\domain)$-topology, as $\eps \downarrow 0$, to the deterministic functional $\widetilde\functionalK: L^p(\domain) \longrightarrow [0,+\infty]$ given by 
    \begin{equation*}
        \widetilde\functionalK(u) := 
        \begin{cases}
            \displaystyle\mathcal{K}(u,\domain\times \domain)
            +\tilde \gamma \int_{\domain}|u|^p\dd{x} &\text{ if } u \in \Wsp(\domain), \\\\
            +\infty &\text{ otherwise, }
        \end{cases} 
    \end{equation*}
where 
\begin{equation*}
    \tilde\gamma := \E\bigg[\sum_{x_i\in Q} \gamma_i \bigg].
\end{equation*}
\end{thm}
\begin{rmk}
Arguing as in Remark~\ref{rmk:On the assumptions} we infer that
\begin{equation}\label{e:expectation gammai vs pi}
\tilde\gamma = \E\bigg[\sum_{x_i\in Q} \gamma_i \bigg]=
\IG \int_{\R_+} \gamma \, \dd{\hat\MD}(\gamma)
\end{equation}
for some probability measure $\hat \MD$ on $\R_+$, where $m_g$ is defined in \ref{h: positive intensity} 
(cf. \eqref{e:expectation rhoi vs pi}).     

    Moreover, as in Remark~\ref{r:Wald}, if we assume that the random variables $\gamma_i$ and $\IF_Q(x_i)$ are independent for every $i \in \N$
    and that $\E[\gamma_i] = \E[\gamma_1]$ for every $i \in \N$, then the Wald Identity yields
        $\tilde\gamma = \IG \E[\gamma_1]$.
Moreover, we observe that in view of Remark~\ref{rmk:periodic} the convergence result in Theorem~\ref{thm: main result random shapes} generalizes the analysis in \cite{caffarelli-mellet} to the case of \emph{randomly shaped clustering obstacles}.

Finally, as in Remark~\ref{thm: main result, nonergodic} we conclude observing that, if we drop the ergodicity assumption, then the $\Gamma$-limit $\widetilde\functionalK$ is a \emph{random functional} with capacitary constant given by
\[
\tilde k^\omega:= \E\bigg[\sum_{x_i\in Q} \gamma_i \, \big | \, \mathcal{I} \bigg](\omega).
\]
\end{rmk}

\section{Some technical auxiliary results}
In this section we prove a number of technical lemmas which will be employed to achieve the main result, Theorem~\ref{thm: main result}.
\subsection{Sets of indices}\label{subsec:set-of-indices}
For later purposes, we start by defining some suitable sets of indices together with the corresponding sets of obstacles. 
For $\eps >0$, $\omega \in \eventSimpleGround$, and $i \in \N$, we recall that 
\begin{equation*}
    \insideDomainEpsOmega:= \{i \in \N: x_i(\omega) \in \eps^{-1}\domain\}.
\end{equation*}
Moreover, for $\delta >0$ we define 
\begin{equation*}
    \insideDomainEps^{\delta,\omega} := I^{\delta,\omega} \cap \insideDomainEpsOmega,
\end{equation*}
where $I^\delta$ is as in Section \ref{sec: thinned process}.

To prove the $\Gamma$-$\liminf$ inequality, Proposition~\ref{prop: gamma-liminf}, it is useful to introduce the following set of indices
\begin{equation}\label{eq: definition of good indices}
    \goodEpsOmega := \left\{i \in \thinnedInsideDomainEpsOmega: \rho_i(\omega) \leq \truncParam, \, \dist(\eps x_i(\omega), \partial \domain)  > \frac{\eps}{\truncParam} \right\},
\end{equation}
for some truncation threshold $\truncParam \in \N$. 
The corresponding set of balls will be called the \textit{good} balls or obstacles. 
We observe that by construction the good balls $\{\ballContainingObstacleEpsIomega\}_{i \in \goodEpsOmega}$ are pairwise disjoint and contain the corresponding obstacles $\{\obstacleEpsI\}_{i \in \goodEpsOmega}$, whenever $\eps$ is sufficiently small (depending on $\truncParam$). 

On the other hand, to prove the $\Gamma$-$\limsup$ inequality, Proposition~\ref{prop: gamma-limsup}, it is convenient to introduce the subset of $\goodEpsOmega$ defined as
\begin{equation}\label{eq: definition of very good indices}
    \veryGoodEpsOmega := \{i \in \goodEpsOmega : \ballContainingObstacleEpsIomega \cap \ballDoubleCriticalRadiusEpsJomega = \emptyset \,\,  \forall j\in \insideDomainEps, \, j \neq i\}.
\end{equation}
In this case the corresponding set of balls will be called the \textit{very good} balls or obstacles.
Finally, we also consider the complement of $\veryGoodEpsOmega$ in $\insideDomainEpsOmega$, that is, the set defined as 
\begin{equation}\label{eq: definition of not very good indices}
    \notVeryGoodEpsOmega := \insideDomainEpsOmega \setminus \veryGoodEpsOmega.
\end{equation}
In analogy, the corresponding set of balls is called the \textit{not very good} balls or obstacles (see Figure~\ref{fig: very good and not very good obstacles}).
\begin{figure}
    \includegraphics[width=0.6\textwidth]{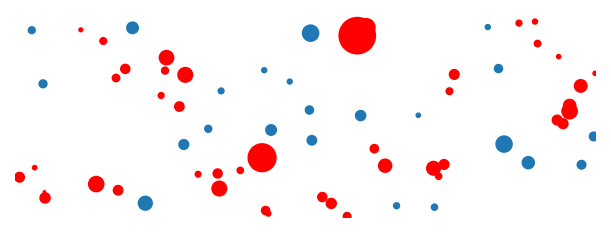}
    \caption{A realization of $\{\obstacleEpsI\}_{i\in\goodEps}$ (in blue) and $\{\obstacleEpsI\}_{i \in \notVeryGoodEps}$ (in red).}
    \label{fig: very good and not very good obstacles}
\end{figure}
By construction we have
\begin{equation}
    \label{eq: very good enlarged obstacles and not very good obstacles with double radius are disjoint}
    \bigg(\bigcup_{i \in \veryGoodEpsOmega} \ballContainingObstacleEpsIomega \bigg) \cap \bigg( \bigcup_{i \in \notVeryGoodEpsOmega} \ballDoubleCriticalRadiusEpsIomega \bigg)= \emptyset,
\end{equation}
that is, around the not very good balls there is a \emph{safety layer} which is still at an $\eps$ distance from the very good balls. 
Property \eqref{eq: very good enlarged obstacles and not very good obstacles with double radius are disjoint} will be pivotal in the construction of the recovery sequence in the proof of Proposition~\ref{prop: gamma-limsup}, combined with the fact that the \textit{not very good} obstacles do not contribute to the capacitary term, as stated in Lemma~\ref{lmm: NVG have vanishing relative capacity} below.


\subsection{Simple consequences of the ergodic theorems}\label{sec: consequences of ergodic theorems}
In this subsection we collect some simple but useful consequences of the results contained in Section \ref{sec: ergodic theorems}. These will be used several times throughout the paper.
\begin{lmm}
    \label{lmm: ergodic thm pp measure convergence for any thinned process N^2/R}
    There exists $\sampleSpace' \in \sigmaAlgebra$ with $\prob(\sampleSpace') =1$ such that for every $\omega \in \sampleSpace'$ and every $\truncParam \in \N$:
    \begin{enumerate}[label=(\alph*)]
    \item \label{enum: ergodic thm pp convergence of measures thinning} Proposition~\ref{prop: ergodic theorem ground pp convergence of measures} holds with $\PP$ and $\EM$ replaced, respectively, by $\PP^{2/\truncParam}$ and $\EM^{2/\truncParam}$;
    \item \label{enum: ergodic thm mpp convergence of measures thinning} Proposition~\ref{prop: ergodic theorem mpp convergence of measures} holds with $\PP$ and $\EM$ replaced, respectively, by $\PP^{2/\truncParam}$ and $\EM^{2/\truncParam}$.
    \end{enumerate}
\end{lmm}
\begin{proof}
    Let $\truncParam \in \N$ be fixed and consider the thinned process $\PP^{2/\truncParam}$. Let $\sampleSpace_\truncParam^1$ denote the event of probability one given by Proposition~\ref{prop: ergodic theorem ground pp convergence of measures} applied to $\PP^{2/\truncParam}$ and $\sampleSpace_\truncParam^2$ the event of probability one given by Proposition~\ref{prop: ergodic theorem mpp convergence of measures} applied to $\PP^{2/\truncParam}$.  We conclude by setting $\sampleSpace' := \cap_{\truncParam \in \N} (\sampleSpace_\truncParam^1 \cap  \sampleSpace_\truncParam^2)$. 
\end{proof}

\begin{lmm}
    \label{lmm: limit of eps^n sum rho_i^n-sp 1_(R,infty)(rho_i)}
    There exists $\sampleSpace' \in \sigmaAlgebra$ with $\prob(\sampleSpace')=1$ such that for every $\omega \in \sampleSpace'$ and every $\truncParam \in \N$
    \begin{equation}\label{e:limit of eps^n sum rho_i^n-sp 1_(R,infty)(rho_i)}
        \lim_{\eps \downarrow 0} \eps^n \sum_{i \in \insideDomainEpsOmega} \rho_i(\omega)^{n-sp} \IF_{(\truncParam, \infty)}(\rho_i(\omega)) = \IG \L^n(\domain) \int_{\truncParam}^\infty \rho^{n-sp} \, \dd{\MD}(\rho).
    \end{equation}
\end{lmm}
\begin{proof}
    Let $\truncParam \in \N$ be fixed and set $h_\truncParam(\rho):= \rho^{n-sp}\IF_{(\truncParam, \infty)}(\rho)$ for $\rho \in \R_+$. In view of \eqref{eq: integrability assumption on radii} we have $h_\truncParam \in L^1(\R_+,\MD)$. Therefore, by Proposition~\ref{prop: ergodic theorem, integral on eps^-1 U of h(rho)}, there exists $\Omega_\truncParam \in \sigmaAlgebra$ with $\prob(\Omega_\truncParam)=1$ such that \eqref{e:limit of eps^n sum rho_i^n-sp 1_(R,infty)(rho_i)} holds for every $\omega \in \Omega_\truncParam$. We conclude by defining $\Omega':= \cap_{\truncParam \in \N} \sampleSpace_\truncParam$. 
\end{proof}
\begin{lmm}
    \label{lmm: eps^n card(F_eps,R) -> 0 implies eps^n sum rho_i^n-sp -> 0}
    There exists $\sampleSpace' \in \sigmaAlgebra$ with $\prob(\sampleSpace')=1$ such that for every $\omega \in \Omega'$
    \begin{equation*}
        \lim_{\truncParam\uparrow\infty}\limsup_{\eps \downarrow 0} \eps^n \sum_{i \in F_{\eps,\truncParam}^\omega} \rho_i(\omega)^{n-sp} = 0,
    \end{equation*}
    whenever $F_{\eps,\truncParam}^\omega \subset \insideDomainEpsOmega$ for every $\eps>0$, $\truncParam \in \N$ and 
    \begin{equation}\label{e:card Feps}
    \lim_{\truncParam\uparrow\infty}\lim_{\eps \downarrow 0} \eps^n \# F_{\eps,\truncParam}^\omega =0.
    \end{equation}
\end{lmm}
\begin{proof}
    Let $\Omega' \in \sigmaAlgebra$ be given by Lemma~\ref{lmm: limit of eps^n sum rho_i^n-sp 1_(R,infty)(rho_i)}. 
    Then, for every $\omega \in \sampleSpace'$ and every $L \in \N$, in view of \eqref{e:card Feps} we have
    \begin{equation*}
        \begin{split}
            &\limsup_{\truncParam\uparrow\infty}\limsup_{\eps \downarrow 0} \eps^n \sum_{i \in F^\omega_{\eps,\truncParam}} \rho_i(\omega)^{n-sp} \\
            &\leq \limsup_{\truncParam\uparrow\infty}\limsup_{\eps \downarrow 0} \eps^n \sum_{i \in F^\omega_{\eps,\truncParam}} \rho_i(\omega)^{n-sp} \IF_{[0,L]}(\rho_i(\omega)) +\lim_{\eps \downarrow 0} \eps^n \sum_{i \in \insideDomainEps^\omega} \rho_i(\omega)^{n-sp} \IF_{(L,\infty)}(\rho_i(\omega)) \\
            &\leq L^{n-sp}  \limsup_{\truncParam\uparrow\infty}\lim_{\eps \downarrow 0} \eps^n \# F^\omega_{\eps,\truncParam} + \IG \L^n(\domain) \int_L^\infty \rho^{n-sp} \, \dd{\MD}(\rho) \\
            &\leq \IG \L^n(\domain) \int_L^\infty \rho^{n-sp} \, \dd{\MD}(\rho). 
        \end{split}
    \end{equation*}
    Then, the claim follows by the arbitrariness of $L$ and by the Dominated Convergence Theorem.
\end{proof}
\subsection{Capacity of the not very good obstacles}
In this subsection we show that the capacity of the union of the \textit{not very good} obstacles is infinitesimal. To this end, we preliminarily need to prove a number of lemmas. 

\begin{lmm}
    \label{lmm:lq norm}
    There exists $\sampleSpace' \in \sigmaAlgebra$ with $\prob(\sampleSpace') = 1$ such that for every $\omega\in\sampleSpace'$ 
    \begin{equation}\label{e:lq norm vanishing}
        \lim_{\eps \downarrow 0}\|\lambda_\eps  \rho_i(\omega)\|_{\ell^q(\insideDomainEpsOmega)} = 0\,,
    \end{equation}
    for every $q\in(n-sp,\infty]$.
\end{lmm}
\begin{proof}
    \intersectionOfEventsGivenBy Lemma~\ref{lmm: limit of eps^n sum rho_i^n-sp 1_(R,infty)(rho_i)} and Lemma~\ref{lmm: eps^n card(F_eps,R) -> 0 implies eps^n sum rho_i^n-sp -> 0}, respectively, and let  $\omega \in \sampleSpace'$.
    Set $F_\eps := \{i \in \insideDomainEps: \lambda_\eps \rho_i > \eps\}$. 
    Recalling that $\lambda_\eps^{n-sp}=\eps^n$, by the very definition of $F_\eps$ and by Lemma~\ref{lmm: limit of eps^n sum rho_i^n-sp 1_(R,infty)(rho_i)} we get
    \begin{equation}\label{e:card-F}
        \limsup_{\eps \downarrow 0} \eps^n \# F_\eps^\omega 
        \leq \limsup_{\eps \downarrow 0}\eps^{sp} \sum_{i \in F_\eps^\omega} (\lambda_\eps\rho_i(\omega))^{n-sp}
        \leq \lim_{\eps \downarrow 0}\eps^{sp} \eps^{n}\sum_{i \in \insideDomainEps^\omega}  \rho_i(\omega)^{n-sp} = 0\,.
    \end{equation}
     For every $q\in(n-sp,\infty)$, writing $\insideDomainEps=(\insideDomainEps\setminus F_{\eps})\cup F_{\eps}$, in view of the continuous embedding $\ell^{n-sp}(\insideDomainEps)\hookrightarrow\ell^{q}(\insideDomainEps)$ (with embedding constant equal to one)
    we obtain
\begin{align}\label{e:lq norm estimate}
\sum_{i \in \insideDomainEps} (\lambda_\eps \rho_i)^q &= 
\sum_{i \in \insideDomainEps\setminus F_\eps} (\lambda_\eps \rho_i)^q +
\sum_{i \in F_\eps} (\lambda_\eps \rho_i)^q\notag\\ 
&\leq\eps^{q-sp}\sum_{i \in \insideDomainEps\setminus F_\eps} 
\rho_i^{n-sp}+\Big(\eps^n\sum_{i \in F_\eps} \rho_i^{n-sp}\Big)^{\frac{q}{n-sp}}\,,
\end{align}
where in the last line we used again that $\lambda_\eps^{n-sp}=\eps^n$.
Then \eqref{e:lq norm vanishing} follows from \eqref{e:lq norm estimate}, 
Lemma~\ref{lmm: limit of eps^n sum rho_i^n-sp 1_(R,infty)(rho_i)}, and
Lemma~\ref{lmm: eps^n card(F_eps,R) -> 0 implies eps^n sum rho_i^n-sp -> 0} also invoking \eqref{e:card-F}. 

Finally, by taking the $q$-th root in \eqref{e:lq norm estimate}, using the Minkowski Inequality, and letting $q\to\infty$, it is easy to infer that
\[
\lambda_{\eps}\max_{i \in \insideDomainEps} \rho_i\leq\eps+
\Big(\eps^n \sum_{i \in F_\eps} \rho_i^{n-sp}\Big)^{\frac{1}{n-sp}}\,,
\]
hence Lemma~\ref{lmm: eps^n card(F_eps,R) -> 0 implies eps^n sum rho_i^n-sp -> 0}
yields the conclusion for $q=\infty$, as well.
\end{proof}

The next lemma provides us with some useful information concerning the cardinality of the of indices introduced in Subsection \ref{subsec:set-of-indices} which are instrumental for the sequel. 
\begin{lmm}\label{lmm: eps^n card(inside minus thinnedInside) -> 0}
    There exists $\sampleSpace' \in \sigmaAlgebra$ with $\prob(\sampleSpace') =1$ such that for every $\omega \in \sampleSpace'$:
    \begin{enumerate}[label=(\alph*)]
    \item \label{enum: card sets insideDomainEps minus thinnedInsideDomainEps} $\lim_{\truncParam\uparrow\infty}\lim_{\eps\downarrow0} \eps^n\#(\insideDomainEpsOmega \setminus \thinnedInsideDomainEpsOmega) = 0$;
    
    \smallskip
    
    \item \label{enum: card sets thinnedInsideDomainEps minus goodEps}$\lim_{\truncParam\uparrow\infty}\limsup_{\eps\downarrow0} \eps^n\#(\thinnedInsideDomainEpsOmega\setminus\goodEpsOmega) = 0$;
    
    \smallskip
    
    \item \label{enum: card sets goodEps minus veryGoodEps}$\lim_{\truncParam\uparrow\infty}\lim_{\eps \downarrow 0}\eps^n \# (\goodEps^\omega \setminus \veryGoodEps^\omega) = 0$;
    
    \smallskip 
    
    \item \label{enum: card sets notVeryGoodEps}$\lim_{\truncParam\uparrow\infty}\lim_{\eps \downarrow 0}\eps^n \#\notVeryGoodEps^\omega =0$.
    \end{enumerate}    
\end{lmm}
\begin{proof}
    Let $\sampleSpace' \in \sigmaAlgebra$ be the intersection of the events of probability one given by Lemmas \ref{lmm: ergodic thm pp measure convergence for any thinned process N^2/R}-\ref{lmm:lq norm} and let $\omega \in \Omega'$. 
    
    We start by proving statement \ref{enum: card sets insideDomainEps minus thinnedInsideDomainEps}. By the 
    weak* convergence in the sense of measures established in Lemma~\ref{lmm: ergodic thm pp measure convergence for any thinned process N^2/R} and by the Lipschitz regularity of $\partial U$, \cite[Proposition~1.62(b)]{AFP} implies
    \begin{equation*}
        \begin{split}
            &\lim_{\truncParam \uparrow \infty} \lim_{\eps \downarrow 0} \eps^n \# (\insideDomainEps^\omega \setminus \thinnedInsideDomainEpsOmega) = \lim_{\truncParam \uparrow \infty} \lim_{\eps \downarrow 0} \eps^n (\PP^\omega_{g,\eps}(\eps^{-1}\domain) - \PP_{g,\eps}^{2/\truncParam,\omega}(\eps^{-1}\domain)) \\
            &= \lim_{\truncParam \uparrow \infty} (\EM^\omega_g(\domain) - \EM_g^{2/\truncParam,\omega}(\domain)) = 0,
        \end{split}
    \end{equation*}
    hence \ref{enum: card sets insideDomainEps minus thinnedInsideDomainEps} is proven.
    
    We now turn to the proof of \ref{enum: card sets thinnedInsideDomainEps minus goodEps}. To this end, we observe that
    \begin{equation}\label{e:dec_index_set}
            \thinnedInsideDomainEps\setminus \goodEps \subset \{i \in \insideDomainEps : \rho_i > \truncParam\} \cup \left\{i \in \thinnedInsideDomainEps: \dist(\eps x_i, \partial\domain) > \frac{\eps}{\truncParam}\right\}.
        \end{equation}
        Moreover, appealing to Lemma~\ref{lmm: limit of eps^n sum rho_i^n-sp 1_(R,infty)(rho_i)} we have
        \begin{equation}\label{e:card-first-set}
            \begin{split}
                &\lim_{\truncParam\uparrow\infty}\limsup_{\eps\downarrow0}\eps^n \#\{i \in \insideDomainEps^\omega : \rho_i(\omega) > \truncParam\} \leq \lim_{\truncParam\uparrow\infty}\lim_{\eps\downarrow0}\eps^n\sum_{i \in \insideDomainEps^\omega} \rho_i(\omega)^{n-sp} \IF_{(\truncParam,\infty)}(\rho_i(\omega)) \\
                &= \IG\L^n(\domain)\lim_{\truncParam\uparrow\infty}\int_{\truncParam}^\infty \rho^{n-sp} \, \dd{\MD}(\rho) = 0.
            \end{split}
        \end{equation}
        Finally, since $\{\ballContainingObstacleEpsI : i \in \thinnedInsideDomainEps, \, \dist(\eps x_i, \partial \domain) > \frac{\eps}{\truncParam}\}$ are pairwise disjoint and their union is contained in $(\partial \domain)_{2 \frac{\eps}{\truncParam}}$, we get
        \begin{equation}\label{e:card-second-set}
            \lim_{\eps\downarrow0}\eps^n\#\left\{i \in \thinnedInsideDomainEps: \dist(\eps x_i, \partial\domain) > \frac{\eps}{\truncParam}\right\}\lesssim_n \truncParam^n \lim_{\eps\downarrow0}\L^n((\partial\domain)_{2 \frac{\eps}{\truncParam}}) = 0.
    \end{equation}
    Therefore, the statement in \ref{enum: card sets thinnedInsideDomainEps minus goodEps} follows by gathering \eqref{e:dec_index_set}-\eqref{e:card-second-set}.

    To prove \ref{enum: card sets goodEps minus veryGoodEps} we start by observing that 
    \begin{equation}\label{e:good-without-verygood-indices} 
        \goodEps \setminus \veryGoodEps = \{i \in \goodEps: \exists j \in \insideDomainEps, \, j \neq i, \, \ballContainingObstacleEpsI \cap \ballDoubleCriticalRadiusEpsJ \neq \emptyset\}.
    \end{equation}
    By the very definition of $\goodEps$, we notice that if $i\in\goodEps \setminus \veryGoodEps$ and $j$ is as in \eqref{e:good-without-verygood-indices}, then necessarily $j\in(\insideDomainEps\setminus\thinnedInsideDomainEps)\cup F_{\eps,\truncParam}$, where
    \[
    F_{\eps,\truncParam} := \left\{i \in \thinnedInsideDomainEps : 2\lambda_\eps \rho_i > \frac{\eps}{\truncParam}\right\}\,.
    \]
    Define the set of indices 
    \begin{align*}
        &E_{\eps,\truncParam}^1 := \{i \in \goodEps : \exists j \in F_{\eps,\truncParam}, \, j \neq i, \, \ballContainingObstacleEpsI \cap \ballDoubleCriticalRadiusEpsJ \neq \emptyset\}, \\
        &E_{\eps,\truncParam}^2 := \{i \in \goodEps: \exists j \in \insideDomainEps\setminus \thinnedInsideDomainEps,\, j\neq i, \, \ballContainingObstacleEpsI \cap \ballDoubleCriticalRadiusEpsJ \neq \emptyset\},
    \end{align*}
    so that
    \[
    \goodEps \setminus \veryGoodEps = E^1_{\eps,\truncParam} \cup E^2_{\eps,\truncParam}.
    \] 
    Hence, to prove \ref{enum: card sets goodEps minus veryGoodEps} it is enough to show that almost surely
    \[
    \lim_{R\uparrow +\infty}\lim_{\eps\downarrow0}\eps^n \# E_{\eps,\truncParam}^{k}, \quad \text{for }\; k=1,2.
    \]
    Since $(2\truncParam)^{n-sp} \rho_i^{n-sp} \eps^{sp} > 1$ for any $i \in F_{\eps,\truncParam}$, Lemma~\ref{lmm: limit of eps^n sum rho_i^n-sp 1_(R,infty)(rho_i)} yields
    \begin{equation}\label{e:FepsR}
        \lim_{\eps \downarrow 0}\eps^n \# F^\omega_{\eps,\truncParam} \lesssim \truncParam^{n-sp} \lim_{\eps \downarrow 0}\eps^{sp} \eps^n \sum_{i \in \insideDomainEps^\omega} \rho_i(\omega)^{n-sp} = 0.
    \end{equation}
    Moreover, since $\{\ballContainingObstacleEpsI\}_{i \in \goodEps}$ are pairwise disjoint we have
    \[
    \eps^n \# E_{\eps,\truncParam}^k \lesssim \truncParam^n \L^n (\cup_{i \in E_{\eps,\truncParam}^k} \ballContainingObstacleEpsI)\quad
    \text{for }\; k=1,2.
    \]
    We also observe that for every $i \in E_{\eps,\truncParam}^1$ there exists $j \in F_{\eps,\truncParam}$, $j\neq i$, such that $\ballContainingObstacleEpsI \cap \ballDoubleCriticalRadiusEpsJ \neq \emptyset$, and hence $\ballContainingObstacleEpsI \subset B_{2(\lambda_\eps \rho_j + \frac{\eps}{\truncParam})}(\eps x_j)$. Therefore, appealing to Lemma~\ref{lmm:lq norm} with $q=n$ and to \eqref{e:FepsR}
    \begin{equation*}
        \begin{split}
            &\lim_{\eps\downarrow0}\eps^n \# E_{\eps,\truncParam}^{1,\omega} \lesssim \truncParam^n \lim_{\eps \downarrow0}\sum_{i \in \insideDomainEps^\omega} (\lambda_\eps \rho_i(\omega))^n + \lim_{\eps \downarrow 0}\eps^n \# F^\omega_{\eps,\truncParam} = 0.
        \end{split}
    \end{equation*}
    Similarly, again invoking Lemma~\ref{lmm:lq norm} with $q=n$, in view of Lemma~\ref{lmm: eps^n card(inside minus thinnedInside) -> 0} we get
    \begin{equation*}
        \begin{split}
            &\lim_{\truncParam\uparrow\infty}\lim_{\eps \downarrow 0}\eps^n \#E_{\eps,\truncParam}^{2,\omega} \lesssim \lim_{\truncParam\uparrow\infty} \lim_{\eps \downarrow 0} \truncParam^n  \sum_{i \in \insideDomainEps^\omega} (\lambda_\eps \rho_i(\omega))^n + \lim_{\truncParam\uparrow\infty}\lim_{\eps \downarrow 0}\eps^n \#(\insideDomainEps^\omega\setminus I_\eps^{2/R,\omega}) =0,
        \end{split}
    \end{equation*}
    thus \ref{enum: card sets goodEps minus veryGoodEps} is achieved.
    
    Finally, we notice that the convergence statement in \ref{enum: card sets notVeryGoodEps} readily follows from \ref{enum: card sets insideDomainEps minus thinnedInsideDomainEps}-\ref{enum: card sets goodEps minus veryGoodEps}, observing that
    
    \begin{equation*}
        \notVeryGoodEps = (\insideDomainEps \setminus \thinnedInsideDomainEps) \cup (\thinnedInsideDomainEps \setminus \goodEps) \cup (\goodEps \setminus \veryGoodEps).
    \end{equation*}
\end{proof}

Thanks to the subadditivity of the fractional capacity, on account of Lemma~\ref{lmm: eps^n card(inside minus thinnedInside) -> 0} we are now able to prove that the relative capacity of the union of the \emph{not very good} obstacles with respect to a \emph{safety layer} (made by the union of balls with the same center and double radius) vanishes \emph{almost surely} when $\eps \to 0$ and $R\to +\infty$.    
\begin{lmm}
    \label{lmm: NVG have vanishing relative capacity}
    \sureEvent  
    \begin{equation*}
        \lim_{\truncParam \uparrow \infty} \limsup_{\eps \downarrow 0} C_{s,p}\left(\bigcup_{i \in \notVeryGoodEpsOmega }  \ballCriticalRadiusEpsIomega , \bigcup_{i \in \notVeryGoodEpsOmega }  \ballDoubleCriticalRadiusEpsIomega \right) = 0.
    \end{equation*}
\end{lmm}
\begin{proof}
    Let $\sampleSpace'\in \sigmaAlgebra$ be the intersection of the events of probability one given by Lemma~\ref{lmm: eps^n card(F_eps,R) -> 0 implies eps^n sum rho_i^n-sp -> 0} and Lemma~\ref{lmm: eps^n card(inside minus thinnedInside) -> 0} and let $\omega \in \sampleSpace'$. We can readily deduce that  
   \begin{equation}\label{e:NVG have vanishing capacity}
        \lim_{\truncParam \uparrow \infty} \limsup_{\eps \downarrow 0} \eps^n \sum_{i \in \notVeryGoodEpsOmega} \rho_i(\omega)^{n-sp} = 0.
    \end{equation}
    By the subadditivity and the scaling property of the fractional capacity (cf.\ \cite[Theorem~2.4 (a), (e)]{shi-xiao-relative-fractional-capacities}) we get
    \begin{equation*}
        \begin{split}
            &\lim_{\truncParam \uparrow \infty} \limsup_{\eps \downarrow 0} C_{s,p}\left(\bigcup_{i \in \notVeryGoodEpsOmega  }  \ballCriticalRadiusEpsIomega , \bigcup_{i \in \notVeryGoodEpsOmega  }  B_{2 \lambda_\eps \rho_i(\omega)}(\eps x_i(\omega)) \right) \\
            &\leq \lim_{\truncParam \uparrow \infty} \limsup_{\eps \downarrow 0} \sum_{i \in \notVeryGoodEpsOmega } C_{s,p}\left(\ballCriticalRadiusEpsIomega, \ballDoubleCriticalRadiusEpsIomega\right) \\
            &= C_{s,p}(B_1, B_2) \lim_{\truncParam \uparrow \infty} \limsup_{\eps \downarrow 0} \eps^n \sum_{i \in \notVeryGoodEpsOmega }  \rho_i(\omega)^{n-sp} = 0,
        \end{split}
    \end{equation*}
    where to establish the very last inequality we used \eqref{e:NVG have vanishing capacity}.
\end{proof}
\subsection{A Joining Lemma for nonlocal functionals}
For the readers' convenience in this subsection we recall the statement of a joining lemma for \emph{nonlocal} functionals originally proven in \cite{focardi-aperiodic-fractional}. This result is an adaptation to the nonlocal setting of the Ansini and Braides Joining Lemma for \emph{local} nonlinear energy functionals  \cite{ansini-braides}. In our case, where clustering obstacles can occur with positive probability, the Joining Lemma will be applied
only in a neighborhood of the so-called good obstacles. 

\medskip

The following result is preliminary to the Joining Lemma. 

\begin{lmm}
    \label{lmm: unscaled joining lemma}
Suppose that $G$ is a finite subset of $\N$ and $\{x_i\}_{i \in G}$ is a collection of points such that, for some $r_G > 0$, $\frac{1}{2} \min_{l\neq i \in G} |x_l - x_i| \geq r_G$ and $\cup_{i \in G} B_{r_G}(x_i) \subset \domain$. Let $\rho \in (0, \frac{1}{2}r_G)$. For $i \in G$, let $A_i' = x_i + B_{m^{-1}\rho} \setminus \overline{B}_{m^{-2} \rho}$ and $A_i = x_i + B_{\rho} \setminus \overline{B}_{m^{-3}\rho}$. Let $\varphi$ be a cutoff function between $B_{m^{-1}\rho} \setminus \overline{B}_{m^{-2} \rho}$ and $B_{\rho} \setminus \overline{B}_{m^{-3}\rho}$ and $\varphi_i(\cdot) = \varphi(\cdot - x_i)$. 
Then there exists $c = c(n,p,s) > 0$ such that for any $u \in \Wsp(\domain)$ and $\{z_i\}_{i \in G} \subset \R$, the function $w: \domain \to \R^n$ defined by 
\begin{equation*}
    w(x) := \sum_{i \in G} \varphi_i(x) z_i + \left(1-\sum_{i \in G} \varphi_i(x)\right) u(x) \quad \forall x \in \domain
\end{equation*} 
is such that $w \in \Wsp(\domain)$, $w=u$ on $\domain \setminus \overline{A}$, where $A:= \cup_{i \in G} A_i$, and, for any $\L^n$-measurable set $E \subset \domain\times \domain$, 
\begin{equation*}
    |\Dsp(w,E) - \Dsp(u,E)| 
    \leq c \left( \Dsp(u,\domain\times A) + m^{2p} \rho^{-sp} \sum_{i \in G} \int_{A_i} |u-z_i|^p \right).
\end{equation*}
\end{lmm}
\begin{proof}
    The proof follows by arguing as in \cite[Proof of Lemma~3.8]{focardi-aperiodic-fractional} up to replacing the Delone set $\Lambda$ with $\{x_i\}_{i \in G}$ and $\mathcal{I}_\Lambda(\domain)$ with $G$. In fact in \cite[Proof of Lemma~3.8]{focardi-aperiodic-fractional} the property of  $\Lambda$ which is used is its discreteness (cf. \cite[Definition 2.2(i)]{focardi-aperiodic-fractional}) together with the fact that $\cup_{i \in \mathcal{I}_\Lambda(\domain)} A_i \subset \domain$ and these two properties are satisfied, by assumption, also in our case.  
\end{proof}

\begin{rmk}
Let $\omega \in \eventSimpleGround$ and let $\truncParam \in \N$. We notice that \eqref{eq: definition of good indices} and \eqref{eq: definition of very good indices} respectively imply that the sets $\goodEpsOmega$ and $\veryGoodEpsOmega$ (and correspondingly $\{x_i(\omega)\}_{i \in \goodEpsOmega}$ and $\{x_i(\omega)\}_{i \in \veryGoodEpsOmega}$) satisfy the assumptions of Lemma~\ref{lmm: unscaled joining lemma}. 
\end{rmk}

Let $G\subset \N$ and $\{x_i\}_{i \in G}$ be as in Lemma~\ref{lmm: unscaled joining lemma}.
For $i \in G$, $m \in \N$, and $h \in \N$ define (see Figure~\ref{fig: joining lemma}):
\begin{equation}
    \label{eq: definition of ball containing annulus}
    B_\eps^{i,h} = B_{m^{-3h} \eps \truncParam^{-1}}(\eps x_i), 
\end{equation}
\begin{equation}
    \label{eq: definition of annulus}
    C_\eps^{i,h} = B_{m^{-3h-1}\eps \truncParam^{-1}}(\eps x_i) \setminus \overline{B}_{m^{-3h-2}\eps \truncParam^{-1}}(\eps x_i),
\end{equation}
\begin{equation}
    \label{eq: definition of ball bounding annulus}
    \tilde{B}_\eps^{i,h} = B_{m^{-3h-1}\eps \truncParam^{-1}}(\eps x_i).
\end{equation}
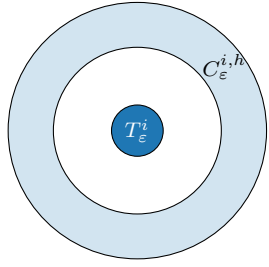
\begin{figure}
    \begin{tikzpicture}[scale=0.85]
        \coordinate (A) at (0,0); 
        \coordinate (B) at (2,0);
        \coordinate (C) at (1,0.3);
        \coordinate (D) at (0,0.8);
        \coordinate (E) at (1.35,1);
        \coordinate (F) at (4,0);
        \draw[fill=MatplotlibC0!20] (A) circle (2);
        \draw[fill=white] (A) circle (1.3);
        \draw[fill=MatplotlibC0,] (A) circle (0.4); 
        \node at (A) {\color{white}\footnotesize$\obstacleEpsI$};
        \node at (E) {\footnotesize$C_{\eps}^{i,h}$};
    \end{tikzpicture}
    \caption{The annulus $C_{\eps}^{i,h}$ surrounding a good obstacle $\obstacleEpsI$.}
    \label{fig: joining lemma}
\end{figure}  
\begin{lmm} 
    \label{lmm: joining lemma}
    Fix $m,\joiningLemmaEnergyDiffParam,\truncParam \in \N$ and a sequence $\eps_j \downarrow 0$. Suppose that $u_j \to u$ in $L^p(\domain)$ and $\sup_{j \in \N} |u_j|_{\Wsp(\domain)} < \infty$. For every $j \in \N$, there exists $h_j \in \{1,\ldots,\joiningLemmaEnergyDiffParam\}$ and a function $w_j \in \Wsp(\domain)$ such that 
    \begin{equation*}
        w_j = u_j \quad \text{on } \domain \setminus \bigcup_{i \in G} ( \ballContainingAnnulusJI \setminus B_j^{i, h_j+1}), 
    \end{equation*}
    \begin{equation*}
        w_j = (u_j)_{\annulusJI} \quad \text{on } \annulusJI, 
    \end{equation*}
    and, for some $c = c(n,p,s,m,\truncParam)>0$, for every $\L^n$-measurable set $E \subset \domain \times \domain$, 
    \begin{equation}\label{e:en-estimate-joining-lemma}
        |\Dsp(u_j,E) - \Dsp(w_j,E)| \leq \frac{c}{\joiningLemmaEnergyDiffParam} |u_j|^p_{\Wsp(\domain)}.
    \end{equation}
    Moreover, $w_j \to u$ in $L^p(\domain)$ and, if $u_j \in L^\infty(\domain)$, then $\norm{w_j}_{L^\infty(\domain)} \leq \norm{u_j}_{L^\infty(\domain)}$.
\end{lmm}
\begin{proof}
    The result follows from an application of Lemma~\ref{lmm: unscaled joining lemma} and Poincaré-Wirtinger Inequality in fractional Sobolev spaces, replicating the argument contained in \cite[Proof of Lemma~3.9]{focardi-aperiodic-fractional}.
\end{proof}

\subsection{Capacitary term as the limit of random sums}\label{sec: capacity term as limit of random sums}
The results proven in this subsection will allow us to reconstruct the capacitary term in the limit functional $\functional$ as the limits of some suitable \textit{random sums}. 

We start proving the following auxiliary lemma. 

\begin{lmm}
    \label{lmm: approximation of integral}
    \sureEventWithFollowingProperty For every $v \in C_c^0(\domain)$ and every $\omega \in \sampleSpace'$
    \begin{equation}\label{e: approximation of integral}
        \lim_{\truncParam \uparrow \infty}
         \lim_{\eps \downarrow 0} \eps^n \sum_{i \in \thinnedInsideDomainEpsOmega} |v(\eps x_i(\omega))|^p  = \IG \norm{v}_{L^p(\domain)}^p 
    \end{equation}
and 
    \begin{equation}\label{e: approximation of integral with radii}
        \lim_{\truncParam \uparrow \infty}
         \lim_{\eps \downarrow 0} \eps^n \sum_{i \in \thinnedInsideDomainEpsOmega} |v(\eps x_i(\omega))|^p \rho_i(\omega)^{n-sp} = \IG \norm{v}_{L^p(\domain)}^p \int_{\R_+} \rho^{n-sp} \, \dd{\MD}(\rho).
    \end{equation}
\end{lmm}
\begin{proof}
    Let $\Omega' \in \sigmaAlgebra$ be given by Lemma~\ref{lmm: ergodic thm pp measure convergence for any thinned process N^2/R} and let $\omega \in \sampleSpace'$. In view of Lemma~\ref{lmm: ergodic thm pp measure convergence for any thinned process N^2/R}-\ref{enum: ergodic thm pp convergence of measures thinning}
     for every $\truncParam\in \mathbb N$ we have
    \begin{equation*}
       \lim_{\eps \downarrow0} \eps^n \sum_{i \in \thinnedInsideDomainEpsOmega} |v(\eps x_i(\omega))|^p = \IG \norm{v}_{L^p(\domain)}^p \int_{\R_+} \probLawPP_{(0,\rho)}(Y_\frac{2}{\truncParam}) \, \dd{\MD}(\rho), 
    \end{equation*}
    while Lemma~\ref{lmm: ergodic thm pp measure convergence for any thinned process N^2/R}-\ref{enum: ergodic thm mpp convergence of measures thinning} applied with $h(\rho) = \rho^{n-sp}$
    yields that for every $\truncParam\in \N$ 
    \begin{equation*}
        \lim_{\eps \downarrow 0} \eps^n \sum_{i \in \thinnedInsideDomainEpsOmega} |v(\eps x_i(\omega))|^p \rho_i(\omega)^{n-sp} = \IG \norm{v}_{L^p(\domain)}^p\int_{\R^+} \rho^{n-sp} \, \probLawPP_{(0,\rho)}(Y_\frac{2}{\truncParam}) \, \dd{\MD}(\rho).
    \end{equation*}
    Therefore, both \eqref{e: approximation of integral} and \eqref{e: approximation of integral with radii} follow by letting $\truncParam \to +\infty$ and invoking the Dominated Convergence Theorem.
\end{proof}

The following proposition will be crucial in the proof of the $\Gamma$-convergence result, Theorem~\ref{thm: main result}.
\begin{prop} 
    \label{lmm: integral approximation G_j,R}
    \sureEventWithFollowingProperty
    Let $\eps_j \downarrow 0$ and $u_j \to u$ in $L^p(\domain)$ as $j\to\infty$. Let $\truncParam, \joiningLemmaEnergyDiffParam, m \in \N$ and $h_j : \sampleSpace \to \{1, \ldots, \joiningLemmaEnergyDiffParam\}$. Then, 
    for every $\omega \in \sampleSpace'$ we have 
    \begin{equation}\label{e: integral approximation G_j,R}
        \lim_{\truncParam \to\infty}\sup_{L \in \mathbb N}
        \limsup_{j\to\infty}\Big|
        \eps_j^n \sum_{i \in \goodJOmega} |(u_j)_{\annulusJIomega}|^p -\IG \norm{u}_{L^p(\domain)}^p\Big|=0,
    \end{equation}
    and
    \begin{equation}\label{e: integral approximation G_j,R with radii}
        \lim_{\truncParam \to\infty}\sup_{L\in \mathbb N}
        \limsup_{j\to\infty}\Big|\eps_j^n \sum_{i \in \goodJOmega} |(u_j)_{\annulusJIomega}|^p \rho_i(\omega)^{n-sp} - \IG \norm{u}_{L^p(\domain)}^p \int_{\R_+} \rho^{n-sp} \, \dd{\MD}(\rho)\Big|=0.
    \end{equation}
    \end{prop}
    
\begin{proof}
    \intersectionOfEventsGivenBy Lemmas \ref{lmm: eps^n card(F_eps,R) -> 0 implies eps^n sum rho_i^n-sp -> 0}, \ref{lmm: eps^n card(inside minus thinnedInside) -> 0}, and \ref{lmm: approximation of integral}, and $\omega \in \sampleSpace'$. To simplify the notation, in what follows we are going to omit the dependence on $\omega$. 
     
     We start proving the convergence in \eqref{e: integral approximation G_j,R}. 
     
     Recalling \eqref{eq: definition of ball containing annulus}-\eqref{eq: definition of ball bounding annulus} and appealing to Jensen's Inequality and Poincaré-Wirtinger Inequality \cite[(2.10) in Remark~2.7]{focardi-aperiodic-fractional} we get
    \begin{equation*}
        \begin{split}
            &\eps_j^n \sum_{i \in \goodJ} |(u_j)_{\annulusJI} - (u_j)_{\ballBoundingAnnulusJI}|^p  \leq \eps_j^n \sum_{i \in \goodJ} \fint_{\ballBoundingAnnulusJI} |u_j - (u_j)_{\annulusJI}|^p  \leq c \, \eps_j^{sp} \sup_{j \in \N} |u_j|^p_{\Wsp(\domain)}  , 
        \end{split}
    \end{equation*}
    for some constant $c = c(n,s,p,m)>0$.
    Similarly, we have
    \begin{equation*}
        \begin{split}
            \eps_j^n \sum_{i \in \goodJ} |(u_j)_{\ballBoundingAnnulusJI} - (u)_{\ballBoundingAnnulusJI}|^p  \leq \eps_j^n \sum_{i \in \goodJ} \fint_{\ballBoundingAnnulusJI} |u_j-u|^p \lesssim_{n,m, \truncParam} \norm{u_j-u}_{L^p(\domain)}^p
        \end{split}
    \end{equation*}    
    and therefore 
    \begin{equation*}
        \lim_{j\to \infty}  \eps_j^n \sum_{i \in \goodJ} |(u_j)_{\annulusJI} - (u)_{\ballBoundingAnnulusJI} |^p = 0.
    \end{equation*}
   Let $\{v_k\}_{k \in \N} \subset C^\infty_c(\domain)$ be such such that $v_k \to u$ in $L^p(\domain)$ as $k \to \infty$; arguing as above it is immediate to deduce that 
    \begin{equation*}
       \lim_{k \to \infty}  \eps_j^n \sum_{i \in \goodJ} |(u)_{\ballBoundingAnnulusJI} - (v_k)_{\ballBoundingAnnulusJI}|^p \lesssim_{n,m,\truncParam} \lim_{k \to \infty}  \norm{u-v_k}_{L^p(\domain)}^p = 0.
    \end{equation*}
    Hence, without loss of generality, we can assume that $u \in C^\infty_c(\domain)$. Let $\eta>0$ be arbitrary and fixed; since $u$ is uniformly continuous, for sufficiently large $j$ we have $|u(x)-u(\eps_j x_i)| < \eta$ for every $x \in \ballBoundingAnnulusJI$ and $i \in \goodJ$, Thus, we obtain 
    \begin{equation*}
        \begin{split}
            \limsup_{j\to \infty}\eps_j^n \sum_{i \in \goodJ} |(u)_{\ballBoundingAnnulusJI} - u(\eps_j x_i)|^p  \leq \eta^p \limsup_{j \to \infty} \eps_j^n \# \insideDomainJ =  \eta^p \IG \L^n(\domain)
        \end{split}
    \end{equation*}
    and the arbitrariness of $\eta>0$ yields  
    \begin{equation}\label{e:u-nel-centro}
        \lim_{j\to\infty }\eps_j^n \sum_{i \in \goodJ} |(u)_{\ballBoundingAnnulusJI} - u(\eps_j x_i)|^p = 0.
    \end{equation}
    Then, \eqref{e: integral approximation G_j,R} follows by \eqref{e: approximation of integral} in Lemma~\ref{lmm: approximation of integral} also observing that, by Lemma~\ref{lmm: eps^n card(inside minus thinnedInside) -> 0}-\ref{enum: card sets thinnedInsideDomainEps minus goodEps}, we have
    \begin{equation*}
        \lim_{\truncParam\to\infty}\limsup_{j \to \infty} \eps_j^n \sum_{i \in \insideDomainJ^{2/\truncParam} \setminus \goodJ} |u(\eps_j x_i)|^p \leq \norm{u}_{L^\infty(\domain)}^p \lim_{\truncParam\to\infty}\limsup_{j \to \infty}  \eps_j^n \#(\thinnedInsideDomainJ \setminus \goodJ) = 0\,.
    \end{equation*}
    We now turn to the proof of \eqref{e: integral approximation G_j,R with radii}.
    
    Since $\rho_i \leq \truncParam$ whenever $i \in G_{j, \truncParam}$, arguing as above readily gives 
    \begin{equation*}
        \lim_{j \to \infty} \eps_j^n \sum_{i \in G_{j, \truncParam}} |(u_j)_{\annulusJI} - (u)_{\ballBoundingAnnulusJI}|^p \rho_i^{n-sp} = 0.
    \end{equation*}
    Finally, by combining \eqref{e:u-nel-centro}, \eqref{e: approximation of integral} in Lemma~\ref{e: approximation of integral with radii}, Lemma~\ref{lmm: eps^n card(F_eps,R) -> 0 implies eps^n sum rho_i^n-sp -> 0}, and Lemma~\ref{lmm: eps^n card(inside minus thinnedInside) -> 0}-\ref{enum: card sets thinnedInsideDomainEps minus goodEps}
    we deduce
    \begin{equation*}
        \lim_{\truncParam \to \infty}\limsup_{j\to\infty} \eps_j^n \sum_{i \in \insideDomainJ^{2/\truncParam} \setminus \goodJ} |u(\eps_j x_i)|^p \rho_i^{n-sp} \leq \norm{u}_{L^\infty(\domain)}^p \lim_{\truncParam \to \infty}\limsup_{j \to \infty} \eps_j^n \sum_{i \in \thinnedInsideDomainJ \setminus \goodJ} \rho_i^{n-sp} = 0
    \end{equation*}
    and hence the claim. 
\end{proof}

\begin{rmk}\label{rmk:random_shapes}
For later use we observe that \eqref{e: integral approximation G_j,R with radii} holds also true if the sequence of random variables $\{\rho_i\}$ is replaced by $\{\gamma_i\}$ as in \ref{h: geometric assumption on obstacle of random shape} and \ref{h: stationarity of random capacities}. Namely, there exists $\Omega'\in \sigmaAlgebra$ with $\mathbb P(\Omega')=1$ such that for every $\omega \in \Omega'$ we have
\begin{equation}\label{e: integral approximation G_j,R with gamma_i}
        \lim_{\truncParam \to\infty}\sup_{L\in \mathbb N}
        \limsup_{j\to\infty}\bigg|\eps_j^n \sum_{i \in \goodJOmega} |(u_j)_{\annulusJIomega}|^p \gamma_i(\omega) - \norm{u}_{L^p(\domain)}^p \E\bigg[\sum_{x_i\in Q} \gamma_i \bigg]\bigg|=0\,.
    \end{equation}
\end{rmk}

\section{Proof of the main result}

This section is primarily devoted to the proof of the almost sure $\Gamma$-convergence result stated in Theorem~\ref{thm: main result}. The argument proceeds in two main steps, establishing separately the lower bound and the upper bound inequalities. At the end of this section we outline the proof of Theorem~\ref{thm: main result random shapes}, highlighting the points at which it differs from the proof of Theorem~\ref{thm: main result}.

\medskip

We start proving the lower-bound inequality.   


\begin{prop}\label{prop: gamma-liminf}
There exists $\sampleSpace' \in \sigmaAlgebra$ with $\mathbb{P}(\sampleSpace')= 1$ such that for every $\omega \in \sampleSpace'$, $\eps_j \downarrow 0$, $u\in L^p(U)$, and $u_j \to u$ in $L^p(\domain)$ we have 
\begin{equation}\label{e:lower-bound}
  \functional(u) \leq  \liminf_{j \to \infty} \functional_j^\omega (u_j).
\end{equation}
\end{prop}
\begin{proof}
    Let $\Omega' \in \sigmaAlgebra$ be as in Proposition~\ref{lmm: integral approximation G_j,R} and $\omega \in \sampleSpace'$. 
    We assume that the right-hand side in \eqref{e:lower-bound} is finite otherwise the liminf-inequality is trivially satisfied. Up to possibly extracting a subsequence, we can additionally assume that the liminf in \eqref{e:lower-bound} is actually a limit. As a consequence we immediately deduce that $\sup_{j \in \N} |u_j|_{W^{s,p}(\domain)} < +\infty$ and $\tilde{u}_j = 0$ $\csp$-q.e. on $\setOfObstaclesJOmega \cap \domain$ for every $j \in \N$.
    
    Fix $\truncParam, \joiningLemmaEnergyDiffParam \in \N$, $\delta > 0$ and define $m_\truncParam := \lfloor \truncParam^\frac{n-sp+1}{sp} \rfloor$.
    Let $\{w_j^\omega\}_{j \in \N}$ be the sequence given by Lemma~\ref{lmm: joining lemma} applied with $G=\goodEpsOmega$ and
    $m$ replaced by $m_R$. In view of \eqref{e:en-estimate-joining-lemma} we can find a constant $c_{n,s,p,\truncParam}>0$ such that 
    \begin{equation}\label{e:lower-bound-modified}
        \left(1+\frac{c_{n,s,p,R}}{\joiningLemmaEnergyDiffParam}\right) \liminf_{j \to \infty} \functional_j^\omega(u_j) \geq \liminf_{j \to \infty} \functional_j^\omega(w_j^\omega).
    \end{equation}
    Up to a possible further subsequence, we may assume that the liminf in the right-hand side of \eqref{e:lower-bound-modified} is a limit and that $w_j^\omega \to u$ $\L^n$-a.e. in $\domain$.
    We observe that $\cup_{i \in \goodJOmega} (\ballContainingAnnulusJIomega \times \ballContainingAnnulusJIomega) \subset \diagonalDelta$, for sufficiently large $j$. Therefore, we have
    \begin{equation*}
        \functional_j^\omega(w_j^\omega) \geq \int_{(\domain \times \domain) \setminus \diagonalDelta} \frac{|w_j^\omega(x) - w_j^\omega(y)|^p}{|x-y|^{n+sp}} \, \dd{x} \, \dd{y} + \sum_{i \in \goodJOmega} |w_j^\omega|^p_{W^{s,p}(\ballContainingAnnulusJIomega)},
    \end{equation*}
    so that Fatou's Lemma entails
    \begin{equation}\label{e:energy-after-Fatou}
        \liminf_{j \to \infty} \functional_j^\omega(w_j^\omega) \geq  \locdefect{(\domain\times \domain) \setminus \diagonalDelta}{u}{x}{y}  + \liminf_{j \to \infty}\sum_{i \in \goodJOmega} |w_j^\omega|^p_{W^{s,p}(\ballContainingAnnulusJIomega)}.
     \end{equation}
     Recalling the definition of $\tilde{B}_{j}^{i,h_j}$ in \eqref{eq: definition of ball bounding annulus} we observe that 
     \begin{equation*}
        \begin{split}
             &\sum_{i \in \goodJOmega} |w_j^\omega|^p_{W^{s,p}(\ballContainingAnnulusJIomega)} \\
             &\geq  \sum_{i \in \goodJOmega} \inf\{\wspseminorm{w}{\tilde{B}_{j}^{i,h_j^\omega}}: w \in \wsp{\R^n}, \,  w = 0 \text{ on } \annulusJIomega, \, \tilde{w}=(u_j)_{\annulusJIomega}   \, \cspqeon \obstacleJIOmega \} \\
             &= \sum_{i \in \goodJOmega} |(u_j)_{\annulusJIomega}|^p \csp(\obstacleJIOmega, \ballBoundingAnnulusJIomega, m_\truncParam^{-3h_j^\omega -2} \eps_j \truncParam^{-1}) \\
             &= \sum_{i \in \goodJOmega} \lambda_j^{n-sp}|(u_j)_{\annulusJIomega}|^p \csp(B_{\rho_i(\omega)}, B_{m_\truncParam^{-3h_j^\omega-1}\eps_j \lambda_j^{-1}\truncParam^{-1}}, m_\truncParam^{-3h_j^\omega -2} \eps_j \lambda_j^{-1}\truncParam^{-1}). 
        \end{split}
     \end{equation*}
    Since $B_{\rho_i(\omega)} \subset B_\truncParam$ for every $i \in \goodJOmega$, appealing to 
    \eqref{e:cap2} in Lemma~\ref{l:loccap}
    we get
     \begin{multline*}
        \sup_{i \in \goodJOmega} (\csp(B_{\rho_i(\omega)}) - \csp(B_{\rho_i(\omega)}, B_{m_\truncParam^{-3h_j^\omega-1}\eps_j \lambda_j^{-1}\truncParam^{-1}},m_\truncParam^{-3h_j^\omega-2}\eps_j \lambda_j^{-1}\truncParam^{-1})) \\ \leq \frac{c_{n,s,p}}{(m_\truncParam-1)^{sp}} C_{s,p}(B_\truncParam,B_{m_\truncParam^{-3h_j^\omega -2} \eps_j \lambda_j^{-1}\truncParam^{-1}}), 
     \end{multline*}
     for some $c_{n,s,p}>0$; therefore 
     \begin{multline*}
        \liminf_{j \to \infty} \sum_{i \in \goodJOmega} |w_j^\omega|^p_{W^{s,p}(\ballContainingAnnulusJIomega)} \geq \csp(B_1) \liminf_{j \to \infty} \eps_j^n \sum_{i \in \goodJOmega} |(u_j)_{\annulusJIomega}|^p \rho_i(\omega)^{n-sp}\\ - \frac{c_{n,s,p}}{(m_\truncParam-1)^{sp}}\limsup_{j \to \infty} C_{s,p}(B_\truncParam,B_{m_\truncParam^{-3h_j^\omega -2} \eps_j \lambda_j^{-1}\truncParam^{-1}}) \eps_j^n \sum_{i \in \goodJOmega} |(u_j)_{\annulusJIomega}|^p . 
     \end{multline*}
     Thus, we find 
     \begin{equation*}
        \begin{split}
            &
            \lim_{\truncParam \to \infty}\lim_{\joiningLemmaEnergyDiffParam \to \infty}\liminf_{j \to \infty}
            \sum_{i \in \goodJOmega} |w_j^\omega|^p_{W^{s,p}(\ballContainingAnnulusJIomega)} \\
            &\geq \csp(B_1) \lim_{\truncParam \to \infty}\lim_{\joiningLemmaEnergyDiffParam \to \infty} \liminf_{j \to \infty} \eps_j^n \sum_{i \in \goodJOmega} |(u_j)_{\annulusJIomega}|^p \rho_i(\omega)^{n-sp}\\ 
            &- \lim_{\truncParam \to \infty}\frac{c_{n,s,p}}{(m_\truncParam-1)^{sp}}\lim_{\joiningLemmaEnergyDiffParam \to \infty} \limsup_{j \to \infty} C_{s,p}(B_\truncParam,B_{m_\truncParam^{-3h_j^\omega -2} \eps_j \lambda_j^{-1}\truncParam^{-1}}) \eps_j^n \sum_{i \in \goodJOmega} |(u_j)_{\annulusJIomega}|^p .         
        \end{split}
     \end{equation*}
     By 
     \eqref{e:cap1} in Lemma~\ref{l:loccap} and by the scaling property of the fractional capacity we have 
     \begin{equation*}
        \lim_{j \to \infty} C_{s,p}(B_\truncParam,B_{m_\truncParam^{-3h_j^\omega -2} \eps_j \lambda_j^{-1}\truncParam^{-1}}) = \csp(B_\truncParam) = \truncParam^{n-sp} \csp(B_1),
     \end{equation*}
     hence recalling the definition of $m_\truncParam$ we get
     \begin{equation*}
       \limTruncParamJoiningLemmaParamJtoInfty
       \frac{c_{n,s,p}}{(m_\truncParam-1)^{sp}} C_{s,p}(B_\truncParam,B_{m_\truncParam^{-3h_j^\omega -2} \eps_j \lambda_j^{-1}\truncParam^{-1}})  = 0.
     \end{equation*}
     Then, Proposition~\ref{lmm: integral approximation G_j,R} yields 
     \begin{equation*}
        \begin{split}
            \lim_{\truncParam \to \infty}\lim_{\joiningLemmaEnergyDiffParam \to \infty}\liminf_{j \to \infty}
            \sum_{i \in \goodJOmega} |w_j^\omega|^p_{W^{s,p}(\ballContainingAnnulusJIomega)} \geq \IG \csp(B_1) \norm{u}_{L^p(\domain)}^p \int_{\R_+} \rho^{n-sp} \, \dd{\MD}(\rho) = \capacityConstantErgodic \norm{u}_{L^p(\domain)}^p, 
        \end{split}
     \end{equation*}
     thus 
     \begin{equation}\label{e:lower-bound-capacity}
        \begin{split}
            \liminf_{j \to \infty} \functional_j^\omega(u_j) &= 
            \lim_{\truncParam \to \infty}\lim_{\joiningLemmaEnergyDiffParam \to \infty}\liminf_{j \to \infty}
            \left(1+\frac{c_{n,s,p,R}}{\joiningLemmaEnergyDiffParam}\right) \functional_j^\omega(u_j) \\ 
            &\geq \locdefect{(\domain\times \domain) \setminus \diagonalDelta}{u}{x}{y} + \capacityConstantErgodic \norm{u}_{L^p(\domain)}^p.
        \end{split}
     \end{equation}
     Finally, the conclusion follows by gathering \eqref{e:lower-bound-modified}-\eqref{e:lower-bound-capacity} and by passing to limit as $\delta \to 0$ thanks to the Dominated Convergence Theorem. 
\end{proof}
In the next proposition we prove the so-called limsup-inequality. 

\begin{prop}
    \label{prop: gamma-limsup}
There exists $\sampleSpace' \in \sigmaAlgebra$ with $\mathbb{P}(\sampleSpace')= 1$ such that for every $\omega \in \sampleSpace'$, $u \in L^p(\domain)$, and $\eps_j \downarrow 0$, there exists $u_j \to u$ in $L^p(\domain)$ such that 
\begin{equation}\label{e:upper-bound}
    \limsup_{j \to \infty} \functional_j^\omega (u_j) \leq \functional(u).
\end{equation}
\end{prop}
\begin{proof}
    \intersectionOfEventsGivenBy Lemmas \ref{lmm: limit of eps^n sum rho_i^n-sp 1_(R,infty)(rho_i)}, \ref{lmm:lq norm}, \ref{lmm: NVG have vanishing relative capacity}, and Proposition~\ref{lmm: integral approximation G_j,R} and let $\omega \in \sampleSpace'$.
    To simplify the notation, in the rest of the proof we are going to omit the explicit dependence on $\omega$ of all the relevant quantities.

  We observe that by a standard continuity and density argument it suffices to prove \eqref{e:upper-bound} for target functions $u \in W^{1,\infty}(U) \cap L^p(U)$ (cf.\ \cite[Proposition~1.28 and Remark~1.29]{braides-beginners}).

    Let $\joiningLemmaEnergyDiffParam, \truncParam \in \N$ be fixed and let $\{w_j\}_{j \in \N}$ be given by Lemma~\ref{lmm: joining lemma} applied with $G=\veryGoodJ$, $u_j \equiv u$,  and $m = 2$.
    Let then $\xi_\joiningLemmaEnergyDiffParam \in W^{s,p}(\Rn; [0,1])$ be such that $\xi_\joiningLemmaEnergyDiffParam = 0$ on $\R^n \setminus B_\joiningLemmaEnergyDiffParam$,  $\tilde{\xi}_\joiningLemmaEnergyDiffParam = 1$ $\cspqe$ on $B_1$, and 
    \begin{equation}\label{e:xiL}
|\xi_\joiningLemmaEnergyDiffParam|^p_{\Wsp(\Rn)} \leq \CSP(B_1, B_\joiningLemmaEnergyDiffParam) + \frac{1}{\joiningLemmaEnergyDiffParam}.
    \end{equation}
    Recalling that $\ballBoundingAnnulusJI=B_{2^{-3h_j-1}\eps_jR^{-1}}(\eps_jx_i)$  (cf. \eqref{eq: definition of ball bounding annulus}), we define 
    \begin{equation*}
        v_j(x):= \begin{cases}
            w_j(x) \quad &\text{if } x \in \domain\setminus \bigcup\limits_{i \in \veryGoodJ} \ballBoundingAnnulusJI, \\
            \left(1-\xi_\joiningLemmaEnergyDiffParam\left(\frac{x-\eps_j x_i}{\lambda_j \rho_i}\right)\right) (u)_{\annulusJI} \quad &\text{if } x \in \bigcup\limits_{i \in \veryGoodJ} \ballBoundingAnnulusJI.
        \end{cases}
    \end{equation*}
    For $i \in \veryGoodJ$ we have $\rho_i \leq \truncParam$, hence $\joiningLemmaEnergyDiffParam \lambda_j \rho_i < 2^{-3h_j-1} \eps_j \truncParam^{-1}$, for sufficiently large $j \in \N$. In particular, this implies that $\obstacleJI \subset \ballBoundingAnnulusJI$ for every $i \in \veryGoodJ$, thus $v_j = 0$ on $\cup_{i \in \veryGoodJ} \obstacleJI$.
    Moreover, since $w_j = (u)_{\annulusJI}$ on $\partial \ballBoundingAnnulusJI$, we immediately deduce that $v_j \in W^{s,p}(\domain)$.

    Now, let $\phi_{j,\truncParam} \in W^{s,p}(\Rn; [0,1])$ be such that 
    \begin{align}\label{e:phijR}
        &\tilde{\phi}_{j,\truncParam} = 0 \quad \cspqeon \bigcup_{i \in \notVeryGoodJ} \ballCriticalRadiusJI, \notag\\
        &\tilde{\phi}_{j,\truncParam} = 1 \quad \cspqeon \Rn \setminus \bigcup_{i \in \notVeryGoodJ} \ballDoubleCriticalRadiusJI, \notag \\
        &|\phi_{j,\truncParam}|^p_{W^{s,p}(\Rn)} \leq 2 \CSP\left(\bigcup_{i \in \notVeryGoodJ} \ballCriticalRadiusJI, \bigcup_{i \in \notVeryGoodJ} \ballDoubleCriticalRadiusJI\right).
    \end{align}

    We define $u_j := v_j \phi_{j,\truncParam}$, then, by construction, $u_j = 0$ on $\setOfObstaclesJ$ and $u_j \in W^{s,p}(U)$. 

    We recall property \eqref{eq: very good enlarged obstacles and not very good obstacles with double radius are disjoint} and define 
        \begin{equation}\label{e:Uj}
            \domain_j:= \domain \setminus ( \tilde{B}_{\veryGoodJ}  \cup  D_{\notVeryGoodJ})
        \end{equation}
        with
        \begin{equation}\label{e:BVG}
            \tilde{B}_{\veryGoodJ} := \bigcup_{i \in \veryGoodJ} \ballBoundingAnnulusJI,    
        \end{equation}
        and
        \begin{equation}\label{e:DNVG}
            D_{\notVeryGoodJ}:=
            \bigcup_{i \in \notVeryGoodJ} \ballDoubleCriticalRadiusJI
            \cap \domain\,.
        \end{equation}
        We observe that for every $\joiningLemmaEnergyDiffParam\in\mathbb N$ we have
        \begin{equation}\label{e: gamma-limsup, u_j converges to u in L^p}
            \lim_{j\to\infty} \sup_{\truncParam\in \N}\norm{u_j-u}_{L^p(\domain)} = 0.
        \end{equation}
        Indeed we have
        \begin{equation}
            \label{eq: gamma-limsup, int_U |u_j-u|^p}
            \int_\domain |u_j - u|^p \leq \int_{\domain_j} |w_j- u|^p + \sum_{i \in \veryGoodJ} \int_{\ballBoundingAnnulusJI} |u_j -u|^p + \sum_{i \in \notVeryGoodJ} \int_{\ballDoubleCriticalRadiusJI} |u_j -u|^p. 
        \end{equation}
        The first term in \eqref{eq: gamma-limsup, int_U |u_j-u|^p} approaches zero as $j\to\infty$ since $w_j \to u$ in $L^p(\domain)$. On the other hand, using that $u_j=v_j= \left(1-\xi_\joiningLemmaEnergyDiffParam\left(\frac{x-\eps_j x_i}{\lambda_j \rho_i}\right)\right) (u)_{\annulusJI}$ on $\bigcup\limits_{i \in \veryGoodJ}\ballBoundingAnnulusJI$ together with the very definition of $\xi_\joiningLemmaEnergyDiffParam$,
        the second term in \eqref{eq: gamma-limsup, int_U |u_j-u|^p} can be estimated as follows: 
        \begin{equation*}
            \begin{split}
                &\sum_{i \in \veryGoodJ} \int_{\ballBoundingAnnulusJI} |u_j - u|^p \\
                &=\sum_{i \in \veryGoodJ} \int_{\obstacle_j^i} |u|^p +  \sum_{i \in \veryGoodJ} \int_{B_{\joiningLemmaEnergyDiffParam \lambda_j \rho_i}(\eps_j x_i) \setminus \obstacle_j^i} |u_j -u|^p + \sum_{i \in \veryGoodJ} \int_{\ballBoundingAnnulusJI \setminus B_{\joiningLemmaEnergyDiffParam \lambda_j \rho_i}(\eps_j x_i)} |u - (u)_{\annulusJI}|^p \\
                &\lesssim_{n,\joiningLemmaEnergyDiffParam} \norm{u}_{L^\infty(\domain)}^p \sum_{i \in \veryGoodJ} (\lambda_j \rho_i)^n +  \sum_{i \in \veryGoodJ}\int_{\ballBoundingAnnulusJI} |u - (u)_{\annulusJI}|^p
            \end{split}
        \end{equation*}
        Furthermore, by the Poincaré-Wirtinger Inequality for fractional Sobolev spaces (see e.g. \cite[Theorem~2.6 and Remark~2.7]{focardi-aperiodic-fractional}) and since $\{\ballBoundingAnnulusJI\}_{i \in \veryGoodJ}$ are pairwise disjoint, we get
        \begin{equation*}
            \begin{split}
                \sum_{i \in \veryGoodJ}\int_{\ballBoundingAnnulusJI} |u - (u)_{\annulusJI}|^p \lesssim_{n,s,p} (2^{-3h_j-1}\eps_j \truncParam^{-1})^{sp} \sum_{i \in \veryGoodJ}|u|^p_{\Wsp(\ballBoundingAnnulusJI)} \leq \eps_j^{sp} |u|^p_{\Wsp(\domain)},
            \end{split}
        \end{equation*}
        which implies, by Lemma~\ref{lmm:lq norm}, 
        that the second term in \eqref{eq: gamma-limsup, int_U |u_j-u|^p} vanishes as $j\to \infty$.
        Regarding the third term in \eqref{eq: gamma-limsup, int_U |u_j-u|^p}, we have 
        \begin{equation*}
            \begin{split}
                \sum_{i \in \notVeryGoodJ} \int_{\ballDoubleCriticalRadiusJI} |u_j -u|^p \lesssim_{n,p} \norm{u}_{L^\infty(\domain)}^p 
                \sum_{i \in \notVeryGoodJ} (\lambda_j \rho_i)^n
                \lesssim_{n,p}\norm{u}_{L^\infty(\domain)}^p \sum_{i \in \insideDomainJ} (\lambda_j \rho_i)^n\,, 
            \end{split}
        \end{equation*}
        hence the conclusion follows at once by Lemma~\ref{lmm:lq norm} with $q=n$.
    
    The rest of the proof is devoted to show that 
        \begin{equation}\label{e:approx-upper-bound}
            \limsup_{j\to\infty}\functional_j(u_j) \leq \functional(u)+\errordRL, 
        \end{equation}
        where $\errordRL=\errordRL(\delta, R,L)>0$ is such that 
        \[
        \displaystyle{\lim_{\delta\to 0}\lim_{\truncParam \to \infty}\lim_{\joiningLemmaEnergyDiffParam\to\infty}\errordRL=0}.
        \]
    Since the computations leading to \eqref{e:approx-upper-bound} are rather involved, we perform them in a number of consecutive steps. 
    
    \begin{claim} In this first step we preliminary reformulate \eqref{e:approx-upper-bound}.
    \end{claim}
        By \eqref{e: gamma-limsup, u_j converges to u in L^p} and the Lebesgue Dominated Convergence Theorem, for every $\truncParam,\,\joiningLemmaEnergyDiffParam \in \mathbb N$, and $\delta>0$, we have
        \begin{equation*}
           \lim_{j\to\infty}\Dsp(u_j, (\domain \times \domain) \setminus \diagonalDelta) = \Dsp(u,(\domain \times \domain)\setminus \diagonalDelta).
        \end{equation*}
        Therefore, we are only left to estimate the energy of $u_j$ in a neighborhood of the diagonal $\Delta$. To this end, since $u_j = w_j$ on $\domain_j$, by Lemma~\ref{lmm: joining lemma} we deduce that
        \begin{equation}
            \label{eq: gamma-limsup, locality defect of u_j and w_j close to the diagonal}
            \begin{split}
                &\Dsp(u_j, (\domain_j \times \domain_j) \cap \diagonalDelta) = \Dsp(w_j, (\domain_j \times \domain_j) \cap \diagonalDelta) \\
                &\leq \Dsp(w_j, (\domain \times \domain) \cap \diagonalDelta) \leq \frac{c_{n,s,p,\truncParam}}{\joiningLemmaEnergyDiffParam} |u|_{\Wsp(\domain)}^p + \Dsp(u,(\domain \times \domain) \cap \diagonalDelta),
            \end{split}
        \end{equation}
        for some $c_{n,s,p,\truncParam} > 0$.
        Hence, by the Dominated Convergence Theorem  
        \begin{equation}\label{e:claim522 1}
            \limsup_{j\to\infty}
            \Dsp(u_j, (\domain_j \times \domain_j) \cap \diagonalDelta)\leq\errordRL\,.
        \end{equation}
        Therefore, in view of \eqref{e:claim522 1}, the inequality in \eqref{e:approx-upper-bound} reduces to proving that 
        \begin{equation}\label{e:claim522 2}
            \limsup_{j\to\infty}(
            \Dsp(u_j, (\domain \times (\domain \setminus \domain_j)) \cap \diagonalDelta) + \Dsp(u_j, ((\domain\setminus\domain_j)\times \domain_j) \cap \diagonalDelta)) 
            \leq \functional(u) +\errordRL.
        \end{equation}
        To establish \eqref{e:claim522 2} we use the definition of $\domain_j$ to get
        \begin{equation*}
            \domain = \tilde{B}_{\veryGoodJ} \cup D_{\notVeryGoodJ} \cup \domain_j,
        \end{equation*}
        where the sets in the above union are pairwise disjoint and defined in \eqref{e:Uj}, \eqref{e:BVG}, and \eqref{e:DNVG}, respectively.
        In particular, we have
        \begin{align*}
            &\domain \times (\domain \setminus \domain_j) = (\domain \setminus \domain_j)^2 \cup (\domain_j \times (\domain \setminus \domain_j)), \\
            &(\domain \setminus \domain_j)^2 = (\tilde{B}_{\veryGoodJ} \cup D_{\notVeryGoodJ})^2, \\
            &(\domain \setminus \domain_j) \times \domain_j = (\tilde{B}_{\veryGoodJ} \cup D_{\notVeryGoodJ}) \times \domain_j.
        \end{align*}
        Since $\Dsp(\cdot, A\times B) = \Dsp(\cdot, B\times A)$, proving \eqref{e:claim522 2} amounts to showing that 
        \begin{multline}
            \label{eq: gamma-limsup, locality defect inequality to prove}
            \limsup_{j\to\infty}\Big( \Dsp(u_j, \tilde{B}_{\veryGoodJ}^2 \cap \diagonalDelta) +2 \Dsp(u_j, (\tilde{B}_{\veryGoodJ} \times D_{\notVeryGoodJ})\cap \diagonalDelta)  \\ + \Dsp(u_j, D_{\notVeryGoodJ}^2 \cap \diagonalDelta) + \Dsp(u_j, (\tilde{B}_{\veryGoodJ} \times \domain_j) \cap \diagonalDelta)  \\+ \Dsp(u_j, (D_{\notVeryGoodJ} \times \domain_j) \cap \diagonalDelta) \Big) \leq \functional(u)+\errordRL. 
        \end{multline}
        In the next steps we will prove \eqref{eq: gamma-limsup, locality defect inequality to prove} by separately treating the different terms in its left-hand side.
        \begin{claim} We show that
            \begin{equation}\label{claim:main contribution}
                  \limsup_{j\to\infty}\Dsp(u_j, \tilde{B}_{\veryGoodJ}^2 \cap \diagonalDelta) \leq \capacityConstantErgodic\norm{u}_{L^p(\domain)}^p
                  +\errordRL\,.
            \end{equation}
        \end{claim}
            We have
            \begin{equation}
                \label{eq: gamma-limsup, VG^2, sum}
                \Dsp(u_j, \tilde{B}_{\veryGoodJ}^2 \cap \diagonalDelta) \leq \sum_{i \in \veryGoodJ} |u_j|^p_{\Wsp(\ballBoundingAnnulusJI)} + \sum_{\substack{i,k \in \veryGoodJ \\ i \neq k}} \Dsp(u_j, (\ballBoundingAnnulusJI \times \ballBoundingAnnulusJK) \cap \diagonalDelta).
            \end{equation}
            We now start by estimating the first term on the right hand side of \eqref{eq: gamma-limsup, VG^2, sum}. To this end let $x,y \in \ballBoundingAnnulusJI$; by definition of $u_j$, we have 
            \begin{equation*}
                |u_j(x)-u_j(y)| =  \lvert(u)_{\annulusJI}\rvert\left\lvert\xi_\joiningLemmaEnergyDiffParam\left(\frac{x-\eps_j x_i}{\lambda_j \rho_i}\right) - \xi_\joiningLemmaEnergyDiffParam\left(\frac{y-\eps_j x_i}{\lambda_j \rho_i}\right) \right\rvert
            \end{equation*}
            and hence, by a change of variables, 
            estimate \eqref{e:cap1} in Lemma~\ref{l:loccap},
            Proposition~\ref{lmm: integral approximation G_j,R}, and \eqref{e:xiL}
            \begin{align}\label{e:claim523 1}
                    \limsup_{j\to\infty}&
                    \sum_{i \in \veryGoodJ} |u_j|^p_{\Wsp(\ballBoundingAnnulusJI)} \leq \limsup_{j\to\infty}
                    \eps_j^n \sum_{i \in \goodJ} \rho_i^{n-sp} |(u)_{\annulusJI}|^p |\xi_\joiningLemmaEnergyDiffParam|^p_{\Wsp(\Rn)}\notag \\
                    &\leq \limsup_{j\to\infty}
                    \left(\CSP(B_1, B_\joiningLemmaEnergyDiffParam) + \frac{1}{\joiningLemmaEnergyDiffParam}\right) \eps_j^n \sum_{i \in \goodJ} \rho_i^{n-sp} |(u)_{\annulusJI}|^p \notag\\
                    &\leq \csp(B_1) \IG \norm{u}_{L^p(\domain)}^p \int_{\R_+} \rho^{n-sp} \, \dd{\MD}(\rho) +\errordRL= \capacityConstantErgodic\norm{u}_{L^p(\domain)}^p+\errordRL.
            \end{align}            
Then, it remains to estimate the second term in \eqref{eq: gamma-limsup, VG^2, sum}. For $(x,y) \in (\ballBoundingAnnulusJI \times \ballBoundingAnnulusJK) \cap \diagonalDelta$, with $i,k \in \veryGoodJ$ and $i\neq k$, we have 
            \begin{multline}
                \label{eq: gamma-limsup, VG^2, different balls}
                |u_j(x) - u_j(y)|^p \lesssim_p |(u)_{\annulusJI} - (u)_{\annulusJK}|^p \\+ \left|\xi_\joiningLemmaEnergyDiffParam\left(\frac{x-\eps_j x_i}{\lambda_j \rho_i}\right)\right|^p |(u)_{\annulusJI}|^p + \left|\xi_\joiningLemmaEnergyDiffParam \left(\frac{y-\eps_jx_k}{\lambda_j \rho_k}\right)\right|^p |(u)_{\annulusJK}|^p. 
            \end{multline}
            Moreover, since $i,k \in \veryGoodJ \subset \goodJ$, by the triangle inequality, for sufficiently large $j$, we get
            \begin{equation*}
                2 \frac{\eps_j}{\truncParam} \leq \eps_j|x_i - x_k| \leq \frac{\eps_j}{\truncParam} + \delta < 2 \delta,
            \end{equation*}
            and
            \begin{equation}
                \label{eq: gamma-limsup, VG^2, relation between |x-y| and |x_i - x_k|}
                \frac{1}{2} \eps_j |x_i - x_k| \leq |x-y| \leq \frac{3}{2} \eps_j |x_i - x_k|.
            \end{equation}
            Then, recalling that $u \in \lip(\domain)$, we obtain
            \begin{equation*}
                |(u)_{\annulusJI} - (u)_{\annulusJK}|^p \lesssim_p \lip(u)^p \eps_j^p |x_i - x_k|^p
            \end{equation*}
            and 
            \begin{equation*}
                \begin{split}
                    &\sum_{\substack{i,k \in \veryGoodJ \\ i \neq k}}\int_{(\ballBoundingAnnulusJI \times \ballBoundingAnnulusJK)\cap \diagonalDelta} \frac{|(u)_{\annulusJI} - (u)_{\annulusJK}|^p}{|x-y|^{n+sp}} \, \dd{x} \, \dd{y} \\
                    &\lesssim_{n,p} \lip(u)^p \eps_j^{-n-(s-1)p}  \truncParam^{-2n} \eps_j^{2n} \sum_{\substack{i,k \in \veryGoodJ \\ i \neq k, \, \eps_j|x_i - x_k| < 2 \delta}}|x_i - x_k|^{-n-(s-1)p}.
                \end{split}
            \end{equation*}
            By a counting argument (see also \cite[(2.7)]{focardi-aperiodic-fractional}), it is easy to see that for every $i\in \goodJ$ we have
            \begin{equation*}
                \#\left\{ k \in \goodJ : k \neq i, \frac{h}{\truncParam} < |x_i - x_k| \leq \frac{h+1}{\truncParam} \right\} \lesssim_n h^{n-1}.
            \end{equation*}
            If $n+(s-1)p \geq 0$, we can estimate the above sum using the fact that $\sum_{h=2}^{l} h^{-(1+a)} \leq \frac{l^{-a}}{-a}$ for any $l \in \N$ and $a<0$. Namely, we have
            \begin{equation*}
                \begin{split}
                    &\sum_{\substack{i,k \in \veryGoodJ \\ i \neq k, \, \eps_j|x_i - x_k| < 2 \delta}}|x_i - x_k|^{-n-(s-1)p} = \sum_{i \in \veryGoodJ} \sum_{\substack{k \in \veryGoodJ \\ k\neq i, \, \eps_j|x_i - x_k| < 2\delta}}|x_i - x_k|^{-n-(s-1)p} \\
                    &= \sum_{i \in \veryGoodJ} \sum_{h=2}^{\lfloor 2 \delta \truncParam / \eps_j \rfloor} \sum_{\substack{k \in \veryGoodJ \\ h < \truncParam |x_i - x_k| \leq h+1 }} |x_i - x_k|^{-n-(s-1)p} \\
                    &\lesssim_n \truncParam^{n+(s-1)p} \sum_{i \in \veryGoodJ} \sum_{h=2}^{\lfloor 2 \delta \truncParam / \eps_j \rfloor} h^{-(1+(s-1)p)} \lesssim_{n,s,p} \truncParam^{n+(s-1)p} \#\veryGoodJ \delta^{(1-s)p} \truncParam^{(1-s)p} \eps_j^{(s-1)p}.
                \end{split}
            \end{equation*}
            Since $\{B_\frac{\eps_j}{\truncParam}(\eps_j x_i)\}_{i \in VG_{j, \truncParam}}$ are pairwise disjoint, we have $\#\veryGoodJ \lesssim_n \L^n(\domain) (\frac{\eps_j}{\truncParam})^{-n}$, which gives 
            \begin{equation*}
                \sum_{\substack{i,k \in \veryGoodJ \\ i \neq k, \, \eps_j|x_i - x_k| < 2 \delta}}|x_i - x_k|^{-n-(s-1)p} \lesssim_{n,s,p} \truncParam^{2n} \eps_j^{-n+(s-1)p} \delta^{(1-s)p}
            \end{equation*}
            and 
            \begin{equation}\label{e:claim523 2}
                \sum_{\substack{i,k \in \veryGoodJ \\ i \neq k}}\int_{(\ballBoundingAnnulusJI \times \ballBoundingAnnulusJK)\cap \diagonalDelta} \frac{|(u)_{\annulusJI} - (u)_{\annulusJK}|^p}{|x-y|^{n+sp}} \, \dd{x} \, \dd{y} 
                \lesssim_{n,s,p} \lip(u)^p 
                \delta^{(1-s)p}\,. 
            \end{equation}
            If, on the other hand, $n+(s-1)p < 0$, we get both
            \begin{equation*}
                \begin{split}
                    &\sum_{\substack{i,k \in \veryGoodJ \\ i \neq k, \, \eps_j|x_i - x_k| < 2 \delta}}|x_i - x_k|^{-n-(s-1)p} \\
                    &\leq \#\veryGoodJ^2 \left(\frac{2\delta}{\eps_j}\right)^{-n-(s-1)p} \lesssim_{n,s,p} \truncParam^{2n} \eps_j^{-2n} \eps_j^{n+(s-1)p} \delta^{-n-(s-1)p}
                \end{split}
            \end{equation*}
            and 
            \begin{equation}\label{e:claim523 3}
                \sum_{\substack{i,k \in \veryGoodJ \\ i \neq k}}\int_{(\ballBoundingAnnulusJI \times \ballBoundingAnnulusJK)\cap \diagonalDelta} \frac{|(u)_{\annulusJI} - (u)_{\annulusJK}|^p}{|x-y|^{n+sp}} \, \dd{x} \, \dd{y} 
                \lesssim_{n,s,p} \lip(u)^p 
                \delta^{-n-(s-1)p}\,. 
            \end{equation}
            We now deal with the second term in \eqref{eq: gamma-limsup, VG^2, different balls} (the third one being analogous). By \eqref{eq: gamma-limsup, VG^2, relation between |x-y| and |x_i - x_k|} a change of variables gives 
            \begin{equation*}
                \begin{split}
                    &\int_{\ballBoundingAnnulusJI \times \ballBoundingAnnulusJK} \frac{|\xi_\joiningLemmaEnergyDiffParam(\lambda_j^{-1}\rho_i^{-1}(x-\eps_jx_i))|^p}{|x-y|^{n+sp}} \, \dd{x} \, \dd{y} \\
                    &= \int_{\ballBoundingAnnulusJI} |\xi_\joiningLemmaEnergyDiffParam(\lambda_j^{-1}\rho_i^{-1}(x-\eps_jx_i))|^p \left(\int_{\ballBoundingAnnulusJK} |x-y|^{-(n+sp)} \, \dd{y} \right) \, \dd{x} \\
                    &\lesssim_n \truncParam^{-n} \eps_j^n \eps_j^{-n-sp} |x_i - x_k|^{-n-sp} \int_{B_{\joiningLemmaEnergyDiffParam \lambda_j \rho_i} (\eps_j x_i)} |\xi_\joiningLemmaEnergyDiffParam(\lambda_j^{-1}\rho_i^{-1}(x-\eps_jx_i))|^p \, \dd{x} \\
                    &= \truncParam^{-n} \eps_j^{-sp} |x_i - x_k|^{-n-sp} (\lambda_j \rho_i)^n \int_{B_\joiningLemmaEnergyDiffParam} |\xi_\joiningLemmaEnergyDiffParam|^p\,,
                \end{split}
            \end{equation*}
            and, as a consequence, 
            \begin{equation}
                \label{eq: gamma-limsup norm p of xi_L}
                \begin{split}
                    &\limsup_{j\to\infty}\sum_{\substack{i,k \in \veryGoodJ \\ i \neq k}}\int_{(\ballBoundingAnnulusJI \times \ballBoundingAnnulusJK) \cap \diagonalDelta} \frac{|\xi_\joiningLemmaEnergyDiffParam(\lambda_j^{-1}\rho_i^{-1}(x-\eps_jx_i))|^p}{|x-y|^{n+sp}} \, \dd{x} \, \dd{y} \\
                    &\leq \limsup_{j\to\infty} \norm{\xi_\joiningLemmaEnergyDiffParam}_{L^p(B_\joiningLemmaEnergyDiffParam)}^p  \lambda_j^n  \eps_j^{-sp} \sum_{\substack{i,k \in \veryGoodJ \\ i \neq k, \, \eps_j|x_i - x_k| < 2 \delta}} |x_i - x_k|^{-n-sp} \\
                    &= \limsup_{j\to\infty}\norm{\xi_\joiningLemmaEnergyDiffParam}_{L^p(B_\joiningLemmaEnergyDiffParam)}^p  \lambda_j^n  \eps_j^{-sp}\sum_{i \in \veryGoodJ} \sum_{h=2}^{\lfloor 2 \delta \truncParam / \eps_j \rfloor} \sum_{\substack{k \in \veryGoodJ \\ h < \truncParam |x_i - x_k| \leq h+1 }}  |x_i - x_k|^{-n-sp} \\
                    &\lesssim_n \limsup_{j\to\infty} \norm{\xi_\joiningLemmaEnergyDiffParam}_{L^p(B_\joiningLemmaEnergyDiffParam)}^p  \lambda_j^n  \eps_j^{-sp} \truncParam^{n+sp} \#\veryGoodJ \sum_{h=2}^{\infty} h^{-n-sp} h^{n-1} \\
                    &\lesssim_n \limsup_{j\to\infty} \norm{\xi_\joiningLemmaEnergyDiffParam}_{L^p(B_\joiningLemmaEnergyDiffParam)}^p  \lambda_j^n  \eps_j^{-sp} \truncParam^{2n+sp} \eps_j^{-n}\sum_{h=2}^{\infty} h^{-1-sp} \lesssim  \norm{\xi_\joiningLemmaEnergyDiffParam}_{L^p(B_\joiningLemmaEnergyDiffParam)}^p \truncParam^{2n+sp}\lim_{j\to\infty} \eps_j^\frac{(sp)^2}{n-sp} = 0.
                \end{split} 
            \end{equation}
            Then \eqref{claim:main contribution} follows at once from \eqref{e:claim523 1}, \eqref{e:claim523 2}, \eqref{e:claim523 3}, and the very last estimate.
    
        \begin{claim} We prove that
            \begin{equation}\label{claim:2nd term}
                \limsup_{j\to\infty}\Dsp(u_j, (\tilde{B}_{\veryGoodJ} \times D_{\notVeryGoodJ})\cap \diagonalDelta)  \leq \errordRL.
            \end{equation}                
        \end{claim}
            Since $\{\ballBoundingAnnulusJI\}_{i \in \veryGoodJ}$ are pairwise disjoint, we deduce 
            \[
            \Dsp(u_j, (\tilde{B}_{\veryGoodJ} \times D_{\notVeryGoodJ})\cap \diagonalDelta) = \sum_{i \in \veryGoodJ} \Dsp(u_j, (\ballBoundingAnnulusJI \times D_{\notVeryGoodJ})\cap \diagonalDelta).
            \]
            Moreover, for every $(x,y) \in (\ballBoundingAnnulusJI \times D_{\notVeryGoodJ})\cap \diagonalDelta$ we have
            \begin{equation*}
                \begin{split}
                    &u_j(x)-u_j(y) = \left(1-\xi_\joiningLemmaEnergyDiffParam\left(\frac{x-\eps_jx_i}{\lambda_j \rho_i}\right)\right) (u)_{\annulusJI} - w_j(y)\phi_{j,\truncParam}(y) \\
                    &=(u)_{\annulusJI} - w_j(x) - (u)_{\annulusJI} \xi_\joiningLemmaEnergyDiffParam\left(\frac{x-\eps_jx_i}{\lambda_j\rho_i}\right) + w_j(x) - w_j(y) + w_j(y) (1-\phi_{j,\truncParam}(y))
                \end{split}
            \end{equation*}
            and hence 
            \begin{equation}
                \label{eq: gamma-limsup, VG x NVG, terms to estimate}
                \begin{split}
                    |u_j(x)-u_j(y)|^p &\lesssim_p |w_j(x) - (u)_{\annulusJI}|^p + \norm{u}_{L^\infty(\domain)}^p \left\lvert \xi_\joiningLemmaEnergyDiffParam\left(\frac{x-\eps_j x_i}{\lambda_j \rho_i}\right) \right\rvert^p \\&+ |w_j(x)-w_j(y)|^p + \norm{u}_{L^\infty(\domain)}^p |1-\phi_{j,\truncParam}(y)|^p.
                \end{split}
            \end{equation}
            Then the estimate in \eqref{e:stima Adams} and the fact that $w_j\equiv (u)_{\annulusJI}$ on $\annulusJI$ by definition give for sufficiently large $j \in \N$,
            \begin{equation}
                \label{eq: gamma-limsup, VG x NVG, w_j - (u)}
                \begin{split}
                    \sum_{i \in \veryGoodJ}& \int_{(\ballBoundingAnnulusJI \times D_{\notVeryGoodJ}) \cap \diagonalDelta} \frac{|w_j(x) - (u)_{\annulusJI}|^p}{|x-y|^{n+sp}} \, \dd{x} \, \dd{y} \\
                    &\lesssim_{n,s,p} \sum_{i \in \veryGoodJ} 
                   \int_{B_{2^{-3h_j-2}\eps_jR^{-1}}(\eps_jx_i)}\frac{|w_j(x) - (u)_{\annulusJI}|^p}{\dist^{sp}(x,\partial\ballBoundingAnnulusJI)} \, \dd{x} \\ 
                  &\lesssim_{n,s,p}\Big(\frac{2^{3h_j+2}R}{\eps_j}\Big)^{sp} \sum_{i \in \veryGoodJ} \int_{B_{2^{-3h_j-2}\eps_jR^{-1}}(\eps_jx_i)}|w_j(x) - (u)_{\annulusJI}|^p \, \dd{x} \\ 
                    &\lesssim_{n,s,p} \sum_{i \in \veryGoodJ} |w_j|^p_{\Wsp(\ballBoundingAnnulusJI)} \leq \Dsp(w_j, (\domain\times \domain) \cap \diagonalDelta),
                \end{split}
            \end{equation}
             where in the last but one inequality we have used the scaled
            Poincaré-Wirtinger Inequality \cite[(2.10) in Remark~2.7]{focardi-aperiodic-fractional}.
            Note that the last term in the above inequality \eqref{eq: gamma-limsup, VG x NVG, w_j - (u)} has been already estimated in \eqref{eq: gamma-limsup, locality defect of u_j and w_j close to the diagonal}.
            
            Concerning the second term in \eqref{eq: gamma-limsup, VG x NVG, terms to estimate}, since $i\in \veryGoodJ$, a change of variables yields   
            \begin{equation}
                \label{eq: gamma-limsup, VG x NVG, xi_L}
                \begin{split}
                    &\sum_{i \in \veryGoodJ} \int_{(\ballBoundingAnnulusJI \times D_{\notVeryGoodJ}) \cap \diagonalDelta} |x-y|^{-(n+sp)} \left\lvert \xi_\joiningLemmaEnergyDiffParam\left(\frac{x-\eps_jx_i}{\lambda_j \rho_i}\right) \right\rvert^p \, \dd{x} \, \dd{y} \\
                    &\leq \sum_{i \in \veryGoodJ} \int_{\ballBoundingAnnulusJI \times (\domain \setminus \ballBoundingAnnulusJI)} |x-y|^{-(n+sp)} \left\lvert \xi_\joiningLemmaEnergyDiffParam\left(\frac{x-\eps_jx_i}{\lambda_j \rho_i}\right) \right\rvert^p \, \dd{x} \, \dd{y} \\
                    &\leq\sum_{i \in \veryGoodJ} (\lambda_j \rho_i)^{n-sp} \Dsp(\xi_\joiningLemmaEnergyDiffParam, B_{2^{-3h_j -1} \eps_j \truncParam^{-1} \lambda_j^{-1} \rho_i^{-1}} \times (\R^n \setminus B_{2^{-3h_j -1} \eps_j \truncParam^{-1} \lambda_j^{-1} \rho_i^{-1}})).
                \end{split}
            \end{equation}
            In the last inequality we have used that, 
            since $\rho_i \leq \truncParam$ for any $i \in \veryGoodJ$, $2^{-3h_j -1} \eps_j \truncParam^{-1} \lambda_j^{-1} \rho_i^{-1} \geq 2^{-3\joiningLemmaEnergyDiffParam -1}\truncParam^{-2}  \eps_j \lambda_j^{-1} \geq \joiningLemmaEnergyDiffParam^2 > \joiningLemmaEnergyDiffParam$ for sufficiently large $j \in \N$, and that   $\xi_\joiningLemmaEnergyDiffParam = 0$ on $\R^n \setminus B_\joiningLemmaEnergyDiffParam$.
            Therefore, by Lemma~\ref{lmm: limit of eps^n sum rho_i^n-sp 1_(R,infty)(rho_i)} and by 
             \eqref{eq: locality defect on B(r) x B(r)^c} in Lemma~\ref{l:loccap}
            we have
            \begin{align}\label{e:claim524 2}
                \limsup_{j \to \infty}& \sum_{i \in \veryGoodJ} \int_{(\ballBoundingAnnulusJI \times D_{\notVeryGoodJ}) \cap \diagonalDelta} |x-y|^{-(n+sp)} \left\lvert \xi_\joiningLemmaEnergyDiffParam\left(\frac{x-\eps_jx_i}{\lambda_j \rho_i}\right) \right\rvert^p \, \dd{x} \, \dd{y} \notag \\
                &\leq 
                \Dsp(\xi_\joiningLemmaEnergyDiffParam, B_{\joiningLemmaEnergyDiffParam^2} \times (\Rn \setminus B_{\joiningLemmaEnergyDiffParam^2})) \lim_{j \to \infty}  \eps_j^n \sum_{i \in \insideDomainJ} \rho_i^{n-sp} \notag\\
                &= \Dsp(\xi_\joiningLemmaEnergyDiffParam, B_{\joiningLemmaEnergyDiffParam^2} \times (\Rn \setminus B_{\joiningLemmaEnergyDiffParam^2})) \IG \L^n(\domain) \int_{\R_+} \rho^{n-sp} \, \dd{\MD}(\rho) \leq\errordRL.
            \end{align}
            The third term in \eqref{eq: gamma-limsup, VG x NVG, terms to estimate} gives us
            \begin{equation}
                \label{eq: gamma-limsup, VG x NVG, w_j(x) - w_j(y)}
                \Dsp(w_j, (\tilde{B}_{\veryGoodJ} \times D_{\notVeryGoodJ}) \cap \diagonalDelta) \leq \Dsp(w_j, (\domain \times \domain) \cap \diagonalDelta),
            \end{equation}
            which was already estimated in \eqref{eq: gamma-limsup, locality defect of u_j and w_j close to the diagonal}. Regarding the last term in \eqref{eq: gamma-limsup, VG x NVG, terms to estimate}, since $\phi_{j,\truncParam} = 1$ on $\Rn \setminus D_{\notVeryGoodJ}$, by Lemma~\ref{lmm: NVG have vanishing relative capacity} and \eqref{e:phijR}, we have
            \begin{equation}
                \label{eq: gamma-limsup, VG x NVG, phi_j,R}
                \begin{split}
                    &\limsup_{j\to\infty}\sum_{i \in \veryGoodJ} \int_{(\ballBoundingAnnulusJI \times D_{\notVeryGoodJ}) \cap \diagonalDelta} \frac{|1-\phi_{j,\truncParam}(y)|^p}{|x-y|^{n-sp}} \, \dd{x} \, \dd{y} \\
                    &=\limsup_{j\to\infty}
                    \sum_{i \in \veryGoodJ} \int_{(\ballBoundingAnnulusJI \times D_{\notVeryGoodJ}) \cap \diagonalDelta} \frac{|(1-\phi_{j,\truncParam}(x))-(1-\phi_{j,\truncParam}(y))|^p}{|x-y|^{n-sp}} \, \dd{x} \, \dd{y} \\
                    &\leq\limsup_{j\to\infty}|\phi_{j,\truncParam}|^p_{\Wsp(\Rn)} =0.
                \end{split}
            \end{equation}
            Then \eqref{claim:2nd term} follows by gathering
            \eqref{eq: gamma-limsup, VG x NVG, w_j - (u)},
            \eqref{e:claim524 2},
            \eqref{eq: gamma-limsup, VG x NVG, w_j(x) - w_j(y)}, and 
            \eqref{eq: gamma-limsup, VG x NVG, phi_j,R}.
        
            \begin{claim} We have
            \begin{equation}\label{claim:3rd term}
                \limsup_{j\to\infty}\Dsp(u_j, D_{\notVeryGoodJ}^2 \cap \diagonalDelta) \leq\errordRL.
            \end{equation}
        \end{claim}
    
            Since $u_j= u \phi_{j,\truncParam}$ on $D_{\notVeryGoodJ}^2 \cap \diagonalDelta$, we obtain,
            \begin{equation*}
                \begin{split}
                    \Dsp(u_j, D_{\notVeryGoodJ}^2 \cap \diagonalDelta) &\lesssim_p \norm{u}_{L^\infty(\domain)}^p \Dsp(\phi_{j,\truncParam}, D_{\notVeryGoodJ}^2 \cap \diagonalDelta) + \Dsp(u,D_{\notVeryGoodJ}^2 \cap \diagonalDelta) \\
                    &\leq \norm{u}_{L^\infty(U)}^p |\phi_{j,\truncParam}|^p_{\Wsp(\Rn)} + \Dsp(u,U^2 \cap \diagonalDelta).
                \end{split}
            \end{equation*}
            Then \eqref{claim:3rd term} follows by Lemma~\ref{lmm: NVG have vanishing relative capacity} (cf. also \eqref{e:phijR}).
    
        \begin{claim} We show that
            \begin{equation}\label{claim:4th term}
                \limsup_{j\to\infty}\Dsp(u_j, (\tilde{B}_{\veryGoodJ} \times \domain_j) \cap \diagonalDelta)  \leq\errordRL\,.
            \end{equation}
        \end{claim}
    
            For $(x,y) \in (\tilde{B}_{\veryGoodJ} \times \domain_j) \cap \diagonalDelta$, we have 
            \begin{equation*}
                \begin{split}
                    &u_j(x) - u_j(y) = \left(1-\xi_\joiningLemmaEnergyDiffParam\left(\frac{x-\eps_j x_i}{\lambda_j \rho_i}\right)\right) (u)_{\annulusJI} - w_j(y) \\
                    &=(u)_{\annulusJI} - w_j(x) - (u)_{\annulusJI} \xi_\joiningLemmaEnergyDiffParam\left(\frac{x-\eps_j x_i}{\lambda_j \rho_i}\right) + w_j(x) - w_j(y),
                \end{split}
            \end{equation*}
            and hence
            \begin{equation*}
                |u_j(x) - u_j(y)|^p \lesssim_p |w_j(x) - (u)_{\annulusJI}|^p + \norm{u}_{L^\infty(\domain)}^p \left\lvert \xi_\joiningLemmaEnergyDiffParam\left(\frac{x-\eps_j x_i}{\lambda_j \rho_i}\right) \right\rvert^p + |w_j(x)-w_j(y)|^p
            \end{equation*}
            We then observe that the terms above give contributions which were already estimated in \eqref{eq: gamma-limsup, VG x NVG, w_j - (u)}, \eqref{eq: gamma-limsup, VG x NVG, xi_L}, and \eqref{eq: gamma-limsup, VG x NVG, w_j(x) - w_j(y)}, respectively.
        \begin{claim} We prove that
            \begin{equation}\label{claim:5th term}
                \limsup_{j\to\infty}\Dsp(u_j, (D_{\notVeryGoodJ} \times \domain_j) \cap \diagonalDelta) \leq\errordRL\,.
            \end{equation}
        \end{claim}
            For $(x,y) \in ( D_{\notVeryGoodJ} \times \domain_j) \cap \diagonalDelta$,
            \begin{equation*}
                u_j(x) - u_j(y) = w_j(x) \phi_{j,\truncParam}(x) - w_j(y) = w_j(x) - w_j(y) - w_j(x) (1-\phi_{j,\truncParam}(x))
            \end{equation*}  
            and hence 
            \begin{equation*}
                |u_j(x) - u_j(y)|^p \lesssim_p |w_j(x) - w_j(y)|^p + \norm{u}_{L^\infty(\domain)}^p |1-\phi_{j,\truncParam}(x)|^p.
            \end{equation*}
            The contributions of the above terms were already analyzed in \eqref{eq: gamma-limsup, VG x NVG, w_j(x) - w_j(y)} and \eqref{eq: gamma-limsup, VG x NVG, phi_j,R}, respectively.

        Therefore gathering \eqref{claim:main contribution}-\eqref{claim:5th term} we deduce
        the estimate in \eqref{e:claim522 2} and therefore 
        \eqref{e:approx-upper-bound} as previously noticed.
\end{proof}

We conclude this section with the proof of the $\Gamma$-convergence result in the case of randomly shaped obstacles and anisotropic, translation-invariant, $-(n+sp)$-homogeneous measurable kernels, stated in Theorem~\ref{thm: main result random shapes}. 
Since the arguments closely follow those in the proof Theorem~\ref{thm: main result}, mainly resting on Lemma~\ref{l:loccap}, we restrict ourselves to highlighting the necessary modifications. 

\begin{proof}[Proof of Theorem~\ref{thm: main result random shapes}]
    We start by observing that, thanks to \ref{h: geometric assumption on obstacle of random shape}, in the proof we can resort to the same sets of indices introduced in Section \ref{subsec:set-of-indices}, namely $\goodEps$ and $\veryGoodEps$ (see \eqref{eq: definition of good indices}, \eqref{eq: definition of very good indices}).

    For the $\Gamma$-$\liminf$ inequality, we proceed as in the proof of Proposition~\ref{prop: gamma-liminf}, up to the estimate of the capacitary contribution. Using the same notation therein, we get 
    \begin{equation*}
        \begin{split}
             &\sum_{i \in \goodJOmega} \mathcal{K}(w_j^\omega,\ballContainingAnnulusJIomega)  \geq \sum_{i \in \goodJOmega} |(u_j)_{\annulusJIomega}|^p \cspK(S_j^{i,\omega}, \ballBoundingAnnulusJIomega, m_\truncParam^{-3h_j^\omega -2} \eps_j \truncParam^{-1}) \\
             &= \sum_{i \in \goodJOmega} \lambda_j^{n-sp}|(u_j)_{\annulusJIomega}|^p \cspK(\lambda_j^{-1} (S_j^{i,\omega} - \eps_j x_i(\omega)), B_{m_\truncParam^{-3h_j^\omega-1}\eps_j \lambda_j^{-1}\truncParam^{-1}}, m_\truncParam^{-3h_j^\omega -2} \eps_j \lambda_j^{-1}\truncParam^{-1}). 
        \end{split}
     \end{equation*}
     In view of assumption \ref{h: geometric assumption on obstacle of random shape},  $\lambda_j^{-1} (S_j^{i,\omega} - \eps_j x_i(\omega)) \subset B_{\rho_i(\omega)} \subset B_\truncParam$  for every $i \in \goodJOmega$.
     Therefore, arguing as in the proof of Proposition~\ref{prop: gamma-liminf}, we get, in view of \eqref{e:cap2} in Lemma~\ref{l:loccap}, 
     \begin{equation*}
        \begin{split}
            \sum_{i \in \goodJOmega} \mathcal{K}(w_j^\omega,\ballContainingAnnulusJIomega) &\geq \sum_{i \in \goodJOmega}|(u_j)_{\annulusJIomega}|^p \cspK(S_j^{i,\omega}) \\
            &- \frac{c_{n,s,p}}{(m_\truncParam-1)^{sp}} \CSPK(B_\truncParam,B_{m_\truncParam^{-3h_j^\omega -2} \eps_j \lambda_j^{-1}\truncParam^{-1}}) \eps_j^n \sum_{i \in \goodJOmega} |(u_j)_{\annulusJIomega}|^p. 
        \end{split}
     \end{equation*}
    Then, invoking formula \eqref{e:expectation gammai vs pi} and
    Remark~\ref{rmk:random_shapes}, we may conclude that
    \begin{equation*}
        \begin{split}
            &\lim_{\truncParam \to \infty}\lim_{\joiningLemmaEnergyDiffParam \to \infty}\liminf_{j \to \infty} \sum_{i \in \goodJOmega} \mathcal{K}(w_j^\omega,\ballContainingAnnulusJIomega) \geq  \lim_{\truncParam \to \infty}\lim_{\joiningLemmaEnergyDiffParam \to \infty}\liminf_{j \to \infty} \eps_j^n \sum_{i \in \goodJOmega}|(u_j)_{\annulusJIomega}|^p  \gamma_i(\omega) \\
            &= \norm{u}_{L^p(\domain)}^p \E\bigg[\sum_{x_i\in Q} \gamma_i \bigg] =  \tilde\gamma\norm{u}_{L^p(\domain)}^p.
        \end{split}
    \end{equation*}

    For the $\Gamma$-$\limsup$ inequality, we proceed as in the proof of Proposition~\ref{prop: gamma-limsup} with the following changes.
    For any $i \in \veryGoodJ$, we let $\xi_{\joiningLemmaEnergyDiffParam,i,j} \in W^{s,p}(\Rn; [0,1])$ be such that $\xi_{\joiningLemmaEnergyDiffParam,i,j} = 0$ on $\R^n \setminus B_\joiningLemmaEnergyDiffParam$,  $\tilde{\xi}_{\joiningLemmaEnergyDiffParam,i,j}= 1$ $\cspqe$ on $(\lambda_j \rho_i)^{-1} (S_j^i - \eps_j x_i) \subset B_1$, and 
    \begin{equation}
        \mathcal{K}(\xi_{\joiningLemmaEnergyDiffParam,i,j},\Rn) \leq \CSPK((\lambda_j \rho_i)^{-1} (S_j^i - \eps_j x_i), B_\joiningLemmaEnergyDiffParam) + \frac{1}{\joiningLemmaEnergyDiffParam}.
    \end{equation}
    Using the same notation as in the proof of Proposition~\ref{prop: gamma-limsup}, by \eqref{e:cap1} in Lemma~\ref{l:loccap} we get 
    \begin{equation}
        \label{eq: gamma-limsup capacity estimate random shape}
        \begin{split}
            \CSPK((\lambda_j \rho_i)^{-1} (S_j^i - \eps_j x_i), B_\joiningLemmaEnergyDiffParam)  &\leq \cspK((\lambda_j \rho_i)^{-1} (S_j^i - \eps_j x_i)) + \errordRL \\
            &= (\lambda_j \rho_i)^{sp-n} \cspK(S_j^i) + \errordRL = \rho_i^{sp-n} \gamma_i + \errordRL.
        \end{split}
    \end{equation}
    Recalling that $\ballBoundingAnnulusJI=B_{2^{-3h_j-1}\eps_jR^{-1}}(\eps_jx_i)$  (cf. \eqref{eq: definition of ball bounding annulus}), we define 
    \begin{equation*}
        v_j(x):= \begin{cases}
            w_j(x) \quad &\text{if } x \in \domain\setminus \bigcup\limits_{i \in \veryGoodJ} \ballBoundingAnnulusJI, \\
            \left(1-\xi_{\joiningLemmaEnergyDiffParam,i,j}\left(\frac{x-\eps_j x_i}{\lambda_j \rho_i}\right)\right) (u)_{\annulusJI} \quad &\text{if } x \in \bigcup\limits_{i \in \veryGoodJ} \ballBoundingAnnulusJI.
        \end{cases}
    \end{equation*}
    Using the functions $\phi_{j,\truncParam}$ introduced in the proof of Proposition~\ref{prop: gamma-limsup}, the definition of $u_j = v_j \phi_{j,\truncParam}$ remains unchanged, as well as its convergence properties. 
    
    To reconstruct the capacitary term, we argue as in \eqref{claim:main contribution}, with the following changes. By a change of variables and \eqref{eq: gamma-limsup capacity estimate random shape}, we have
    \begin{align}
        \limsup_{j\to\infty}&
        \sum_{i \in \veryGoodJ} \mathcal{K}(u_j,\ballBoundingAnnulusJI) 
        \leq \limsup_{j\to\infty}
        \eps_j^n \sum_{i \in \goodJ} \rho_i^{n-sp} |(u)_{\annulusJI}|^p \mathcal{K}(\xi_{\joiningLemmaEnergyDiffParam,i,j},\Rn)\notag \\
        &\leq \limsup_{j\to\infty}
        \eps_j^n \sum_{i \in \goodJ}  |(u)_{\annulusJI}|^p \gamma_i \notag\\
        &\leq \norm{u}_{L^p(\domain)}^p \E\bigg[\sum_{x_i\in Q} \gamma_i \bigg] +\errordRL =  \tilde\gamma\norm{u}_{L^p(\domain)}^p + \errordRL,
    \end{align}
    where the last inequality follows in view of formula
    \eqref{e:expectation gammai vs pi} and Remark~\ref{rmk:random_shapes}.
    
    Moreover, since $\xi_{\joiningLemmaEnergyDiffParam,i,j}$ takes values in $[0,1]$, $\norm{\xi_{\joiningLemmaEnergyDiffParam,i,j}}_{L^p(B_\joiningLemmaEnergyDiffParam)}^p \leq \L^n(B_\joiningLemmaEnergyDiffParam)$, and hence \eqref{eq: gamma-limsup norm p of xi_L} holds true.
Furthermore, in \eqref{e:claim524 2}, by the definition of $\xi_{\joiningLemmaEnergyDiffParam,i,j}$ and since $(\lambda_j \rho_i)^{-1} (S_j^i - \eps_j x_i) \subset B_1$, \eqref{e:Kernels}  and 
    \eqref{eq: locality defect on B(r) x B(r)^c} in Lemma~\ref{l:loccap} imply
    \begin{align}
        \limsup_{j \to \infty}& \sum_{i \in \veryGoodJ} \int_{(\ballBoundingAnnulusJI \times D_{\notVeryGoodJ}) \cap \diagonalDelta} |x-y|^{-(n+sp)} \left\lvert \xi_{\joiningLemmaEnergyDiffParam,i,j}\left(\frac{x-\eps_jx_i}{\lambda_j \rho_i}\right) \right\rvert^p \, \dd{x} \, \dd{y} \notag \\
        &\leq 
         \lim_{j \to \infty}  \eps_j^n \sum_{i \in \insideDomainJ} \Dsp(\xi_{\joiningLemmaEnergyDiffParam,i,j}, B_{\joiningLemmaEnergyDiffParam^2} \times (\Rn \setminus B_{\joiningLemmaEnergyDiffParam^2})) \rho_i^{n-sp} \notag\\
        &= \errordRL \, \IG \L^n(\domain) \int_{\R_+} \rho^{n-sp} \, \dd{\MD}(\rho) \leq\errordRL.
    \end{align}
    Hence, we conclude as in Proposition~\ref{prop: gamma-limsup}.
\end{proof}
\appendix
\section{}
\label{section: appendix}

For the readers' convenience, in this appendix we prove Proposition~\ref{prop: ergodic theorem mpp convergence of measures} and Proposition~\ref{prop: ergodic theorem, integral on eps^-1 U of h(rho)}.

\begin{proof}[Proof of Proposition~\ref{prop: ergodic theorem mpp convergence of measures}]
    The proof follows the ideas of \cite[B.1]{faggionato2020stochastichomogenizationamorphousmedia}. We recall that a sequence $\{A_j\}_{j \in \N}$ of bounded Borel sets in $\R^n$ is called a \textit{convex averaging sequence} \cite[Definition 12.2.I]{daley-vere-jones-2} if $A_j \subset A_{j+1}$, $A_j$ is convex for any $j \in \N$, and $r(A_j) \to \infty$ as $j\to\infty$, where $r(A)$ is the supremum of all radii $r\geq0$ such that $A$ contains a ball of radius $r$. By the Ergodic Theorem for stationary and ergodic m.p.p.\ \cite[Theorem~12.2.IV and Corollary~12.2.V]{daley-vere-jones-2}, for any convex averaging sequence $\{A_j\}_{j \in \N}$ in $\R^n$ and any $h \in L^1(\MD)$
    \begin{equation*}
        \lim_{j \to \infty} \frac{1}{\L^n(A_j)}\int_{A_j \times \R_+} h(\rho) \, \dd{\PP}(x,\rho) = \IG \int_{\R_+} h(\rho) \, \dd{\MD}(\rho) \quad \prob\text{-a.e.}
    \end{equation*}
        We now show that for any $h \in L^1(\MD)$ there exists $\sampleSpace_h \in \sigmaAlgebra$ with $\prob(\sampleSpace_h) = 1$ such that for any $\omega \in \sampleSpace_h$, if $A = \prod_{i=1}^n (a_i, b_i]$ with $a_i < b_i$ and $a_i, b_i \in \Q$ for $i=1,\ldots, n$, then 
        \begin{equation}
            \label{eq: proof of ergodic theorem convergence of measures, limit to prove}
            \lim_{t \to +\infty} \frac{1}{t^n} \int_{(tA)\times \R_+} h(\rho) \, \dd{\PP^\omega}(x,\rho) = \IG \L^n(A) \int_{\R_+} h(\rho) \, \dd{\MD}(\rho).
        \end{equation}
        To this end we define 
        \begin{equation*}
            F^\omega(t,A):= \frac{1}{t^n} \int_{(tA)\times \R_+} h(\rho) \, \dd{\PP^\omega}(x,\rho), \quad   G(A) :=\IG \L^n(A) \int_{\R_+} h(\rho) \, \dd{\MD}(\rho).
        \end{equation*}
        Without loss of generality, we assume that $h \geq 0$. If $A = \prod_{i=1}^n (0,b_i]$ with $b_i \geq 0$, then $\{j A \}_{j \in \N}$ is a convex averaging sequence and $\prob$-a.e. $\lim_{j \to \infty} F(j,A) = G(A)$.
        Moreover, for any $t \geq 0$, we have $\floor{t} \leq t < \floor{t}+1$ and 
        \begin{equation*}
            \frac{\floor{t}^n}{(\floor{t}+1)^n} F(\floor{t},A) \leq F(t,A) \leq \frac{(\floor{t}+1)^n}{\floor{t}^n} F(\floor{t}+1,A), 
        \end{equation*}
        which implies that $\prob$-a.e. $\lim_{t \to +\infty} F(t,A) = G(A)$. If \eqref{eq: proof of ergodic theorem convergence of measures, limit to prove} holds $\prob$-a.e. for $A = A_1$ and $A= A_2$ then it clearly holds $\prob$-a.e. for $A = A_1 \setminus A_2$. As a consequence, exploiting the previous case, by the argument contained in \cite[Proof of Lemma~B.1]{faggionato2020stochastichomogenizationamorphousmedia}, it is easy to see that \eqref{eq: proof of ergodic theorem convergence of measures, limit to prove} holds $\prob$-a.e. for $A = \prod_{i=1}^n (a_i,b_i]$ with $a_i < b_i$. So that we can conclude by observing that $\{\prod_{i=1}^n (a_i, b_i]: a_i < b_i, \, a_i,b_i \in \Q, \,  i = 1,\ldots, n\}$ is countable.
    
    Now let $\Omega_h$ be the event of probability one such that \eqref{eq: proof of ergodic theorem convergence of measures, limit to prove} holds, and let $\omega \in \Omega_h$. Moreover, let $f \in C_c^0(\Rn)$ and $\delta >0$. Since $f$ is uniformly continuous, there exists $\sigma >0$ such that $|f(x)-f(y)|<\delta$ whenever $|x-y|<\sigma$. Let $K:=\min\{k \in \N: \supp f \subset (-k,k]^n\}$. Then there are $A_1, \ldots, A_l$ pairwise disjoint sets of the kind described above such that $(-K,K]^n = \cup_{i=1}^l A_i$ and $\rm{diam}(A_i) < \sigma$ for any $i=1,\ldots, l$. Then, by construction, $\sum_{i=1}^l \inf_{A_i} f \, \IF_{A_i} \leq f \leq \sum_{i=1}^l \sup_{A_i} f \, \IF_{A_i}$ and hence
    \begin{equation*}
        \begin{split}
            &\sum_{i=1}^l \inf_{A_i} f \, \eps^n \int_{(\eps^{-1}A_i) \times \R_+} h(\rho) \, \dd{\PP^\omega}(x,\rho) \\
            &\leq \int_{\R^n\times\R_+} f(x) h(\rho) \, \dd{\PP_\eps^\omega}(x,\rho) \leq  \sum_{i=1}^l \sup_{A_i} f \, \eps^n \int_{(\eps^{-1}A_i) \times \R_+} h(\rho) \, \dd{\PP^\omega}(x,\rho).
        \end{split}
    \end{equation*}
    Then, by the choice of $\Omega_h$, we get
    \begin{equation*}
        \begin{split}
            &\IG \int_{\R_+} h(\rho) \, \dd{\MD}(\rho) \left(\int_{\R^n} f(x) \, \dd{x} - \delta (2K)^n \right) \leq \IG \int_{\R_+} h(\rho) \, \dd{\MD}(\rho) \sum_{i=1}^l \inf_{A_i} f \L^n(A_i) \\
            &\leq \liminf_{\eps \downarrow 0}\int_{\R^n\times\R_+} f(x) h(\rho) \, \dd{\PP_\eps^\omega}(x,\rho) \leq \limsup_{\eps \downarrow 0} \int_{\R^n\times\R_+} f(x) h(\rho) \, \dd{\PP_\eps^\omega}(x,\rho) \\
            &\leq \IG \int_{\R_+} h(\rho) \, \dd{\MD}(\rho) \sum_{i=1}^l \sup_{A_i} f \L^n(A_i) \leq \IG \int_{\R_+} h(\rho) \, \dd{\MD}(\rho) \left(\int_{\R^n} f(x) \, \dd{x} + \delta (2K)^n \right).
        \end{split}
    \end{equation*}
    Finally, we conclude by the arbitrariness of $\delta$.
\end{proof}
\begin{proof}[Proof of Proposition~\ref{prop: ergodic theorem, integral on eps^-1 U of h(rho)}] 
    Let $\Omega_h$ be the event of probability one given by Proposition~\ref{prop: ergodic theorem mpp convergence of measures} and let $\omega \in \Omega_h$. For $\delta >0$, let $\domain_\delta := \{x \in \Rn: \dist(x,\partial \domain) < \delta\}$ and $\domain_{-\delta}:= \{x \in \domain: \dist(x,\partial \domain) > \delta\}$. By the Urysohn Lemma~\cite[Proposition~8.18]{folland}, there exist $\varphi_\delta \in C^\infty_c(\Rn; [0,1])$ and $\varphi_{-\delta} \in C^\infty_c(\Rn; [0,1])$ such that $\varphi_\delta = 1$ on $\overline{\domain}$, $\supp \varphi_\delta \subset \domain_\delta$, $\varphi_{-\delta} = 1$ on $\overline{\domain_{-\delta}}$, $\supp \varphi_{-\delta} \subset \domain$. By construction, $\varphi_{-\delta} \leq \IF_\domain \leq \varphi_{\delta}$, thus we have 
    \begin{equation*}
        \begin{split}
            &\int_{\R^n \times \R_+} \varphi_{-\delta}(x) h(\rho) \, \dd{\PP_\eps^\omega}(x,\rho) \\
            &\leq \eps^n \int_{(\eps^{-1}U) \times \R_+} h(\rho) \, \dd{\PP^\omega}(x,\rho) \leq \int_{\R^n \times \R_+} \varphi_\delta(x) h(\rho) \, \dd{\PP_\eps^\omega}(x,\rho)
        \end{split}
    \end{equation*}
    while, by the choice of $\Omega_h$, there holds
    \begin{equation*}
        \begin{split}
            &\IG \int_{\R_+} h(\rho) \, \dd{\MD}(\rho) \L^n(U_{-\delta}) \leq \IG \int_{\R_+} h(\rho) \, \dd{\MD}(\rho) \int_{\R^n} \varphi_{-\delta}(x) \, \dd{x} \\
            &\leq \liminf_{\eps \downarrow 0} \eps^n \int_{(\eps^{-1}U) \times \R_+} h(\rho) \, \dd{\PP^\omega}(x,\rho) \leq \limsup_{\eps \downarrow 0} \eps^n \int_{(\eps^{-1}U) \times \R_+} h(\rho) \, \dd{\PP^\omega}(x,\rho) \\
            &\leq \IG \int_{\R_+} h(\rho) \, \dd{\MD}(\rho) \int_{\R^n} \varphi_{\delta}(x) \, \dd{x} \leq \IG \int_{\R_+} h(\rho) \, \dd{\MD}(\rho) \L^n(U_{\delta}).
        \end{split}
    \end{equation*}
    Since $\partial U$ is Lipschitz, we can conclude by taking the limit as $\delta \to 0$.
\end{proof}

\section*{Acknowledgements}  
F.~Deangelis and C.~I.~Zeppieri were funded by the Deutsche Forschungsgemeinschaft (DFG, German Research Foundation) under Germany's Excellence Strategy EXC 2044/2 - 390685587, Mathematics Münster: Dynamics--Geometry--Structure.

\bibliography{refs}
\bibliographystyle{alpha}
\end{document}